\documentclass[11pt,a4paper]{article}
\usepackage{etex}
\usepackage{inputenc}
\usepackage[english]{babel}
\usepackage{lmodern}					
\usepackage{textcomp}				
\usepackage{verbatim}			
\usepackage[toc,page]{appendix} 		
\usepackage{cite}
\usepackage{amsmath,					
			amsthm,
			amsfonts,
			amssymb, 
			mathtools}
\usepackage{stmaryrd}
\usepackage{physics}
\usepackage{enumitem}
\usepackage{hyperref}
\hypersetup{
colorlinks=true,
linkcolor=red,
filecolor=magenta,
urlcolor=cyan
}

 \usepackage{tikz-cd} 
\usetikzlibrary{tikzmark}
\usetikzlibrary{arrows}
 \usepackage{verbatim}
 \newcommand{\bigslant}[2]{{\left.\raisebox{.2em}{$#1$}\middle/\raisebox{-.2em}{$#2$}\right.}}
\definecolor{DarkGreen}{HTML}{1cad22}
\usepackage[hmargin=3cm,vmargin=1.5cm]{geometry}
\newtheorem{definition}{Definition}[section]
\newtheorem{theorem}[definition]{Theorem}
\newtheorem{proposition}[definition]{Proposition}
\newtheorem{lemma}[definition]{Lemma}

\newtheorem{corollary}[definition]{Corollary}
\newtheorem{remark}[definition]{Remark}
\newtheorem{example}[definition]{Example}
\theoremstyle{remark}

\title{Atiyah classes of Lie algebroid homotopy modules}
\author{Panagiotis Batakidis$^\dagger$ \\\href{mailto:batakidis@math.auth.gr}{batakidis@math.auth.gr}  \and Sylvain Lavau$^\ddag$
\\\href{mailto:lavau@math.univ-lyon1.fr}{lavau@math.univ-lyon1.fr}}

\date{$^\dagger$ Department of Mathematics, Aristotle University of Thessaloniki\\%
    $^\ddag$ Division of Theoretical Physics, Ru\dj er Bo\u{s}kovi\'c Institute\\[2ex]%
}

\begin{document}

\maketitle

\begin{abstract}
\noindent
For a Lie algebroid pair $A\hookrightarrow L$ we study cocycles constructed from the extension to $L$ of the higher connection forms of a representation up to homotopy $E$ of the Lie algebroid $A$. We show that there exists a cohomology class with values in the endomorphism bundle of $E$ that is independent of the extension above and vanishes whenever a homotopy $A$-compatible extension exists. Whenever the representation up to homotopy $E$ is the resolution of a Lie algebroid representation $K$ of $A$, it is shown that there exists a quasi-isomorphism sending the new Atiyah class to the classical one, associated to extensions to $L$ of the Lie algebroid representation $K$.
\end{abstract}

\vspace{5mm} \noindent {\emph{Keywords: }} Lie pairs; representations up to homotopy; Atiyah class; graded geometry.

\vspace{3mm} \noindent MSC (2020): 53C05, 58J10, 58A15, 16S32, 17B35, 17B55, 70G45.

\tableofcontents

\section{Introduction}

This paper is the first step towards the definition and study of an analogue of the Atiyah-Molino class for singular foliations. Although the regular foliation case $(M,\mathcal{F})$ is by now standard, and can be treated with the use of Lie (algebroid) pairs $A\hookrightarrow L$, taking $L=TM$ and $A=T\mathcal{F}$ (see \cite{chenAtiyahClassesHomotopy2016}), the fact that in the singular case the dimension of the leaves is non-constant prevents one from using the traditional definitions fit for vector bundles. 
As is standard in that situation, one may replace the singular foliation by (one of) its universal Lie $\infty$-algebroids\footnote{A universal Lie $\infty$-algebroid of a singular foliation $\mathcal{F}$ is a Lie $\infty$-algebroid covering the singular foliation~$\mathcal{F}$ and such that the linear part is a resolution of $\mathcal{F}$.} \cite{laurent-gengouxUniversalLieInfinityalgebroid2020}.  
Following the classical strategy, one would then want to define pairs of Lie $\infty$-algebroids and then  define the Atiyah-Molino class of a singular foliation as the Atiyah class of an appropriate Lie $\infty$-algebroid pair. The same remark applies to singular subalgebroids \cite{zambonSingularSubalgebroids2022}.

Towards this goal, we consider Lie algebroid pairs $A\hookrightarrow L$ and provide here a definition of the Atiyah class of the extension of a  representation up to homotopy (r.u.t.h.) $E$ of~$A$ to an $L$-superconnection on $E$. We show that the construction is consistent with the traditional notion of Atiyah class of Lie algebroid pairs \cite{chenAtiyahClassesHomotopy2016} when the representation up to homotopy is (a resolution of) 
a classical Lie algebroid representation. Namely, in that particular case, the former does not contain any further information that is not already contained in the later.

From the present intermediate step, we plan, in a future paper, to construct the Atiyah class of Lie $\infty$-algebroid pairs (whose representations are by definition up to homotopy). This construction should naturally reduce to that given for $L_\infty$-algebra pairs \cite{chenAtiyahClassesStrongly2019} when the Lie $\infty$-algebroid pairs are over a point. Schematically we have the following steps towards the Atiyah class of singular foliations:
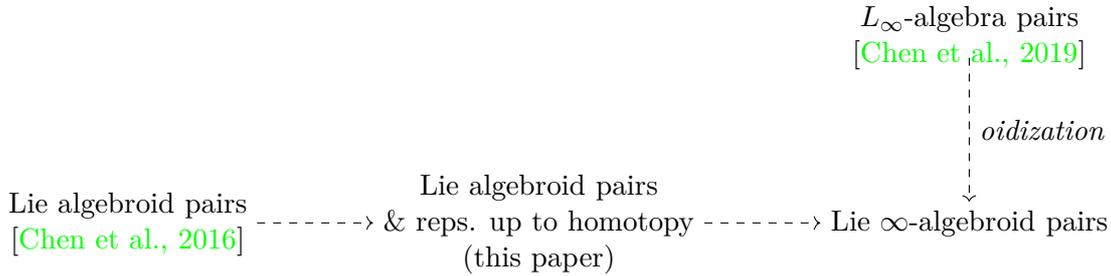
\begin{center}
    \begin{tikzpicture}
\matrix(a)[matrix of math nodes, 
row sep=5em, column sep=4em, 
text height=1.5ex, text depth=0.25ex] 
{&&\begin{tabular}[c]{@{}c@{}}$L_\infty$-algebra pairs\\ \cite{chenAtiyahClassesStrongly2019} \end{tabular}\\
\begin{tabular}[c]{@{}c@{}}Lie algebroid pairs\\ \cite{chenAtiyahClassesHomotopy2016}\end{tabular}& \begin{tabular}[c]{@{}c@{}}Lie algebroid pairs\\ \& reps. up to homotopy\\ (this paper)\end{tabular}&\text{Lie $\infty$-algebroid pairs}\\};
\path[->, dashed](a-2-1) edge (a-2-2);
\path[->, dashed](a-2-2) edge (a-2-3);
\path[->, dashed](a-1-3) edge[right,"\emph{oidization}"] (a-2-3);
\end{tikzpicture}
\end{center}

In Section \ref{sec0} we first recall the notions of superconnections and representations up to homotopy $(E,D_A)$ of a Lie algebroid $A$, that we also call \emph{homotopy $A$-modules}. These are precisely the graded vector bundles $E$ over $M$ equipped with a flat $A$-superconnection $D_A$ on $E$. In particular, this differential operator involves a degree $+1$ vector bundle morphism $\partial:E_\bullet\to E_{\bullet+1}$ squaring to zero and controlling the homotopy theory.
We then discuss the extension to $L$ of the higher connection forms of a not necessarily flat $A$-superconnection  on $E$. For flat $A$-superconnections we define homotopy $A$-compatible superconnections of $L$ on $E$ (Definition \ref{def:compatible}) and construct the Atiyah class of representations up to homotopy for Lie algebroid pairs. The first main result of this section is Theorem \ref{main1} which is a generalization to the context of representations up to homotopy of \cite[Theorem 2.5]{chenAtiyahClassesHomotopy2016}. It exhibits the existence of a distinguished cohomology class associated to any  representation up to homotopy $E$ with respect to the Lie algebroid pair $A\hookrightarrow L$, that we call the \emph{Atiyah class} of $(E,D_A)$.
The notion is well defined with respect to the set of isomorphism
classes of representations up to homotopy of $A$ \cite{abadRepresentationsHomotopyLie2011}, as two isomorphic r.u.t.h. of $A$ have simultaneously vanishing Atiyah classes (Proposition~\ref{propisom}).

The section ends with a look at the case where the homotopy $A$-module $E$ is a regular resolution of some vector bundle $K$. We call a representation up to homotopy \emph{regular} if its chain map $\partial$ is of constant rank. Being a resolution of a vector bundle $K$ means, in the present context, that $K= H^0(E,\partial)$. 
The second main result of this section is Theorem~\ref{prop:BRST} stating that there is an isomorphism between the cohomology spaces where the Atiyah classes of $E$ and $K$ live, and in fact the first is sent to the second. The main results of this section are gathered in the following:

\begin{theorem}\label{Theorem1}
    Let $A\hookrightarrow L$ be a Lie pair, $(E,D_A)$ be a representation up to homotopy of $A$, and $D_L$ be a superconnection of $L$ on $E$ extending $D_A$.  Then, $D_L$ induces a total degree~$+1$ differential $s$  on the bigraded space  
    \begin{equation*}
    \widehat{\Omega}(E)=\bigoplus_{j=-\infty}^\infty\bigoplus_{k=0}^{\mathrm{rk}(A)}\Omega^k(A,A^\circ\otimes\mathrm{End}_{j}(E)),
\end{equation*}
making the graded vector space $A^\circ\otimes\mathrm{End}(E)$ a representation up to homotopy of $A$.
 There exists a 1-cocycle $\alpha=\sum_k\alpha^{(k)}$ of the cochain complex $\big(\widehat{\Omega}(E),s\big)$ such that:
\begin{enumerate}
    \item Its cohomology class $[\alpha]\in H^1\big(\widehat{\Omega}(E),s\big)$ does not depend on the choice of extension~$D_L$ of the $A$-superconnection $D_A$; 
    \item The cohomology class $[\alpha]$ vanishes if an only if there exists a homotopy $A$-compatible $L$-superconnection on $E$;

\item Suppose $E$ is a non-positively graded regular homotopy $A$-module, and that $\mathcal{H}^\bullet(E,\partial)=\mathcal{H}^0(E,\partial)=:K$ (in which case $E$ is a resolution of $K$). Then for every $p\in\mathbb{Z}$, there is an isomorphism
\begin{equation*}
    H^p\big(\widehat{\Omega}(E),s\big)\simeq H^p\big(A,A^\circ\otimes \mathrm{End}(K)\big),
\end{equation*}

\noindent  under which the class $[\alpha]\in H^1\big(\widehat{\Omega}(E),s\big)$ above is mapped to the Atiyah class $[\mathrm{at}_K]\in H^1\big(A,A^\circ\otimes \mathrm{End}(K)\big)$ of the Lie algebroid representation $K$.
\end{enumerate}
\end{theorem}

In the third section we work out basic examples of representations up to homotopy. First we discuss the simple example of the representation up to homotopy known as \textit{double} vector bundle. The second example, the \textit{normal complex}, essentially recovers the Atiyah class of the Bott representation as the Atiyah class of the corresponding r.u.t.h.. The third example is about the r.u.t.h. of a regular Lie algebroid $A$ known as the \textit{adjoint complex}, i.e.  $E=\mathrm{Ker}(\rho|_A)[2]\oplus A[1]\oplus TM$. For both the second and third examples, we show that the corresponding classical Lie algebroid Atiyah class vanishes if and only if the Atiyah class of the r.u.t.h. $E$ vanishes. The main results of the section, applications of the isomorphism of Theorem \ref{prop:BRST}, appear in Propositions \ref{propo} and \ref{atnu} and can be summarized as follows:

\begin{theorem}\label{Theorem2} Let $A\hookrightarrow L$ be a Lie pair. 
\begin{enumerate}
\item Consider the two-term homotopy  $A$-module $(E,\partial)$,  with nonzero components $E_{-1}=A[1], E_0=L$ and where $\partial=\iota\circ sp:A[1]\to L$ is the suspension of the inclusion map $\iota:A\hookrightarrow L$. Let $[\alpha]$ be the Atiyah class of this r.u.t.h. as in Theorem \ref{Theorem1} and $[\mathrm{at}^{Bott}]$ be the Atiyah class of the Bott representation of $A$ on $L/A$. Then 
\[[\alpha]=0\;\; \text{if and only if}\;\; [\mathrm{at}^{Bott}]=0.\]
\item Let $A$ be a  regular Lie algebroid. Consider the three-term homotopy $A$-module $(E,\partial)$ with nonzero components  $E=\mathrm{Ker}(\rho|_A)[2]\oplus A[1]\oplus TM$ and where $\partial\big|_{E_{-2}}=\iota\circ sp:\mathrm{Ker}(\rho|_A)[2]\to A[1]$ and $\partial\big|_{E_{-1}}=\rho\circ sp:A[1]\to TM$.  Let $[\alpha]$ be the Atiyah class of this r.u.t.h. as in Theorem \ref{Theorem1} and $[\mathrm{at}^{\nu}]$ be the Atiyah class of the 
 Lie representation of $A$ on $\nu(A)=TM/\rho(A)$ defined by $\nabla^\nu_a(p(X))=p([\rho(a),X])$. Then 
\[[\alpha]=0\;\; \text{if and only if}\;\; [\mathrm{at}^{\nu}]=0.\] 
\end{enumerate}
\end{theorem}

In the final section, we deal with the graded geometric aspects. The notion of Atiyah class has been extended and studied for differential graded (dg) manifolds  \cite{mehtaAtiyahClassDgvector2015} and dg Lie algebroids \cite{batakidisAtiyahClassesDgLie2018, stienonFedosovDgManifolds2020}. When the representation of the Lie algebroid $A\subset L$ is the normal bundle $L/A$, the  Atiyah class of this representation can be seen as the Atiyah class of the differential graded Lie algebroid $L[1]\oplus L/A\oplus L/A\to L[1]\oplus L/A$.
In \cite{stienonAinfinityAlgebrasLie2022, liaoAtiyahClassesTodd2023}, it was proved that given a Lie pair $A\hookrightarrow L$, one may define a dg Lie algebroid $\mathcal{L}=\pi^{!}L$ -- where $\pi:A[1]\to M$ denotes the projection map -- 
such that  the Atiyah class of a connection $\nabla:\;\Gamma_{A[1]}(\mathcal{L})\times\Gamma_{A[1]}(\mathcal{L})\to\Gamma_{A[1]}(\mathcal{L})$ in the sense of \cite{mehtaAtiyahClassDgvector2015} is equivalent to the Atiyah class of the Bott representation of the given Lie pair, so that in the regular foliation case it recovers the Atiyah-Molino class.
We extend this result to arbitrary representations up to homotopy of $A$. The main result of this section, Theorem \ref{theoremfinal}, shows that there is a quasi-isomorphism that maps the Atiyah class of an extension of a r.u.t.h. $E$ of $A$ to a $L$-superconnection on $E$ to the Atiyah class of $\nabla:\Gamma_{A[1]}(\mathcal{L})\times\Gamma_{A[1]}(\mathcal{E})\to\Gamma_{A[1]}(\mathcal{E})$ whenever $\mathcal{E}=\pi^*E\to A[1]$. This generalization of the main result of \cite{liaoAtiyahClassesTodd2023} to Atiyah classes of arbitrary r.u.t.h. of $A$ is proved using similar homological arguments for the generalized objects that we introduce in the paper (see Theorem \ref{theoremLiao}). In particular, and building on Theorem \ref{Theorem1}, we have:

\begin{theorem}\label{Theorem3}
Let $A\hookrightarrow L$ be a Lie pair, $(E,D_A)$ be a non-positively graded regular representation up to homotopy of $A$, 
whose cohomology is concentrated in degree $0$, in which case it is a resolution of $K:=\mathcal{H}^\bullet(E,\partial)=\mathcal{H}^0(E,\partial)$. 
Let $\mathcal{E}=\pi^*E$ (resp. $\mathcal{K}=\pi^\ast K$) be the pullback of $E$ (resp. $K$) over the projection $\pi: A[1]\to M$. 
\begin{enumerate}
\item  The operator $d=D_A-\partial$ is a small perturbation of the contraction  $(\Gamma_{A[1]}(\mathcal{E}), \partial, \theta)$ over $\mathbb{R}$. The perturbed contraction
\[\begin{tikzcd}
      \left(\Gamma_{A[1]}(\mathcal{K}),d_A^{\nabla^K}\right) \arrow[r, shift left=0.75ex, "\varsigma", hook] & \arrow[l, shift left=0.75ex, "\varphi", ->>] \left(\Gamma_{A[1]}(\mathcal{E}),D_A\right) \ar[loop,out=10,in=-10,distance=20, "\vartheta"]
    \end{tikzcd}\]
    forms a contraction data over $(\mathcal{C}^\infty(A[1]),d_A)$.

\item There is a quasi-isomorphism
\begin{equation*}
 \left(\widehat{\Omega}(E)^{\bullet,\bullet},s\right)\simeq \left(\Gamma_{A[1]}\big(\mathcal{L}^*\otimes\mathrm{End}(\mathcal{E})\big),\mathcal{Q}+[D_A,.\,]\right).
\end{equation*}
where $\mathcal{Q}$ is the homological vector field associated to the dg Lie algebroid $\mathcal{L}\to A[1]$,
\end{enumerate}
\end{theorem}

\textit{Notation:} Throughout this paper, the base manifold (body) is a smooth manifold $M$. A Lie pair will always be denoted as $A\hookrightarrow L$, so that $A$ is a Lie subalgebroid of $L$. Graded vector bundles, and more specifically representations up to homotopy of $A$, are denoted by $E$. 
A \emph{(split) graded vector bundle} over  $M$ is a family $E=(E_i)_{i\in\mathbb{Z}}$ of vector bundles over $M$, labelled over $\mathbb{Z}$. We will assume that the label $i$ -- that we call the \emph{degree of $E_i$} -- is bounded above and below. Following \cite{kotovCategoryGradedManifolds2024}, we say that such a graded vector bundle $E$ is of \emph{finite dimension}, i.e. that all but a finite number of $E_i$ are zero, and those which are not, are finite dimensional. The main consequence of this assumption is that tensor product stabilizes the class of split finite dimensional graded vector bundles over $M$, and that the graded vector bundle $\mathrm{End}(E)$ of endomorphisms of $E$ is of finite dimension as well.  In this paper we always consider (split) graded vector bundles of finite dimension.

\textit{Acknowledgments:} Thanks to Dimitris Makris, Ping Xu, Hsuan-Yi Liao, Camille Laurent-Gengoux and Thomas Basile for fruitful discussions. Thanks to Athanasios Chatzistavrakidis and to Anastasios Fotiadis for remarks on the preliminary version. The second author wants to deeply thank the Department of Mathematics of the Aristotle University of Thessaloniki at which most of this work has been done, for its warmth welcome. This work has been supported by the Croatian Science Foundation project IP-2019-04-4168
“Symmetries for Quantum Gravity”.

\section{Atiyah classes of representations up to homotopy of Lie algebroid pairs}\label{sec0}

\subsection{Superconnections and Lie algebroid homotopy modules}\label{sec1}
Let $M$ be a smooth manifold, $(A,\rho,[\cdot,\cdot])$ be a Lie algebroid over $M$. 
The sheaf $\Omega(A)$ of forms on $A$ is equipped with the Lie algebroid differential defined by the formula
\begin{align*}
    d_A(u)(a_0,\ldots,a_p)&=\sum_{i=0}^p\,(-1)^{i}\rho(a_i)\big(u(a_0,\ldots,\widehat{a_i},\ldots,a_p)\big)\\ 
    &\hspace{1cm}+\sum_{0\leq i,j\leq p}\,(-1)^{i+j}u\big([a_i,a_j],a_0,\ldots,\widehat{a_i},\ldots,\widehat{a_j},\ldots,a_p\big),
\end{align*}
for $u\in\Omega^p(A)$ and $a_0,\ldots,a_p\in\Gamma(A)$.  In other words, the Lie algebroid differential is the pullback through $\rho^*$ of the de Rham derivative, $d_A=\rho^*d_{\mathrm{dR}}$.

 Let $E$ be a (split, finite dimensional) graded vector bundle over $M$, that is to say: a family of vector bundles $(E_i)_{i\in\mathbb{Z}}$ indexed over $\mathbb{Z}$, such that all but a finite number of $E_i$ are zero, and those which are not are finite dimensional. We say that $E$ is positively (resp. negatively) graded if the index $i$ is concentrated in positive (resp. negative) integers. Let us define $(m,M)$ to be the unique pair of integers satisfying the two conditions:\begin{itemize}
     \item $E_m\neq0$, and $E_{i}=0$ for any $i\leq m-1$;
     \item $E_M\neq0$, and $E_{i}=0$ for any $i\geq M+1$.
 \end{itemize}
 In particular, $m\leq M$, and the positive integer $M_E=M-m+1$ is called the \emph{amplitude of $E$}.
Let $\mathrm{End}_i(E)$ be the vector bundle of endomorphisms of $E$ of degree $i$ -- i.e. those (degree 0) vector bundle morphisms $E_\bullet\rightarrow E_{\bullet+i}$. 
 Denote by $\mathrm{End}_{\bullet}(E)=\bigoplus_{j\in\mathbb{Z}}\mathrm{End}_j(E)$ the graded vector bundle of graded endomorphisms of $E$. It is well defined as we have assumed $E$ to be of finite dimension. The direct sum is actually bounded above by $M_E-1$ and below by~$-M_E+1$.
 
The sheaf $\Omega(A,E)$ of differential forms on $A$ valued in the graded vector bundle $E$ is bigraded: one grading is by form degree (or \textit{arity}) on $A$ and the other is by the total degree of the form, i.e. a $k$-form $\omega$ on $A$ valued in $E_{i}$ is said to be of total degree $t=k+i$, and we write $\omega\in\Omega^k(A,E)_t= \Omega^k(A,E_{t-k})$. The total degree is bounded above and below since  $A$ is of finite rank and $E$ is of finite dimension.  We denote by $\Omega^\bullet(A,E)_\blacktriangle$ the following direct sum of sheaves:
\begin{equation*}
    \Omega^\bullet(A,E)_\blacktriangle=\bigoplus_{t=-\infty}^\infty\bigoplus_{k=0}^{\mathrm{rk}(A)}\Omega^k(A,E)_t.
\end{equation*}
The sum is finite as we assumed that the graded vector bundle is of finite dimension. Indeed, the lower bound on $t$ is $m$, while the upper bound is $\mathrm{rk}(A)+M$.
When we want to emphasize the role of the total degree $t$, we omit the $\bullet$ exponent and only write $\Omega(A,E)_\blacktriangle$. 
The Lie algebroid differential $d_A$ can be straightfowardly extended to $\Omega(A,E)$ or $\Omega(A,\mathrm{End}(E))$, by acting on the $A$-form part only. 

\begin{remark}\label{remarksevere}
Let  $\eta\in\Omega(A,\mathrm{End}(E))$ be of total degree $t_\eta$. By definition of the amplitude of the (split, finite dimensional) graded vector bundle $E$, the total degree of $\eta$ is bounded above and below: $-M_E+1\leq t_\eta\leq \mathrm{rk}(A)+M_E-1$.
Then $\eta$
splits into a finite sum \begin{equation}\label{eq:decompositionD0}
\eta=\sum_{k=0}^{\mathrm{rk}(A)}\eta^{(k)},\end{equation}
where each component $
\eta^{(k)}\colon\Omega^\bullet(A,E)_{\blacktriangle}\to \Omega^{\bullet+k}(A,E)_{\blacktriangle+t_\eta}$  is a $k$-form on $A$ valued in $\mathrm{End(E)}_{t_\eta-k}$. 
We see that whenever $t_\eta=-M_E+1$ only the component $\eta^{(0)}$ contributes to the sum, while whenever $t_\eta=\mathrm{rk}(A)+M_E-1$ only the component $\eta^{(\mathrm{rk}(A))}$ does.

Moreover, if the amplitude of $E$ is higher than $1$,  then whenever $t_\eta=0$ (resp. $t_\eta=1$), the sum \eqref{eq:decompositionD0} is bounded above by $\mathrm{min}(\mathrm{rk}(A),M_E-1)$ (resp. $\mathrm{min}(\mathrm{rk}(A),M_E)$).
In both cases, we can however adopt the convention that the sum ranges from 0 to $\mathrm{rk}(A)$, keeping in mind that  that some of its components identically vanish if $\mathrm{rk(A)}\geq M_E$ (resp. $\mathrm{rk(A)}\geq M_E+1$) -- see the examples in Section \ref{sectionExamples} where this phenomenon occurs.
\end{remark}

Let $E=E_0$ be a vector bundle over $M$. A (linear) $A$-connection on $E$ is an $\mathbb{R}$-bilinear map $\nabla:\;\Gamma(A)\times\Gamma(E)\to\Gamma(E)$ such that $\nabla_{fa}e=f\nabla_ae$ and $\nabla_afe=f\nabla_ae+\rho(a)(f)e$, for every smooth function $f\in\mathcal{C}^\infty(M)$. 
The $A$-connection $\nabla$ on $E$  can be equivalently described by an exterior covariant derivative $d^\nabla_A:\Omega^\bullet(A,E)\to\Omega^{\bullet+1}(A,E)$ such that
\begin{equation}\label{eq:covariantderive}
d^\nabla_A(\eta\otimes e)=d_A(\eta)\otimes e+(-1)^{m}\eta\otimes \nabla(e),
\end{equation}
for every $\eta\in\Omega^m(A)$ and $e\in\Gamma(E)$. It extends to an $\Omega(A)$-derivation on $\Omega(A,E)$ by the graded Leibniz rule
\begin{equation*}
d^\nabla_A(\eta\wedge u)=d_A(\eta)\wedge u+(-1)^{m}\eta\wedge d^\nabla_A(u),
\end{equation*}
for every $\eta\in\Omega(A),\;u\in\Omega(A,E)$, and in local sections is defined by the formula 
\begin{align*}
    d^\nabla_A(u)(a_0,\ldots,a_p)&=\sum_{i=0}^p\,(-1)^{i}\nabla_{a_i}\big( u(a_0,\ldots,\widehat{a_i},\ldots,a_p)\big)\\
    &\hspace{1cm}+\sum_{0\leq i,j\leq p}\,(-1)^{i+j} u\big([a_i,a_j],a_0,\ldots,\widehat{a_i},\ldots,\widehat{a_j},\ldots,a_p\big).
\end{align*}

The curvature associated to the connection $\nabla$ is a $\mathrm{End}(E)$-valued 2-form on $A$, defined by the usual formula: $R_\nabla(a,b)=[\nabla_a,\nabla_b]-\nabla_{[a,b]}$. The connection $\nabla$ is \emph{flat} when the associated curvature $R_{\nabla}$ vanishes, or, equivalently, when $(d^{\nabla}_A)^2=0$. In this case, $d^\nabla_A$ defines a chain complex structure on $\Omega^\bullet(A,E)$, whose associated (Chevalley-Eilenberg) cohomology is denoted by $H^\bullet(A,E)$, and $E$ is said to be a \emph{Lie algebroid representation of~$A$}, or an \emph{$A$-module}.

Flat Lie algebroid connections on vector bundles $E=E_0$ characterize representations of Lie algebroids. Passing to the graded context, flatness of a Lie algebroid connection makes sense only \emph{up to homotopy}, justifying the introduction of the following notion. 

\begin{definition} \label{defquillen}\cite{quillenSuperconnectionsChernCharacter1985}
    Let $A$ be a Lie algebroid over $M$ and $E=(E_i)_{i\in\mathbb{Z}}$ be a (split, finite dimensional) graded vector bundle over $M$. We call    \emph{$A$-superconnection}  any differential operator  $D_A:\Omega(A,E)_\blacktriangle\to\Omega(A,E)_{\blacktriangle+1}$ satisfying  the Leibniz rule
\begin{equation*}
D_A(\eta\otimes u)=d_A(\eta)\otimes u+(-1)^{m}\eta\otimes D_A(u),
\end{equation*}
for every $\eta\in\Omega(A),\;u\in\Omega(A,E)$.
\end{definition}

Any $A$-superconnection $D_A$ splits into a finite sum 
 \begin{equation}\label{eq:decompositionD}
D_A=\sum_{k=0}^{\mathrm{rk}(A)}D_A^{(k)},\end{equation}
where we used the convention on the sum introduced in Remark \ref{remarksevere}.
 For $k=1$, the operator $D_A^{(1)}$ is the exterior covariant derivative~$d_A^\nabla$ associated to some $A$-connection~$\nabla$. 
   Given a choice of open covering $(U_\alpha)_\alpha$ of $M$ trivializing both $A$ and $E$ (this is possible since $E$ is split of finite dimension), there exists a family~$(\omega^{(1)}_\alpha)_{\alpha}$ of locally defined 1-forms $\omega^{(1)}_\alpha\in\Omega^1\big(A\big|_{U_\alpha},\mathrm{End}_0(E)\big|_{U_\alpha}\big)$  
such that on the overlap $U_\alpha\cap U_\beta$ they satisfiy the conditions
\begin{equation}\label{eq:transitionmap}
\omega^{(1)}_\beta=g_{\alpha\beta}^{-1}d_Ag_{\alpha\beta}+g_{\alpha\beta}^{-1}\omega^{(1)}_\alpha g_{\alpha\beta}.
\end{equation}
Here $g_{\alpha\beta}$ represents the transition map of $E$ over $U_\alpha\cap U_\beta$. 
Such a family  of locally defined 1-forms $(\omega^{(1)}_\alpha)_{\alpha}$ satisfying Equation \eqref{eq:transitionmap} on overlaps is abusively called a \emph{connection 1-form} and is denoted by $\omega^{(1)}$ in the rest of the text. 

 For every $k\neq1$, the restriction of the action of $D_A^{(k)}$ on $\Omega^0(A,E)=\Gamma(E_\bullet)$ canonically induces an $\mathrm{End}_{1-k}(E)$-valued $k$-form $\omega_A^{(k)}$ on $A$:
  \begin{equation*}
D_A^{(k)}\colon\Gamma(E_\bullet)\longrightarrow \Omega^k(A,E_{\bullet+1-k})\qquad\longleftrightarrow \qquad \omega_A^{(k)}\in \Omega^k\big(A,\mathrm{End}_{1-k}(E)\big),
  \end{equation*}
  by the relation $D_A^{(k)}(e)(a_1\wedge\cdots\wedge a_k)=\omega_A^{(k)}(a_1\wedge\cdots\wedge a_k)(e)$.
The correspondence between~$D_A^{(k)}$ and $\omega_A^{(k)}$ is described in the following equation:
 \begin{equation}
D_A^{(k)}(u)=\omega_A^{(k)}\wedge u, \label{eqokkk} 
\end{equation}
for every $u\in\Omega(A,E)$. Here, the wedge product means that we apply a concatenation on the $A$-form part of $u$, and that $\omega^{(k)}_A$ acts as an endomorphism on the part of $u$ sitting in~$E$. Namely, for $\eta\otimes \alpha\in\Omega^p(A,\mathrm{End}_i(E))$ and $u=\omega\otimes e\in\Omega^q(A,E)$, one has:
\begin{equation*}
(\eta\otimes \alpha)\wedge(\omega\otimes e)=(-1)^{iq}(\eta\wedge\omega)\otimes \alpha(e).
\end{equation*}
In particular, the 0-form $\omega_A^{(0)}\in\Gamma\big(\mathrm{End}_{1}(E)\big)$ corresponds to a degree $+1$ vector bundle endomorphism of~$E$ that we denote $\partial:E_{\bullet}\to E_{\bullet+1}$.
For $2\leq k\leq M(A,E)$, 
we call \emph{connection $k$-form} the $\mathrm{End}(E)$-valued $k$-form $\omega_A^{(k)}$. 

This discussion establishes that an $A$-superconnection $D_A$ can be symbolically written
 \begin{equation}\label{eq:defD}
     D_A=\partial+d^\nabla_A + \sum_{k\geq 2}^{\mathrm{rk}(A)}\omega_A^{(k)}\wedge\cdot.
 \end{equation}
where we used the convention on the sum introduced in Remark \ref{remarksevere}.
 In analogy with Lie algebroid representations, whenever $D_A$ squares to zero, the graded vector bundle $E$ is said to be a \emph{representation up to homotopy of~$A$}.

\begin{definition}\cite{abadRepresentationsHomotopyLie2011}\label{repuptohom} Let $A$ be a Lie algebroid over a smooth manifold~$M$.
A \emph{representation up to homotopy of $A$} -- or \emph{homotopy $A$-module} -- is a (split, finite dimensional) graded vector bundle $E$ over $M$ with a vector bundle morphism $\partial:E_\bullet\to E_{\bullet+1}$ forming a chain complex $(E_\bullet,\partial)$, together with either one of the following equivalent sets of data:
\begin{enumerate}
    \item 
an $A$-superconnection $D_A$ on $E$ such that $D_A^{(0)}=\partial$ and $(D_A)^2=0$;
\item an $A$-connection $\nabla$ on $E$, together with a family of forms $\big(\omega_A^{(k)}\big)_{k\geq2}\in\Omega^k(A,\mathrm{End}(E))_1$ of total degree $+1$ such that \begin{align}
[d_A^\nabla,\partial]&=0,\label{eqnablapartial}\\
R_A^{(k)}+\big[\partial,\omega_A^{(k)}\big]&=0\hspace{1cm}\text{for every $k\geq2$},\label{eqimpot}
\end{align}
where $R_A^{(k)}$ is the \emph{curvature $k$-form} associated to the connection $k-1$-form $\omega_A^{(k-1)}$: 
\begin{equation*}
R_A^{(k)}=d_A\omega_A^{(k-1)}+\underset{s+t=k}{\sum_{1\leq s, t\leq k-1}}\ \omega_A^{(s)}\wedge\omega_A^{(t)}.
\end{equation*}
\end{enumerate}
\end{definition}

\begin{remark}\label{remark1}
    The bracket in Equations \eqref{eqnablapartial} and \eqref{eqimpot} is the bracket of forms on $A$ taking values in $\mathrm{End}(E)$, namely the commutator on the space $\Omega(A,\mathrm{End}(E))$. More precisely, for every homogeneous elements $\omega\in\Omega(A,\mathrm{End}(E))_m$ and $\omega'\in\Omega(A,\mathrm{End}(E))_n$, it is
 \begin{equation}\label{eq:commutatorEnd}
     [\omega,\omega']=\omega\wedge \omega'- (-1)^{mn}\omega'\wedge \omega.
 \end{equation} 
 The wedge symbol means that we take the wedge product on the form part and the endomorphism composition on the $\mathrm{End}(E)$ part, respectively.  
 In particular, for sections $\alpha\in \Gamma(\mathrm{End}_i(E))$\ and $\beta\in\Gamma(\mathrm{End}_j(E))$, we have
\begin{equation*}
[\alpha,\beta]=\alpha\circ\beta-(-1)^{ij}\beta\circ\alpha,
\end{equation*}
while for $\eta\otimes \alpha\in\Omega^p(A,\mathrm{End}_i(E))$ and $\omega\otimes \beta\in\Omega^q(A,\mathrm{End}_j(E))$, we have
\begin{equation*}
[\eta\otimes \alpha,\omega\otimes \beta]=(-1)^{iq}(\eta\wedge\omega)\otimes [\alpha,\beta].
\end{equation*} 
 Since $\partial, d^\nabla_A$ and $\omega_A^{(k)}$ are of total degree $+1$, any commutator \eqref{eq:commutatorEnd} between them reduces to an ordinary sum. Eventually, notice that  evaluating Equation \eqref{eqnablapartial} on a section $a\in \Gamma(A)$ gives $\nabla_a\circ\partial=\partial\circ\nabla_a$.
\end{remark}

Following \cite{abadRepresentationsHomotopyLie2011}, a \emph{morphism} between two homotopy $A$-modules $(E,\partial_E)$ and $(F,\partial_F)$ is a $\Omega(A)$-linear map 
$\Phi:\Omega(A,E)_\blacktriangle\to\Omega(A,F)_\blacktriangle$
of total degree 0 
commuting with the $A$-superconnections $D_E$ and $D_F$. Namely, $\Phi$ is a morphism of differential graded $\mathcal{C}^\infty(A[1])$-modules.  Such a morphism consists of a family of linear maps $\Phi^{(i)}\in\Omega^{i}(A,\mathrm{End}_{-i}(E,F))$, and in particular the element $\Phi^{(0)}\in\Gamma(\mathrm{Hom}_0(E,F))$ is a chain map.
Whenever there exists a morphism of homotopy $A$-modules $\Psi:\Omega(A,F)_\blacktriangle\to\Omega(A,E)_\blacktriangle$ such that $\Phi\circ\Psi=\mathrm{id}_{\Omega(A,F)}$ and $\Psi\circ\Phi=\mathrm{id}_{\Omega(A,E)}$, we say that $E$ and $F$ are \emph{isomorphic $A$-modules}. 
Decomposing these two identities, one observes that $\Phi^{(0)}\circ\Psi^{(0)}=\mathrm{id}_{F}$ and $\Psi^{(0)}\circ\Phi^{(0)}=\mathrm{id}_{E}$ so that $E$ and $F$ are isomorphic as graded vector bundles.

\subsection{Extension of superconnections and the bigraded complex  
\texorpdfstring{$A^\circ\otimes\mathrm{End}(E)$}{(L/A)*⊗End(E)}}\label{seq:extension}

 A \emph{Lie algebroid pair}, or \emph{Lie pair}, is a pair of Lie algebroids $L\to M$ and $A\to M$ together with a Lie algebroid inclusion $A\hookrightarrow L$ covering the identity on $M$. The following Lemma generalizes Lemma 2.1 of \cite{chenAtiyahClassesHomotopy2016} to the present context.

\begin{lemma}\label{lem:extension}
Let $A\hookrightarrow L$ be a Lie pair, and $D_A$ be an $A$-superconnection on a (split, finite dimensional) graded vector bundle $E$. Then $D_A$ extends to an $L$-superconnection $D_L$ on~$E$. Moreover, if $D_L$ and $D_L'$ are two such extensions, then $D_L-D_L'\in\Gamma(\wedge^\bullet A^\circ\otimes \mathrm{End}(E))$, where $A^\circ\subset L^*$ denotes the annihilator bundle of $A$.
\end{lemma}

\begin{proof}
Let $\iota:\;A\hookrightarrow L$ be the inclusion map. By duality we have a canonical projection $\iota^\ast:\Omega(L)\to\Omega(A)$, obtained by restricting to $A$ a form initially defined on $L$. Tensoring $\iota^*$ with the identity on $E$, we denote by the same symbol the surjective map $\iota^*: \Omega(L,E)\to\Omega(A,E)$. We can then define an $L$-superconnection $D_L$ on~$E$ as a map that makes the  following diagram commutative \cite{chenAtiyahClassesStrongly2019}:
\begin{center}
\begin{tikzcd}[column sep=2cm,row sep=2cm]
  \Omega(L,E)_{\blacktriangle}\ar[d,"\iota^*" left] \ar[r, dashed, "D_L"]& \Omega(L,E)_{\blacktriangle+1} \ar[d,"\iota^*" right] \\
 \Omega(A,E)_{\blacktriangle} \ar[r, rightarrow ,"D_A" ]   & \Omega(A,E)_{\blacktriangle+1}
\end{tikzcd} 
\end{center}
The difference between two choices $D_L$ and $D_L'$  of $L$-superconnections lands in $\Omega(L/A,E)$, and 
 the canonical isomorphism $A^\circ\simeq(L/A)^*$ gives that $D_L-D_L'\in\Gamma(\wedge^\bullet A^\circ\otimes \mathrm{End}(E))$.
\end{proof}

Let $E$ be a representation up to homotopy of $A$, and  assume that $D_A$ is extended to an $L$-superconnection $D_L$ on $E$. As in Equation \eqref{eq:decompositionD}, this operator can be decomposed into a sum of components $D^{(k)}_L$ characterized as follows:
\begin{itemize}
\item For $k=0$, we have the identity $D_L^{(0)}=D_A^{(0)}=\partial$. 
\item For $k=1$, the (flat) $A$-connection $\nabla$ on $E$ extends to a (non-necessarily flat) $L$-connection on $E$, also denoted  $\nabla$, and $D_L^{(1)}$ is the exterior covariant derivative $d^\nabla_L$ associated to this $L$-connection $\nabla$, so that its action on elements of $\Omega(L,E)$ is given by Equation \eqref{eq:covariantderive}, where $A$ is replaced by $L$.
\item For $2\leq k\leq \mathrm{rk}(A)$, the connection $k$-form $\omega^{(k)}_A$ extends to a $k$-form on $L$ valued in $\mathrm{End}_{1-k }(E)$ and denoted $\omega_L^{(k)}$, and the action of  $D^{(k)}_L$ on $\Omega(L,E)$ is defined along the lines of Equation \eqref{eqokkk}, 
where $A$ is replaced by $L$.
\end{itemize}
Notice that with respect to this setup, the extension to $L$ of the connection $0$-form $\omega_A^{(0)}$ is itself, so that $\omega_L^{(0)}=\omega_A^{(0)}=\partial$ and the first case is a particular case of the last one.
 
 Although the $L$-superconnection $D_L$  does not necessarily satisfy $(D_L)^2=0$, it does whenever it is restricted to $A$, namely: $(D_L|_A)^2=(D_A)^2=0$. 
  As in the classical theory of linear connections, $D_L$ 
  canonically induces an $L$-superconnection on $\mathrm{End}(E)$ which will be simply denoted by $D$ when the context is clear. The superconnection $D$ is defined through the adjoint action
\begin{equation}\label{eqdefD}
    D(\omega)=[D_L,\omega],
\end{equation}
for every $\omega\in\Omega(L,\mathrm{End}(E))$. The commutator is the one defined in Remark \ref{remark1}. 
As above, the $L$-superconnection $D$ can be decomposed into a sum of components~$D^{(k)}$ of various arities, that can be described as follows:
\begin{itemize}
\item For $k=1$, the operator $D^{(1)}$ is the adjoint action of the exterior covariant derivative~$d^{\nabla}_L$ associated to the $L$-connection extending the $A$-connection $\nabla$ on $E$:
\begin{equation}\label{eqaction3}
\omega\in\Omega^\bullet(L,\mathrm{End}_{\bullet}(E))\overset{D^{(1)}}{\xrightarrow{\hspace{1cm}}} [d^{\nabla}_L,\omega] \in\Omega^{\bullet+1}(L,\mathrm{End}_{\bullet}(E)). 
\end{equation}
\item For $k\neq1$, 
the operator $D^{(k)}$ is the adjoint action of the $k$-form $\omega_L^{(k)}$:
\begin{equation}\label{eqaction2}
    \omega\in\Omega^\bullet(L,\mathrm{End}_{\bullet}(E))\overset{D^{(k)}}{\xrightarrow{\hspace{1cm}}} [\omega_L^{(k)},\omega]\in\Omega^{\bullet+k}(L,\mathrm{End}_{\bullet+1-k}(E)).
\end{equation}
In particular, for $k=0$ we have that $D^{(0)}=[\partial,.\,]$.
\end{itemize}
Since $(D_L)^2\neq0$, there is a priori no reason whatsoever that $D^2=0$, namely, that $\mathrm{End}(E)$ is a homotopy $L$-module. However, we will see in the following that the operator $D$ can be associated to a particular homotopy $A$-module structure on $A^\circ\otimes\mathrm{End}(E)$.

Denote as $\widehat{\Omega}(E)^{\bullet,\bullet}:=\Omega^{\bullet}(A,A^\circ\otimes \mathrm{End}_\bullet(E))$ -- or simply $\widehat{\Omega}(E)$ -- the bigraded vector space of differential forms on $A$ with values in $A^\circ\otimes \mathrm{End}(E)$, namely:
\begin{equation*}
    \widehat{\Omega}(E)^{\bullet,\bullet}=\bigoplus_{j=-\infty}^\infty\bigoplus_{k=0}^{\mathrm{rk}(A)}\Omega^k(A,A^\circ\otimes\mathrm{End}_{j}(E)).
\end{equation*}
The sum is finite as we assumed that the graded vector bundle $E$ is split of finite dimension.
There is a total degree on $\widehat{\Omega}(E)$ defined as the sum of the $A$-form degree and the endomorphism degree. More precisely, an element $\varpi\in\widehat{\Omega}(E)$ of total degree $p$ is a finite sum $\varpi=\sum_l\varpi^{(l)}$ of elements $\varpi^{(l)}\in\Omega^{l}(A,A^\circ\otimes \mathrm{End}_{p-l}(E))$.  



Define a differential operator $s$ on  $\widehat{\Omega}(E)$ 
by
\begin{equation*}
\text{$s=\sum_{k=0}^{\mathrm{rk}(A)}s^{(k)}$, where $s^{(k)}:\widehat{\Omega}(E)^{\bullet, \bullet}\to\widehat{\Omega}(E)^{\bullet+k, \bullet+1-k}$,}
\end{equation*}
so that $s^{(1)}$ corresponds to the exterior covariant derivative associated to the action of $A$ on $A^\circ\otimes \mathrm{End}(E)$, while for $k\neq 1$, the image of a section $\varpi\in\Omega^p(A,A^\circ\otimes\mathrm{End}_q(E))$, is
\begin{align}
    s^{(k)}(\varpi)(a_1,\ldots,a_{k+p};l)&=\frac{1}{k!p!}\Big(
     \omega_A^{(k)}(a_1,\ldots,a_k)\circ\varpi(a_{k+1},\ldots,a_{k+p};l)\label{eqoperators1}\\
    &\quad-(-1)^{p+q+1 +k(p+1)}\varpi(a_{k+1},\ldots,a_{k+p};l)\circ \omega_A^{(k)}(a_1,\ldots,a_k)\nonumber\\
    &\hspace{8cm}+\underset{a_1,\ldots,a_{k+p}}{\circlearrowright}\Big).\nonumber
\end{align}
Here $\circ$ denotes the composition of endomorphisms, and $\circlearrowright$ stands for the signed permutation on $a_1,\ldots,a_{k+p}$. The term $k!p!$ comes from the definition of the wedge product. 

For $k=1$, the component $s^{(1)}$ involves the covariant derivative of the dual of the Bott $A$-action on $L/A$, together with the adjoint action $[\nabla,.\,]$ on $\mathrm{End}(E)$. 
More precisely, the quotient vector bundle $L/A$ admits a distinguished $A$-connection $\nabla^{Bott}:\Gamma(A)\times \Gamma(L/A)\to \Gamma(L/A)$, called the \emph{Bott connection}. Letting $p:L\to{L/A}$ be the quotient map, the Bott connection for any $a\in\Gamma(A)$ and $\ell\in\Gamma(L)$ is
\begin{equation}\label{equationBott}
\nabla^{Bott}_a(p(\ell))=p([a,\ell]).
\end{equation}
 One can check that this definition does not depend on the choice of preimage, and that the connection is flat, thus turning $L/A$ into a representation of the Lie algebroid $A$.
We denote $\widehat{\nabla}$ the sum of the induced $A$-connection $[\nabla,\cdot\,]$ on $\mathrm{End}(E)$ and of the dual connection of $\nabla^{Bott}$ acting on $(L/A)^*\simeq A^\circ$, so that $\widehat{\nabla}$ is an $A$-connection on $A^\circ\otimes \mathrm{End}(E)$.
Then, the operator $s^{(1)}:\widehat{\Omega}(E)^{\bullet,\bullet}\to \widehat{\Omega}(E)^{\bullet+1,\bullet}$ is the exterior covariant derivative of $\widehat{\nabla}$: $s^{(1)}=d^{\widehat{\nabla}}_A$. Its action on a $p$-form $\varpi\in \Omega^p(A,A^\circ\otimes \mathrm{End}_\bullet(E))$ reads
\begin{align}
    d^{\widehat{\nabla}}_A
    \varpi(a_0,\ldots,a_p;l)&=\sum_{i=0}^p\,(-1)^{i}\big[\nabla_{a_i},\varpi(a_0,\ldots,\widehat{a_i},\ldots,a_p;l)\big]\label{sconnect}\\
    &\hspace{1cm}+\sum_{0\leq i,j\leq p}\,(-1)^{i+j}\varpi\big([a_i,a_j],a_0,\ldots,\widehat{a_i},\ldots,\widehat{a_j},\ldots,a_p;l\big)\nonumber\\
     &\hspace{1cm}-\sum_{i=0}^p\,(-1)^{i} \varpi\big(a_0,\ldots,\widehat{a_i},\ldots,a_p;\nabla^{Bott}_{a_i}(l)\big).\nonumber
\end{align}
for every $a_1,\ldots, a_p\in\Gamma(A)$ and $l\in\Gamma(L/A)$.
In the first line the bracket is the graded commutator on $\mathrm{End}(E)$ defined in Remark \ref{remark1}; in the second line it is the Lie algebroid bracket. In the last line, the minus sign is justified by the fact that $\widehat{\nabla}$ involves the dual of the Bott connection. The next Lemma follows from the discussion.

\begin{lemma}
The differential operator $s$ is an $A$-superconnection on $A^\circ\otimes\mathrm{End}(E)$.
\end{lemma}

The relationship between $D$ and $s$ can be understood through the following observation. Let $\omega\in\Omega^{p+1}(L,\mathrm{End}(E))$ be a $p+1$-form on $L$ with the property that it is identically vanishing on $\wedge^{p+1}A$. By the canonical isomorphism between $(L/A)^*$ and $A^\circ\subset L^*$, one can define a $p$-form $\varpi$ on $A$ taking values in $A^\circ\otimes \mathrm{End}(E)$ as the restriction of $\omega$ to $\wedge^p A\otimes L/A$. More precisely, if one denotes $\ell$ a preimage in $\Gamma(L)$ of $l\in \Gamma(L/A)$, then $\varpi$ is defined by the following formula:
\begin{equation}\label{eq:restriction}
     \varpi(a_1,\ldots, a_p;l)=\omega(a_1,\ldots, a_p,\ell), 
\end{equation}
for any $a_1,\ldots, a_p\in \Gamma(A)$ and $l\in \Gamma(L/A)$. Another choice of preimage $\ell'$ does not change the left-hand side of Equation \eqref{eq:restriction}, since the difference $\ell'-\ell$ is a section of $A$, and~$\omega$ identically vanishes on $\wedge^{p+1}A$. Moreover, using Equations \eqref{eqaction3} and \eqref{eqaction2}, straightforward calculations show that $D(\omega)$ vanishes identically when restricted to $A$.  Regarding the total degree, the $p+1$-form $\omega\in\Omega^{p+1}(L,\mathrm{End}_{q}(E))$ of total degree $p+q+1$ induces, by the restriction \eqref{eq:restriction}, a $p$-form $\varpi\in \widehat{\Omega}(E)^{p,q}= \Omega^{p}(A,A^\circ\otimes\mathrm{End}_{q}(E))$, of total degree $p+q$. Note that, by definition, the parities of the degrees of $\omega$ and $\varpi$ are opposite to each other.

\begin{lemma}\label{lemgringo}
For every $0\leq k \leq \mathrm{rk}(A)$, the $k$-th component of the $A$-superconnection $s$ on $A^\circ\otimes \mathrm{End}(E)$ is related to the $k$-th component of the $L$-superconnection $D$ on $\mathrm{End}(E)$ by the following formula:
\begin{equation}
s^{(k)}(\varpi)(a_1,\ldots,a_{k+p};l)=D^{(k)}(\omega)(a_1,\ldots,a_{k+p},\ell),\label{eq:relation1}
\end{equation}
where $\omega\in\Omega^{p+1}(L,\mathrm{End}(E))$ has the property that it is identically vanishing on $\wedge^{p+1}A$, and where $\varpi$ is the canonical $p$-form on $A$ obtained from $\omega$ through Equation \eqref{eq:restriction}.
\end{lemma}

\begin{proof}
This comes from a straightforward application of Conventions~\eqref{eq:commutatorEnd} and \eqref{eq:restriction}. 
 Indeed, the right-hand side of Equation~\eqref{eq:relation1} reads:
 \begin{align}
     D^{(k)}(\omega)(a_1,\ldots,a_{k+p},\ell)&=\frac{1}{k!(p+1)!}\Big(
     \omega_L^{(k)}(a_1,\ldots,a_k)\circ\omega(a_{k+1},\ldots,a_{k+p},\ell)\label{eqoperators1000}\\
    &\quad-(-1)^{p+q+1 +k(p+1)}\omega(a_{k+1},\ldots,a_{k+p},\ell)\circ \omega_L^{(k)}(a_1,\ldots,a_k)\nonumber\\
    &\hspace{8cm}+\underset{a_1,\ldots,a_{k+p},\ell}{\circlearrowright}\Big).\nonumber
\end{align}
The signed permutation $\circlearrowright$ is now made over all elements, including $\ell$.

Whenever $\ell$ is entering a slot of $\omega^{(k)}_L$, it means that all slots of 
$\omega$ are filled with sections of~$A$, but $\omega|_{\wedge^{p+1}A}=0$. Hence the only non vanishing contributions in Equation~\eqref{eqoperators1000} come from 
terms of the form $\omega(\ldots,\ell,\ldots)$. But then, among all such terms, many of them are similar once we put 
$\ell$ in the very last slot. More precisely, for each combination of $p$ elements $a_i,\ldots, a_{i+p-1}$, we have:
\begin{align*}
\omega(a_{i},\ldots,a_{i+p-1},\ell)
&+\sum_{k=1}^{p-1} (-1)^{kp}\omega(a_{i+k},\ldots,a_{i+p-1},\ell,a_i,\ldots,a_{i+k-1})
\\
&+(-1)^{p^2}\omega(\ell,a_{i},\ldots,a_{i+p-1})=(p+1)\omega(a_{i},\ldots,a_{i+p-1},\ell)
\end{align*}
Using this trick for all kind of combinations $a_i,\ldots, a_{i+p-1}$, there is a factor $p+1$ coming outside of the parenthesis in Equation \eqref{eqoperators1000}, giving back the right-hand side of Equation~\eqref{eqoperators1} once we use Equation \eqref{eq:restriction} to identify $\varpi$ and $\omega$. \end{proof}

As Equations  \eqref{eqoperators1} and \eqref{sconnect} explicitly show, the $A$-superconnection $s$ does not a priori depend on any choice of $L$-superconnection $D_L$ extending $D_A$. This is not so clear from Equation  \eqref{eq:relation1} but the following proposition confirms the original observation.

\begin{proposition}\label{proprestriction}
   Equation \eqref{eq:relation1} does not depend on the choice of $L$-superconnection $D_L$ extending the $A$-superconnection $D_A$.
\end{proposition}

\begin{proof}
The first component $s^{(0)}$, depending only on $\partial$, does not depend on the extension. Thus, we want to show that the left-hand side of Equations \eqref{eq:relation1}, for $1\leq k\leq \mathrm{rk}(A)$,  does not depend on the choice of extension to $L$ of the $A$-superconnection $D_A$.

Let $D_L$ and $D'_L$ be two extensions to $L$ of the $A$-superconnection $D_A$. Their associated connection forms are denoted $\omega_L^{(k)}$ and $\omega'^{(k)}_L$, respectively. In particular, the connection $1$-form $\omega_L^{(1)}$ (resp. $\omega'^{(1)}_L$) corresponds to a Lie algebroid $L$-connection $\nabla$ (resp. $\nabla'$) extending the Lie algebroid $A$-connection $\nabla$ on $E$. We set  $\varphi^{(1)}=\nabla-\nabla'$ so that $d_L^\nabla-d_L^{\nabla'}=\varphi^{(1)}\wedge\cdot$. 
For every $2\leq k\leq \mathrm{rk}(A)$, we set 
\[\varphi^{(k)}=\omega^{(k)}_L-\omega'^{(k)}_L.\] 
It is a $k$-form on $L$ valued in $\mathrm{End}_{-k+1}(E)$ of total degree $1$, vanishing when restricted to~$\wedge^{k}A$. 
We denote $s$ and $s'$ the two $A$-superconnections associated to the two $L$-superconnections~$D_L$ and $D'_L$ through Equation \eqref{eq:relation1}, respectively.
Let $\omega$ be a $p+1$-form on $L$ valued in $\mathrm{End}(E)$ and vanishing identically on~$\wedge^{p+1}A$. 

For $k=1$, we have that
\begin{equation*}
    (d_L^{\nabla}-d_L^{\nabla'})\omega(a_0,\ldots,a_p,\ell) 
    =\varphi^{(1)}\wedge\omega (a_0,\ldots,a_p,\ell) .
\end{equation*}
The term on the right-hand side reads
\begin{align*}
\varphi^{(1)}\wedge\omega (a_0,\ldots,a_p,\ell)=\frac{1}{(p+1)!}\Big((-1)^{p+1}&\varphi^{(1)}(\ell)\circ\omega(a_0,\ldots,a_p)\\
&+\sum_{j=0}^p (-1)^j\varphi^{(1)}(a_j)\circ\omega(a_0,\ldots,\widehat{a_j},\ldots,a_p,\ell)\Big).
\end{align*}
The factorial $(p+1)!$ comes from the definition of the wedge product.
The first term on the right-hand side is zero because $\omega\big|_{\wedge^{p+1}A}=0$, and the second term vanishes too because $\varphi^{(1)}\big|_{A}=0$. Hence, from Equation \eqref{eq:relation1} for $k=1$, we deduce that $s^{(1)}=s'^{(1)}$.

For $2\leq k\leq \mathrm{rk}(A)$, then $p+k>p+1$, and we have:
\begin{align*}
\big[\omega^{(k)}_L-\omega'^{(k)}_L,\omega\big](a_0,\ldots,a_{k+p},\ell)
&=\big[\varphi^{(k)},\omega\big] (a_0,\ldots,a_{k+p},\ell)\\
&= \frac{1}{k!(p+1)!}\Big(\varphi^{(k)}(a_0,\ldots,a_k)\circ \omega(a_{k+1},\ldots,a_{k+p},\ell)\\
&\qquad-(-1)^{|\omega|}\omega(a_0,\ldots,a_{p+1})\circ  \varphi^{(k)}(a_{p+2},\ldots,a_k, \ell) +\circlearrowright\Big),
\end{align*}
where $|\omega|$ is the total degree of the form $\omega$, and where the symbol ${\circlearrowright}$ means that we take all signed permutations of all entries, including $\ell$. Both terms on the right-hand side vanish because either $\varphi^{(k)}\big|_{\wedge^kA}=0$ or $\omega\big|_{\wedge^{p+1}A}=0$. This proves that $(D^{(k)}-D'^{(k)})\omega(a_0,\ldots,a_{k+p},\ell)=0$, hence proving by Equation \eqref{eq:relation1} that $s^{(k)}=s'^{(k)}$.
\end{proof}


Although the $L$-superconnection $D$ on $\mathrm{End}(E)$ does not necessarily satisfy $D^2=0$, the induced operator~$s$ does.

\begin{proposition}\label{prophomologys}
The $A$-superconnection $s=s^{(0)}+d^{\widehat{\nabla}}_A +s^{(2)} +\ldots+s^{(\mathrm{rk}(A))}$ on $A^\circ\otimes \mathrm{End}(E)$ 
defines a representation up to homotopy of $A$, i.e. it is a total degree $+1$ differential on the bigraded vector space $\widehat{\Omega}(E)^{\bullet,\bullet}$.
\end{proposition}


\begin{proof}The equation $s^2=0$ is equivalent to the following set of equations (together with their meaning on the right-hand side):
\begin{align}
    \big(s^{(0)}\big)^2&=0 &&\text{$s^{(0)}$ is a differential},\label{eqdiff00}\\
    [s^{(0)}, d^{\widehat{\nabla}}_A]
    &=0 &&\text{$s^{(0)}$ and $s^{(1)}=d^{\widehat{\nabla}}_A$ anticommute}, \label{eq:anticommute}\\
    \big(d^{\widehat{\nabla}}_A\big)^2+[s^{(0)},s^{(2)}]&=0&&\text{$d^{\widehat{\nabla}}_A$ is a differential modulo $s^{(0)}$},\label{eq:commutator}
\end{align}
and more generally, for any $k\geq0$:
\begin{equation}\label{eq:highereq}
    \sum_{l=0}^k s^{(l)}s^{(k-l)}=0.
\end{equation}

Let $\omega$ be a $p+1$-form on $L$ taking values in $\mathrm{End}(E)$ with the property that it identically vanishes on $\wedge^{p+1}A$, and let $\varpi$ be the  restricted $p$-form on $A$ valued in $A^\circ\otimes\mathrm{End}(E)$ obtained by Equation~\eqref{eq:restriction}. 
Then, $D^{(0)}(\omega)=[\partial,\omega]$ is also vanishing on $A$ so, by applying
Equation~\eqref{eq:relation1} to $D^{(0)}(\omega)$, we have that 
\begin{equation*}\big(s^{(0)}\big)^2(\varpi)(a_1,\ldots,a_p;l)=\big(D^{(0)}\big)^2(\omega)(a_1,\ldots,a_p,\ell).
\end{equation*}
By Equation \eqref{eqaction2}, the right-hand side corresponds to $[\partial,[\partial,\omega(a_1,\ldots,a_p,\ell)]]$, where the commutator is the graded commutator on $\Omega(L,\mathrm{End}(E))$ extending that on $\Omega(A,\mathrm{End}(E))$ defined in Remark \ref{remark1}).  This $\mathrm{End}(E)$-valued $A$-form can be further rewritten as $\big[\frac{1}{2}[\partial,\partial],\omega\big]$, which squares to zero 
so that Equation \eqref{eqdiff00} is straightforwardly satisfied.  

Regarding Equation \eqref{eq:anticommute}, we will use the identity $[\nabla_{a},\partial]=0$, valid for every $a\in\Gamma(a)$. This identity indeed implies that the first summand on the right hand side of Equation~\eqref{sconnect} reads
\begin{equation*}
    s^{(0)}\big([\nabla_{a_i},\varpi(a_0,\ldots,\widehat{a_i},\ldots,a_p;l)]\big)=\big[\nabla_{a_i},s^{(0)}(\varpi)(a_0,\ldots,\widehat{a_i},\ldots,a_p;l)\big].
\end{equation*}
Then, since  the other terms of Equation \eqref{sconnect} trivially commute with $\partial$, we obtain that $[\partial,d_A^{\widehat{\nabla}}]=0$, implying in turn Equation \eqref{eq:anticommute}.
Finally, Equation \eqref{eq:commutator} can be shown directly through Equation \eqref{sconnect}, 
by applying  $d^{\widehat{\nabla}}_A$ twice  on a section $f\in\Gamma(A^\circ\otimes \mathrm{End}(E))$:
\begin{equation*}
(d^{\widehat{\nabla}}_A)^2(f)(a,b; l)=\big[ R^{(2)}_A(a,b),f(l)\big]-f(R_{\nabla^{Bott}}(a,b)(l)),
\end{equation*}
where $R_{\nabla^{Bott}}=0$ is the curvature of the Bott connection.
Using the identity $R^{(2)}_A=-[\partial,\omega^{(2)}]$, 
one obtains Equation \eqref{eq:commutator} restricted to $\Gamma(A^\circ\otimes \mathrm{End}(E))$.
The argument is still valid for any $p$-form  $\varpi\in\widehat{\Omega}(E)^{p,q}$ of total degree $p+q$, but the computation is more cumbersome.

More generally, 
the annihilator bundle $A^\circ$ together with the Bott connection is a Lie algebroid representation of $A$ while $\mathrm{End}(E)$ is a homotopy $A$-module. Their tensor product is then a homotopy $A$-module as well, and it is governed by Equation \eqref{eq:highereq}, which is Equation  \eqref{eqimpot} adapted to the particular  present context.
\end{proof}

\subsection{Atiyah classes of representations up to homotopy}

In the classical case, i.e. when we consider a Lie algebroid representation $E=E_0$ of $A$ extended to an $L$-connection on $E$, the covariant derivative $d^\nabla_L$ acting on $\Omega^\bullet(L,E)$, can be seen as a degree $+1$ vector field on the said space of functions on an appropriate graded manifold\footnote{Namely, it is the pullback graded vector bundle $\pi^*E\to L[1]$ along the projection $\pi\colon L[1]\to~M$.}, so it necessarily satisfies
\begin{equation*}[d^\nabla_L,[d^\nabla_L,d^\nabla_L]]=0.\end{equation*}
This equation is precisely the  Bianchi identity $d^{\nabla}_L(R^L_\nabla)=0$. Here and in the following, the bracket is the graded commutator of endomorphisms as defined in Remark \ref{remark1}, where~$A$ is replaced by $L$.

Now let us prove a statement similar to the Bianchi identity for the operator $D_L$ extending $D_A$. First, we set
\begin{equation}\label{eq:Dcurvature0}
    a^{(1)}= [d_L^\nabla,\partial]. 
\end{equation}
The right-hand side is symmetric, in the sense that $ [d_L^\nabla,\partial]=[\partial,d_L^\nabla]$, and from Equation~\eqref{eqaction3}, we see that it is  
 a 1-form on $L$ taking values in $\mathrm{End}_{1}(E)$.
 For every $k\geq2$,  define the following $k$-form on $L$ taking values in $\mathrm{End}_{2-k}(E)$:
\begin{equation}\label{eq:Dcurvature}
a^{(k)}=R_L^{(k)}+[\partial,\omega_L^{(k)}],
\end{equation}
where by convention $\omega_L^{(k)}$ is assumed to be zero whenever $k\geq\mathrm{rk}(A)+1$, and where
\begin{equation}\label{eq:curvatureL}
R_L^{(k)}=d_L\omega_L^{(k-1)}+\underset{s+t=k}{\sum_{1\leq s, t\leq k-1}}\ \omega_L^{(s)}\wedge\omega_L^{(t)}.
\end{equation}
By construction, for every $k\geq1$, we have that
\begin{equation}\label{eq:kojima}
a^{(k)}\big|_{\wedge^{k}A}=0.
\end{equation}
For $1\leq k\leq \mathrm{rk}(A)$, this comes from the homotopy $A$-module structure on $E$, while for $k\geq \mathrm{rk}(A)+1$, this comes from the purely algebraic fact that there is no form on $A$ of form degree higher than  $\mathrm{rk}(A)$. 


Since the highest possibly non-vanishing connection $k$-form is $\omega_L^{(\mathrm{rk}(A))}$, the curvature $k$-form $R_L^{(k)}$ -- and thus also $a^{(k)}$ -- vanishes for $k\geq 2\,\mathrm{rk}(A)+1$. Then, the sums
$a=\sum_{k\geq1}a^{(k)}$ and $R_L=\sum_{k\geq2}R_L^{(k)}$ 
are finite. Furthermore, the sum $\sum_{k\geq2}\omega_L^{(k)}$ terminates at $k=\mathrm{rk}(A)$, and we set $\omega_L=\sum_{k=2}^{\mathrm{rk}(A)}\omega_L^{(k)}$.  From the observation that in the present setup the $L$-superconnection is $D_L=\partial+d_L^{\nabla}+\omega_L$, one deduces that Equations~\eqref{eq:Dcurvature0} and~\eqref{eq:Dcurvature} are equivalent to
\begin{equation}
a=(D_L)^2=[d_L^\nabla,\partial]+R_L\wedge\cdot+[\partial,\omega_L].\label{eq:aD}
\end{equation} 
This equation is reminiscent of the definition of the curvature of a connection $(d^\nabla)^2=R_\nabla\wedge\cdot$.
Thus, the following Lemma establishes a generalization of the Bianchi identity adapted to the present setup.

\begin{lemma}\label{prop:Bianchi} Let $A\hookrightarrow L$ be a Lie pair and $E$ be a representation up to homotopy of $A$. Let $D_A$ be the corresponding flat $A$-superconnection and assume it is extended to an $L$-superconnection $D_L$, and set $D=[D_L,.\,]$. 
Then 
\begin{equation}\label{eq:curvature2}
    \sum_{l=0}^k\,D^{(l)}a^{(k+1-l)}=0, \qquad \text{for every $k\geq0$.} 
\end{equation}
\end{lemma}

\begin{proof} The identities can be shown by direct computation but  this is cumbersome. It is indeed a straightforward computation to show that the set of Equations \eqref{eq:curvature2} is equivalent to the equation $D(a)=0$. Instead we will use a simpler, graded geometric approach.

In the present case, where we consider representations up to homotopy, the covariant derivative $d^\nabla_L$ is replaced by the $L$-superconnection $D_L$, but the idea is the same. As in the classical case the operator $D_L$ does not necessarily square to $0$ but, as a degree 1 operator on $\Omega(L,\mathrm{End}(E))$, it satisfies 
\begin{equation}
[D_L,[D_L,D_L]]=0.\label{eq:3D}
\end{equation}
Since $[D_L,D_L]=2\, (D_L)^2$, Equations  \eqref{eqdefD}, \eqref{eq:aD}
 and \eqref{eq:3D} imply that $D(a)=0$.
\end{proof}

From Equation \eqref{eq:kojima}, 
we deduce that the $k+1$-forms $a^{(k+1)}\in\Omega^{k+1}(L,\mathrm{End}_{1-k}(E))$ ($0\leq k\leq\mathrm{rk}(A)$) defined by Equations \eqref{eq:Dcurvature0} and \eqref{eq:Dcurvature}, give rise, through Equation \eqref{eq:restriction}, to a family of $k$-forms $\alpha^{(k)} \in\Omega^k(A,A^\circ\otimes\mathrm{End}_{1-k}(E))$.
In other words, if one denotes $\ell$ a preimage in $\Gamma(L)$ of $l\in \Gamma(L/A)$, we have:
\begin{align}
    \alpha^{(0)}(l)&
    =    [\nabla_\ell,\partial],\label{eq:alpha0}\\
    \alpha^{(k)}(a_1,\ldots,a_k;l)&=R_L^{(k+1)}(a_1,\ldots,a_k,\ell)+[\partial,\omega^{(k+1)}_L](a_1,\ldots,a_k,\ell),\label{eq:alpha4}
\end{align}
for every $a_1,\ldots, a_k\in\Gamma(A)$, and $1\leq k\leq\mathrm{rk}(A)$. 
Equation \eqref{eq:alpha0} is obtained through direct calculation from Equation \eqref{eq:Dcurvature0}, when applied to a section of $E$. Using Equation \eqref{eq:commutatorEnd}, the last term in Equation \eqref{eq:alpha4} reads $\partial\circ\omega^{(k+1)}_L(a_1,\ldots,a_k,\ell)+\omega^{(k+1)}_L(a_1,\ldots,a_k,\ell)\circ\partial$. 

The family of forms $\alpha^{(k)}$ defines an element $\alpha=\sum_{k=0}^{\mathrm{rk}(A)}\alpha^{(k)}$ of total degree $+1$  in the bigraded vector space $\widehat{\Omega}(E)$. 
 The following notion connects this family of forms with the operators $D_A$ and $D_L$, by generalizing to the context of representations up to homotopy the notion of compatible $A$-connection introduced in \cite[Definition 2.2]{chenAtiyahClassesHomotopy2016}, and their relationship with the notion of Atiyah class \cite[Theorem 2.5]{chenAtiyahClassesHomotopy2016}.

\begin{definition}\label{def:compatible} Let $A\hookrightarrow L$ be a Lie pair and $(E,D_A)$ be a representation up to homotopy of~$A$.
An $L$-superconnection $D_L$ on $E$ is said to be \emph{homotopy $A$-compatible} if
\begin{enumerate}
\item it extends the given homotopy $A$-module structure on $E$, i.e. $D_L|_A=D_A$, and
\item the family of forms defined by Equations \eqref{eq:alpha0} and \eqref{eq:alpha4} satisfies $\alpha^{(k)}=0$, for all $k$. 
\end{enumerate}
\end{definition}




\begin{theorem}\label{main1}
Let $A\hookrightarrow L$ be a Lie pair and $(E,D_A)$ be a representation up to homotopy of $A$. Assume that  the $A$-superconnection $D_A$ is  extended to   
an $L$-superconnection $D_L$ on $E$. Then,
\begin{enumerate}
    \item The form $\alpha=\sum_k\alpha^{(k)}$ is a  cocycle of the cochain complex $\big(\widehat{\Omega}(E),s\big)$; 
    \item Its cohomology class $[\alpha]\in H\big(\widehat{\Omega}(E),s\big)$ does not depend on the choice of extension~$D_L$ of the $A$-superconnection $D_A$; 
    \item The class $[\alpha]$ vanishes if an only if there exists a homotopy $A$-compatible $L$-superconnection on $E$.
\end{enumerate}
\end{theorem}

\begin{proof}

\textbf{Item 1:} 
Being a cocycle with respect to the total differential $s$ is equivalent to (see Figure~\ref{figure1})
\begin{equation*}
    \sum_{l=0}^k\,s^{(l)}\alpha^{(k-l)}=0.
\end{equation*}
This is a particular case of Equation \eqref{eq:curvature2}, once translated with Equation \eqref{eq:relation1}.

\begin{figure}[ht!]
  \centering
  \begin{tikzpicture}[scale=1]
    \coordinate (Origin)   at (0,0);
    \coordinate (XAxisMin) at (-1,0);
    \coordinate (XAxisMax) at (13,0);
    \coordinate (YAxisMin) at (6,0);
    \coordinate (YAxisMax) at (6,10);
    \draw [ultra thick, black,-latex] (XAxisMin) -- (XAxisMax) node [below] {\begin{tabular}[c]{@{}c@{}}endomorphism\\  degree\end{tabular}};
    \draw [ultra thick, black,-latex] (YAxisMin) -- (YAxisMax) node [above] {form degree};

    \clip (-1,-1) rectangle (14cm,9cm); 
    \coordinate (Btwo) at (-1,7);
    \coordinate (B2) at (10,2);
    \coordinate (B3) at (10,4);
    \coordinate (B4) at (12,2);
    \coordinate (B5) at (8,6);
    \coordinate (B6) at (6,8);
    \draw[style=help lines,dashed] (0,1) grid[step=2cm] (5,10);
    \draw[style=help lines,dashed] (7,1) grid[step=2cm] (12,10);
    \draw[style=help lines,dashed] (0,0) grid[step=2cm] (0,1);
    \draw[style=help lines,dashed] (2,0) grid[step=2cm] (2,1);
    \draw[style=help lines,dashed] (4,0) grid[step=2cm] (4,1);
    \draw[style=help lines,dashed] (8,0) grid[step=2cm] (8,1);
    \draw[style=help lines,dashed] (10,0) grid[step=2cm] (10,1);
    \draw[style=help lines,dashed] (12,0) grid[step=2cm] (12,1);
    \draw[style=help lines,dashed] (5,2) grid[step=2cm] (7,2);
    \draw[style=help lines,dashed] (5,4) grid[step=2cm] (7,4);
    \draw[style=help lines,dashed] (5,6) grid[step=2cm] (7,6);
    \draw[style=help lines,dashed] (5,8) grid[step=2cm] (7,8);
    \foreach \x in {0,1,...,5}{
      \foreach \y in {1,2,...,5}{
        \node[draw,circle,inner sep=1pt,fill] at (2*\x,2*\y) {};
        \node[draw,circle,inner sep=1pt,fill,red] at (0,6){};
        \node[draw,circle,inner sep=1pt,fill,red] at (6,0){};
         \node[draw,circle,inner sep=1pt,fill,blue] at (4,4){};
          \node[draw,circle,inner sep=1pt,fill,blue] at (6,2){};
           \node[draw,circle,inner sep=1pt,fill,blue] at (8,0){};
        \node[draw,circle,inner sep=1pt,fill,blue] at (2,6){};
        \node[draw,circle,inner sep=1pt,fill,blue] at (0,8){};
        \node[draw,circle,inner sep=1pt,fill,red] at (2,4){};
         \node[draw,circle,inner sep=1pt,fill,red] at (4,2){};
         \node[draw,circle,inner sep=1pt,fill,red] at (6,0){};
       \node[draw,circle,inner sep=1pt,fill,DarkGreen] at (B2){};
      }
    }
    \node [below] at (6,0)  {$0$};
    \node [below] at (4,0)  {$-1$};
    \node [below] at (2,0)  {$-2$};
    \node [below] at (8,0)  {$1$};
    \node [below] at (10,0)  {$2$};
    \node [below left] at (6,2)  {$1$};
    \node [below left] at (6,4)  {$2$};
    \node [below left] at (6,6)  {$3$};
    
    \draw [ultra thick,red] (Btwo) -- (6,0);
    \node [below left,red] at (3,3) {\large $\phi$};
    \draw [ultra thick,DarkGreen, -latex] (B2) -- (B4);
    \draw [ultra thick,DarkGreen, -latex] (B2) -- (B3);
    \draw [ultra thick,DarkGreen, -latex] (B2) -- (B5);
    \draw [ultra thick,DarkGreen, -latex] (B2) -- (B6);
    \node [below right,DarkGreen] at (11,2) {\large $s^{(0)}=[\partial,.\,]$};
    \node [right,DarkGreen] at (10,3) {\large $s^{(1)}=d^{\widehat{\nabla}}_A$};
    \node [above,DarkGreen] at (9,5) {\large $s^{(2)}$};
     \node [right,DarkGreen] at (7,7) {\large $s^{(3)}$};
    \draw [ultra thick,blue] (-1,9)
        -- (8,0);
    \node [above right,blue] at (5,3) {\large $\alpha$};
  \end{tikzpicture}
\caption{Representation of the bi-graded space $\widehat{\Omega}(E)=\Omega^{\bullet}(A,A^\circ\otimes \mathrm{End}_\bullet(E))$ and an~Atiyah cocycle $\alpha$ of total degree $+1$, together with an element $\phi$ of total degree $0$. The component $\alpha^{(p)}$ sits at the node of coordinates $(1-p,p)$.}   \label{figure1}
\end{figure}
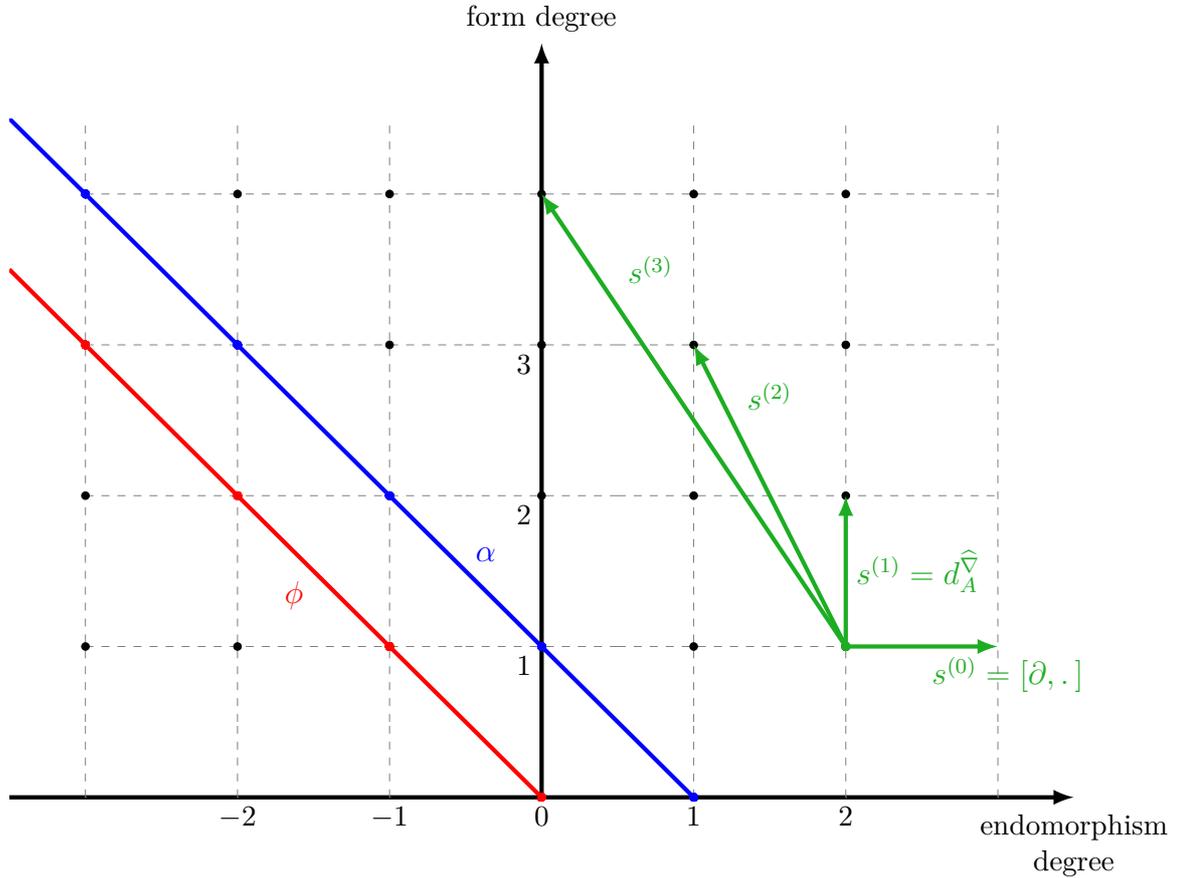

\textbf{Item 2:} Let $D_L'$ be another choice of extension to $L$ of the $A$-superconnection $D_A$.
For every $2\leq k\leq \mathrm{rk}(A)$, let $\omega_L'^{(k)}$ be the associated connection $k$-form, which thus extends the connection $k$-form $\omega_A^{(k)}$. Then, for every $k\geq2$ the difference $\varphi^{(k)}=\omega_L^{(k)}-\omega_L'^{(k)}$ vanishes on  $\wedge^{k}A$. Furthermore, we denote by $\nabla'$ the Lie algebroid $L$-connection on $E$ associated to first component of $D_L'$, that is: $d_L^{\nabla'}=D_L'^{(1)}$. 
We then set $\varphi^{(1)}=\nabla-\nabla'$, so that $d^{\nabla'}_L=d^{\nabla}_L-\varphi^{(1)}\wedge\cdot$, and $\varphi^{(1)}|_A=0$.  Eventually, denote $\phi^{(k)}$ the $k$-form on $A$ obtained from $\varphi^{(k+1)}$ through Equation \eqref{eq:restriction}; it takes values in $A^\circ\otimes\mathrm{End}_{1-k}(E)$. 

Let us denote $a'^{(k+1)}$  the $\mathrm{End}(E)$-valued $k+1$-form on $L$ associated to $D_L'$ through Equation \eqref{eq:aD}, and let $\alpha'^{(k)}$ be the corresponding $k$-form on $A$ obtained from the former via Equations~\eqref{eq:alpha0} and \eqref{eq:alpha4}.  Finally, let $b^{(k+1)}=a^{(k+1)}-a'^{(k+1)}$ and $\beta^{(k)}=\alpha^{(k)}-\alpha'^{(k)}$, for every $0\leq k\leq\mathrm{rk}(A)$. By construction, $b^{(k+1)}$ vanishes on $\wedge^{k+1}A$, and its associated $k$-form on $A$  induced through Equation \eqref{eq:restriction} is $\beta^{(k)}$:
\begin{equation}\label{eq:restrict1}
\beta^{(k)}(a_1,\ldots,a_k;l)=b^{(k+1)}(a_1,\ldots,a_k,\ell).
\end{equation}
Setting $\beta=\sum_k\beta^{(k)}$, we now prove that $\beta=s(\phi)$. From Equation \eqref{eq:Dcurvature0}, we have 
\begin{equation}\label{eq:karte1}
b^{(1)}=[\partial, d^{\nabla}_L]-[\partial, d^{\nabla'}_L]=[\partial,\varphi^{(1)}\wedge].
\end{equation}
This $\mathrm{End}(E)$-valued $L$-form vanishes on $A$, and its restriction through Equation \eqref{eq:restrict1} is~$\beta^{(0)}$.
Applying Equation \eqref{eq:relation1} to the right-hand side of Equation \eqref{eq:karte1}, we deduce that $\beta^{(0)}=s^{(0)}(\phi^{(0)})$.

For $1\leq k\leq\mathrm{rk}(A)$, let us denote $R^{(k+1)}_L$ (resp. $R'^{(k+1)}_L$) the curvature $k+1$-form associated to the connection $k$-form $\omega^{(k)}_L$ (resp. $\omega'^{(k)}_L$) via Equation \eqref{eq:curvatureL}. Then we have
\begin{align*}
R^{(k+1)}_L-R'^{(k+1)}_L&=d_L\left(\omega^{(k)}_L-\omega'^{(k)}_L\right)+\underset{s+t=k+1}{\sum_{1\leq s, t\leq k}}\ \omega^{(s)}_L\wedge\omega^{(t)}_L-\omega'^{(s)}_L\wedge\omega'^{(t)}_L\\
&=d_L\varphi^{(k)}+\underset{s+t=k+1}{\sum_{1\leq s, t\leq k}}\ \varphi^{(s)}\wedge\omega'^{(t)}_L+\omega^{(s)}_L\wedge\varphi^{(t)}\\
&=\sum_{j=1}^{k}D'^{(j)}(\varphi^{(k-j+1)})+\underset{s+t=k+1}{\sum_{1\leq s, t\leq k}}\ \varphi^{(s)}\wedge\varphi^{(t)}.
\end{align*}
The definitions of $a^{(k+1)}$ and $a'^{(k+1)}$ then establish that
\begin{equation*}
b^{(k+1)}=\sum_{j=0}^{k}D'^{(j)}(\varphi^{(k-j+1)})+\underset{s+t=k+1}{\sum_{1\leq s, t\leq k}}\ \varphi^{(s)}\wedge\varphi^{(t)}.
\end{equation*}
Applying these terms to $(a_1,\ldots,a_k,\ell)$, we observe that the last sum identically vanishes as $\varphi^{(s)}\big|_{\wedge^sA}=0$ for every $s$, while, using Equation \eqref{eq:relation1}, together with Proposition \ref{proprestriction}, each term of the first sum reads
 \begin{equation*}
D'^{(j)}(\varphi^{(k-j+1)})(a_1,\ldots,a_k,\ell)=s^{(j)}\phi^{(k-j)}(a_1,\ldots,a_k;l), \qquad\text{for $0\leq j\leq k$.}
\end{equation*}
Thus, from Equation \eqref{eq:restrict1}, we deduce that $\beta^{(k)}=\sum_{j=0}^{k}s^{(j)}(\phi^{(k-j)})$.
The result being true for every $1\leq k\leq \mathrm{rk}(A)$, it implies that $\alpha=\alpha'+s(\phi)$, i.e. the cohomology class of $\alpha$  does not depend on the choice of extension of the connection $k$-forms to $L$.

\textbf{Item 3:} The Atiyah cocycle $\alpha$ being $s$-exact means that there exists a family of forms $\phi^{(k)}\in\Omega^{k}(A,A^{\circ}\otimes\mathrm{End}_{-k}(E))$ (of total degree 0) such that the form
\begin{equation}\label{eq:equationexactzero}
    \phi=\sum_{k=0}^{\mathrm{rk}(A)}\phi^{(k)}
\end{equation}
satisfies the identity
\begin{equation}\alpha=s(\phi).\label{eq:equationexact}
\end{equation}
As usual, in the sum of Equation \eqref{eq:equationexactzero} we used the convention introduced in Remark \ref{remarksevere}. 

By definition, $\phi^{(0)}\in\Gamma(A^\circ\otimes \mathrm{End}_0(E))$. 
Using \eqref{eq:restriction}, it canonically induces a 1-form $\varphi^{(1)}\in\Omega^1(L,\mathrm{End}_0(E))$  of  total degree $+1$, which vanishes on $A$. More precisely, letting $\ell$ be a preimage in $\Gamma(L)$ of $l\in \Gamma(L/A)$, then 
\begin{equation*}
\varphi^{(1)}(\ell)=\phi^{(0)}(l).
\end{equation*} Any other choice of pre-image $\ell'$ for $l$ does not change the right-hand side, since the difference $\ell'-\ell$ sits in $\Gamma(A)$.
Let $\nabla$ be the Lie algebroid $A$-connection on $E$ corresponding to the given homotopy $A$-module structure on $E$; we denote its extension to $L$ with the same symbol. Then,  $\nabla'=\nabla-\varphi^{(1)}$ is another Lie algebroid $L$-connection on $E$ extending the Lie algebroid $A$-connection, since $\varphi^{(1)}\big|_A=0$. The associated exterior covariant derivatives satisfy $d^{\nabla'}_L=d^\nabla_L-\varphi^{(1)}\wedge\cdot$.

More generally, for every $k\geq2$, every $k-1$-form $\phi^{(k-1)}$  induces a  (non-necessarily unique) $k$-form $\varphi^{(k)}$ on $L$ taking values in $\mathrm{End}_{1-k}(E)$ (hence, of total degree $+1$) satisfying the identity
\begin{equation*}
    \varphi^{(k)}(a_1,\ldots,a_{k-1},\ell)=\phi^{(k-1)}(a_1,\ldots,a_{k-1};l),
\end{equation*}
for every $a_1,\ldots, a_{k-1}\in \Gamma(A)$ and $\ell\in\Gamma(L)$ with projection $l$ in $\Gamma(L/A)$.
In particular, we have that $\varphi^{(k)}\big|_{\wedge^{k}A}=0$, and 
we set $\omega'^{(k)}_L=\omega^{(k)}_L-\varphi^{(k)}$, so that they coincide whenever restricted to $A$. These connection $k$-forms, together with $d^{\nabla'}_L$, then define a new $L$-superconnection $D'_L$ on $E$ which  extends the $A$-superconnection $D_A$. Let
$\alpha'=\sum_{k\geq 0}\alpha'^{(k)}$ be the  cocycle associated to this newly defined $L$-superconnection $D_L'$. The fact that
\begin{equation*}
    \alpha'^{(k)}=0,\qquad \text{for every $k\geq0$,}
\end{equation*}
can be shown by properly writing the arguments. 

For $k=0$, by definition we have $\alpha^{(0)}(l)=a^{(1)}(\ell)$, i.e. using Equation \eqref{eq:Dcurvature0}, it reads 
\begin{equation}\label{eqzooro}
s^{(0)}(\phi^{(0)})(l)= [\partial, d_L^\nabla](\ell).
\end{equation}
Then, using Equation \eqref{eq:relation1} on the left-hand side, Equation \eqref{eqzooro} can be recast as
\begin{equation}\label{eq:orwell}
[\partial,\varphi^{(1)}\wedge\cdot](\ell)=   [\partial, d_L^\nabla](\ell).
\end{equation}
Since $d^{\nabla'}_L=d^\nabla_L-\varphi^{(1)}\wedge\cdot$, we obtain that Equation \eqref{eq:orwell} reads
\begin{equation*}
   [\partial, d_L^{\nabla'}](\ell)= 0.
\end{equation*}
From Equation \eqref{eq:Dcurvature0}, it means that $a'^{(1)}(\ell)=0$, and since $\alpha'^{(0)}(l)=a'^{(1)}(\ell)$, this identity establishes that $\alpha'^{(0)}=0$.

Let $1\leq k\leq \mathrm{rk}(A)$.
The $k+1$-curvature associated to the connection $k$-form $\omega'^{(k)}_L$ has been given in Equation \eqref{eq:curvatureL}:
\begin{align*}
R'^{(k+1)}_L
&=d_L\omega'^{(k)}_L+\underset{s+t=k+1}{\sum_{1\leq s, t\leq k}}\ \omega'^{(s)}_L\wedge\omega'^{(t)}_L\\
&=R^{(k+1)}_L-d_L\varphi^{(k)} -\underset{s+t=k+1}{\sum_{1\leq s, t\leq k}}\ \varphi^{(s)}\wedge\omega^{(t)}_L +\omega^{(s)}_L\wedge\varphi^{(t)}-\varphi^{(s)}\wedge\varphi^{(t)}\\
&=R^{(k+1)}_L-\sum_{j=1}^kD^{(j)}(\varphi^{(k-j+1)})+\underset{s+t=k+1}{\sum_{1\leq s, t\leq k}}\ \varphi^{(s)}\wedge\varphi^{(t)}.
\end{align*}
From this identity, we deduce that $a'^{(k+1)}=R'^{(k+1)}_L+[\partial,\omega'^{(k+1)}_L]$ is equal to:
\begin{equation*}
a'^{(k+1)}=a^{(k+1)}-\sum_{j=0}^kD^{(j)}(\varphi^{(k-j+1)})+\underset{s+t=k+1}{\sum_{1\leq s, t\leq k}}\ \varphi^{(s)}\wedge\varphi^{(t)}.
\end{equation*}
When this identity is applied to a tuple $(a_1,\ldots, a_k,\ell)$, the last sum identically vanishes since $\varphi\big|_{\wedge^sA}=0$ for every $1\leq s\leq k$, so that we obtain
\begin{equation*}
a'^{(k+1)}(a_1,\ldots, a_k,\ell)=a^{(k+1)}(a_1,\ldots, a_k,\ell)-\sum_{j=0}^kD^{(j)}(\varphi^{(k-j+1)})(a_1,\ldots, a_k,\ell).
\end{equation*}
Using Equation \eqref{eq:restriction}, we have 
\begin{equation*}
\alpha'^{(k)}(a_1,\ldots, a_k;l)=\alpha^{(k)}(a_1,\ldots, a_k;l)-\sum_{j=0}^ks^{(j)}(\phi^{(k-j)})(a_1,\ldots, a_k;l).
\end{equation*}
We recognize on the right-hand side the expression of Equation \eqref{eq:equationexact} at form degree $k$, 
so we deduce that $\alpha'^{(k)}=0$.  The result being true for every $1\leq k\leq \mathrm{rk}(A)$, we have that $\alpha'=0$, as desired.
\end{proof}

The cocycle $\alpha$ will be henceforth called \textit{Atiyah cocycle} associated to the homotopy $A$-module $E$ with respect to the Lie pair $A\hookrightarrow L$, and the corresponding class $[\alpha]$ will be called \emph{Atiyah class} of $E$.

\begin{proposition}\label{propisom}
Let $E$ and $F$ be two isomorphic representations up to homotopy of~$A$, and denote their respective Atiyah cocycles by $\alpha_E$ and $\alpha_F$ (with respect to extensions to $L$  of their respective $A$-superconnections). There exists an isomorphism in cohomology mapping $[\alpha_E]$ to $[\alpha_F]$.
\end{proposition}

\begin{proof}
Let $\Phi:\Omega(A,E)_\blacktriangle\to\Omega(A,F)_\blacktriangle$ and $\Psi:\Omega(A,F)_\blacktriangle\to\Omega(A,E)_\blacktriangle$ be two morphisms of homotopy $A$-modules such that $\Phi\circ\Psi=\mathrm{id}_{\Omega(A,F)}$ and $\Psi\circ\Phi=\mathrm{id}_{\Omega(A,E)}$. Recall that in that case the graded vector bundles $E$ and $F$ are isomorphic, i.e. that $E_i\simeq F_i$ for every $i\in\mathbb{Z}$.
The dual morphism $\Psi^*:\Omega(A,E^*)_\blacktriangle\to\Omega(A,F^*)_\blacktriangle$ is an isomorphism between the homotopy $A$-modules $E^*$ and $F^*$. By the isomorphism $\mathrm{End}_i(E)\simeq\bigoplus_{j+k=i}E_j\otimes E^*_k$, 
the tensor product $\Phi\otimes\Psi^*$ gives rise to an isomorphism of homotopy $A$-modules $\Omega(A,\mathrm{End}(E))_\blacktriangle\to\Omega(A,\mathrm{End}(F))_\blacktriangle$.

From this, we deduce that $(\widehat{\Omega}(E),s_E)$ and $(\widehat{\Omega}(F),s_F)$ are isomorphic as differential graded $\Omega(A)$-modules, through a $\Omega(A)$-linear map $\Xi:\widehat{\Omega}(E)\to \widehat{\Omega}(F)$ of total degree 0. This shows that the cohomologies of $s_E$ and $s_F$ are in bijection. Even more, since the Atiyah cocycle $\alpha_E$ is mapped through $\Xi$ to the Atiyah cocycle $\alpha_F$, we deduce that the Atiyah class $[\alpha_E]$ is mapped through $\Xi$ to the Atiyah class $[\alpha_F]$, hence the result.
\end{proof}

\subsection{Relationship with Atiyah classes of Lie algebroid pairs}\label{subway}

Although Proposition \ref{propisom} sheds light on Atiyah classes of isomorphic representations up to homotopy, it is too restrictive for establishing a relationship between the Atiyah class of a representation up to homotopy $E$ of a Lie algebroid $A$ and that of a Lie algebroid representation $K$ of $A$. This is the objective of the present section, whenever the homotopy $A$-module $E$ is a resolution of some $K$.

In this section we assume that the  representation up to homotopy $E$ of the Lie algebroid~$A$ is only graded over non-positive degrees, i.e $E_\bullet=\bigoplus_{-n\leq i\leq0}E_{i}$ for some $n\geq0$ -- so in particular the amplitude of $E$ is $n+1$.
Assume  moreover that the chain complex $(E,\partial)$ is \emph{regular}, i.e. that for every $i\leq0$, the vector bundle morphism $\partial:E_{i}\to E_{i+1}$ has constant rank. This assumption allows to define the cohomology of the chain complex $(E,\partial)$ as the following quotient vector bundle:
\begin{equation*}
\mathcal{H}^i(E,\partial)=\bigslant{\mathrm{Ker}(\partial:E_i\to E_{i+1})}{\mathrm{Im}(\partial:E_{i-1}\to E_i)}.
\end{equation*}
Then,  
$\mathcal{H}^{\bullet}(E,\partial)$ is a graded vector bundle, which can be equipped with the zero differential. 

Assume now that the cohomology $\mathcal{H}^\bullet(E,\partial)$ is concentrated in degree 0,  and set $K=\mathcal{H}^0(E,\partial)=E_0/\mathrm{Im}(\partial:E_{-1}\to E_0)$. In other words, the chain complex $(E,\partial)$ is a resolution of the vector bundle $K$. We denote $\phi:E\to K$ the vector bundle morphism satisfying the chain map condition projecting $E$ to its cohomology:

\begin{center}
\begin{tikzcd}[column sep=2cm,row sep=1.5cm]
  \ldots\ar[r,  "\partial"]&E_{-2}\ar[d,"\phi" left]\ar[r,  "\partial"]&E_{-1}\ar[d,"\phi" left] \ar[r,  "\partial"]& E_{0}\ar[d,"\phi" left] \ar[r,  "0"]& 0 \\
 \ldots\ar[r, dashed,  "0"]&0\ar[r, dashed,  "0"]&0\ar[r, dashed,  "0"]& K\ar[r, dashed,  "0"]& 0 
\end{tikzcd} 
\end{center}

The $A$-superconnection $D_A$ on $E$ descends to a Lie algebroid $A$-connection $\nabla^K$ on~$K$ via the formula
\begin{equation}
    d_A^{\nabla^K}(\overline{e})=\overline{D_A(e)},\label{eqproj}
\end{equation}
where $\overline{e}\in\mathcal{H}^0(E,\partial)=K$ denotes the cohomology class of an element $e\in\Gamma(E_0)$. 
Since the differential in the complex $K$ is the zero differential, the representation up to homotopy on $E$ descends to a Lie algebroid representation on $K$, i.e.  the connection $\nabla^K$ is flat, and there are no higher connection forms. More generally, the extension to $L$ of the superconnection $D_A$ is a $L$-superconnection $D_L$ on $E$ which defines a Lie algebroid connection $\nabla$ of $L$ on $K$, via an identity similar to Equation \eqref{eqproj}:
\begin{equation}
    d_L^{\nabla}(\overline{e})=\overline{D_L(e)}.\label{eqproj2}
\end{equation} By construction, the connection $\nabla$ is an extension of the aforementioned $A$-connection $\nabla^K$. Any other choice of $L$-superconnection $D_L$ defines another $L$-connection $\nabla'$, such that the difference $\nabla-\nabla'$ is a one form on $L$ which sits actually in $A^\circ$.


Recall that $\big(\mathrm{End}_\bullet(E), [\partial,\cdot\,]\big)$ is a chain complex concentrated in degrees $-n,\ldots, n$, of which we denote the cohomology spaces $\mathcal{H}^{i}\big(\mathrm{End}_\bullet(E),[\partial,\cdot\,]\big)$, where the index $i$ ranges from  $-n$ to $n$. In other words:
\begin{equation*}
\mathcal{H}^{i}\big(\mathrm{End}_\bullet(E),[\partial,\cdot\,]\big)=\frac{\mathrm{Ker}\big([\partial,\cdot\,]\colon\mathrm{End}_i(E)\to\mathrm{End}_{i+1}(E)\big)}{\mathrm{Im}\big([\partial,\cdot\,]\colon\mathrm{End}_{i-1}(E)\to\mathrm{End}_{i}(E)\big)},
\end{equation*}
with the understanding that $\mathrm{End}_{-n-1}(E)=0$ and $\mathrm{End}_{n+1}(E)=0$.

\begin{proposition}\label{prop:quasiso}
The cohomology of the chain complex $\big(\mathrm{End}_\bullet(E), [\partial,\cdot\,]\big)$ satisfies
\begin{align}
\mathcal{H}^{0}\big(\mathrm{End}_\bullet(E),[\partial,\cdot\,]\big)&\simeq\mathrm{End}(K),\label{eqlemma1}\\
\mathcal{H}^{i}\big(\mathrm{End}_\bullet(E),[\partial,\cdot\,]\big)&=0\qquad\qquad\text{for every $i\neq0$}.\label{eqlemma2}
\end{align}
\end{proposition}

\begin{proof}
Let us first prove the isomorphism \eqref{eqlemma1}.
Let $f\in\mathrm{End}_0(E)$ be a degree 0 graded vector bundle morphism $f=(f_i:E_i \to E_i)_{i\leq 0}$ closed under  the differential $[\partial, \cdot\,]$, i.e. $[\partial,f]=0$. This  means that  $f$ is a chain map from $(E_\bullet,\partial)$ to itself. Moreover, $f$ is a $[\partial,\cdot\,]$-coboundary, i.e. of the form $f=[\partial,h]$ for some vector bundle morphism $h=(h_i:E_i \to E_{i-1})_{i\leq 0}$, if and only if it is null-homotopic. From this, we deduce that $\mathcal{H}^{0}\big(\mathrm{End}_\bullet(E),[\partial,\cdot\,]\big)$ consists of homotopy equivalent chain maps of $E$ into itself.

Let $\overline{f_i}:\mathcal{H}^i(E_\bullet,\partial)\to \mathcal{H}^i(E_\bullet,\partial)$ be the vector bundle morphisms induced by the  chain map $f:E_\bullet\to E_\bullet$. Since the cohomology of $(E_\bullet,\partial)$ vanishes for every $i\neq0$, the only non-zero induced morphism is $\overline{f_0}:K\to K$. We call $\pi$ the linear map $\mathcal{H}^{0}\big(\mathrm{End}_\bullet(E),[\partial,\cdot\,]\big)\to\mathrm{End}(K),\; f\mapsto \overline{f_0}$. Let us show that it is a bijection.

For every vector bundle morphism $\overline{f}\in\mathrm{End}(K)$, the fundamental theorem of homological algebra \cite[Theorem 2.2.6]{weibelIntroductionHomologicalAlgebra1994} establishes that $f$ can be lifted to a chain map $f:E_\bullet\to E_\bullet$, and that two such choices of lifts are homotopy equivalent. Given the discussion above, this means that  every element of $\mathrm{End}(K)$ has a unique pre-image in $\mathcal{H}^{0}\big(\mathrm{End}_\bullet(E),[\partial,\cdot\,]\big)$, hence $\pi$ is bijective.

Let us turn to the second set of equalities \eqref{eqlemma2}. The chain complex $(E_\bullet,\partial)$ being a resolution of the vector bundle $K$ (of finite length), the fundamental theorem of homological algebra establishes that there exists a section $\sigma:K\to E$ of the chain complex morphism $\phi:E\to K$, together with a vector bundle map $\theta:E_\bullet\to E_{\bullet-1}$, so that $\phi:E\to K$ is a homotopy equivalence between $(E_\bullet, \partial)$ and $(K,0)$, with homotopy inverse $\sigma$ and homotopy operator~$\theta$:
\begin{center}
\begin{tikzcd}[column sep=2cm,row sep=1.5cm]
  \ldots\ar[r,  "\partial" below]&\ar[l,  "\theta" above, bend right]E_{-2}\ar[d,"\phi" left, shift right] \ar[r,  "\partial" below]& \ar[l,  "\theta" above, bend right]E_{-1}\ar[d,"\phi" left, shift right] \ar[r,  "\partial" below]& \ar[l,  "\theta" above, bend right]E_{0}\ar[d,"\phi" left, shift right]\ar[r,  "0" below]&0 \\
 \ldots\ar[r, dashed,  "0"]&0\ar[r, dashed,  "0"]\ar[u,"\sigma" right, shift right]& 0\ar[r, dashed,  "0"]\ar[u,"\sigma" right, shift right]& K\ar[r, dashed,  "0"]\ar[u,"\sigma" right, shift right]&0
\end{tikzcd} 
\end{center}
The vector bundle morphism $\theta$ is a differential 
 and is a contracting homotopy for $\mathrm{id}_E - \sigma\circ\phi$ with respect to the differential $\partial$:
 \begin{equation}
 [\partial,\theta]=\mathrm{id}_E - \sigma\circ\phi.\label{chainhomotop2}
 \end{equation}

 Let us compute $\mathcal{H}^{i}\big(\mathrm{End}_\bullet(E),[\partial,.\,]\big)$ for some $i\neq 0$. Let us set $i\geq 1$, so that $f\in\mathrm{End}_i(E)\simeq\mathrm{Hom}(E,E[i])$ is a vector bundle morphism from $E_\bullet$ to $E_{\bullet+i}$. In particular, $f(E_0)=0$, and necessarily $f\circ \sigma\circ \phi=0$. Assume that $f$ is a $[\partial,\cdot\,]$-cocycle, i.e. such that $[\partial, f]=0$ and set $g_+=(-1)^i f\circ\theta\in \mathrm{End}_{i-1}(E)$. Using the derivation property of the commutator $[\partial,.\,]$ on $g_+$ and Equation \eqref{chainhomotop2}, we then have
\begin{equation*}
    [\partial,g_+]=(-1)^i[\partial,f]\circ \theta+f\circ[\partial, \theta]=f-f\circ \sigma\circ \phi=f.
\end{equation*}
This shows that $f$ is a $[\partial,\cdot\,]$-coboundary, so $\mathcal{H}^{i}\big(\mathrm{End}_\bullet(E),[\partial,\cdot\,]\big)=0$.

With the same $i\geq 1$, a degree $-i$ endomorphism $f\in\mathrm{End}_{-i}(E)\simeq\mathrm{Hom}(E,E[-i])$ is a vector bundle morphism from $E_\bullet$ to $E_{\bullet-i}$. In particular, $E_0\cap\mathrm{Im}(f)=\emptyset$, and necessarily $\phi\circ f=0$. Assume that $f$ is a $[\partial,\cdot\,]$-cocycle, i.e. such that $[\partial, f]=0$ and define $g_-= \theta\circ f\in \mathrm{End}_{-i-1}(E)$. With the derivation property of $[\partial,.\,]$ and Equation \eqref{chainhomotop2}, we have
\begin{equation*}
    [\partial,g_-]=[\partial,\theta]\circ f-\theta\circ[\partial, f]=f- \sigma\circ \phi\circ f=f.
\end{equation*}
This shows that $f$ is a $[\partial,\cdot\,]$-coboundary, so $\mathcal{H}^{-i}\big(\mathrm{End}_\bullet(E),[\partial,\cdot\,]\big)=0$.
To conclude, $\mathcal{H}^{i}\big(\mathrm{End}_\bullet(E),[\partial,\cdot\,]\big)=0$ for every $i\neq 0$, hence proving the second claim. \end{proof}


Recall from Proposition \ref{prophomologys} that $s$ is a differential (with respect to the total degree) on the bigraded vector space $\widehat{\Omega}(E)$. To ease the notation, we henceforth denote the associated cohomology by $\widehat{H}^\bullet(E,s)$. By construction, for $p<0$ and $p>\mathrm{rk}(A)$, we have $\widehat{H}^p(E,s)=0$. A spectral sequence argument allows to compare the Atiyah classes of $E$ and $K$:

\begin{theorem}\label{prop:BRST}
Let $A\hookrightarrow L$ be a Lie pair, and $E$ be a non-positively graded regular homotopy $A$-module. Assume that $\mathcal{H}^\bullet(E,\partial)=\mathcal{H}^0(E,\partial)=:K$. Then, for every $p\in\mathbb{Z}$, 
\begin{equation}\label{isomatiyah}
    \widehat{H}^p(E,s)\simeq H^p\big(A,A^\circ\otimes \mathrm{End}(K)\big).
\end{equation}

\noindent In particular, for $p\leq-1$, $\widehat{H}^p(E,s)=0$. Moreover, the Atiyah class $[\alpha_E]\in \widehat{H}^1(E,s)$ of the homotopy $A$-module $E$ is sent by the isomorphism \eqref{isomatiyah} to the Atiyah class $[\mathrm{at}_K]\in H^1\big(A,A^\circ\otimes \mathrm{End}(K)\big)$ of the Lie algebroid representation $K$.
\end{theorem}

\begin{proof} 
The first page of the spectral sequence associated to the filtered chain complex $(\widehat{\Omega}(E),s)$ is $E^{p,q}_1=\Omega^p\big(A,A^\circ\otimes \mathcal{H}^q(\mathrm{End}_\bullet(E),[\partial,\cdot\,])\big)$,
while the second page is given by $E^{p,q}_2=H^p\big(A,A^\circ\otimes \mathcal{H}^q(\mathrm{End}_\bullet(E),[\partial,\cdot\,])\big)$,  for $0\leq p\leq\mathrm{rk}(A)$. By the set of identities~\eqref{eqlemma2} in Proposition \ref{prop:quasiso}, we have $E^{p,q}_2=0$ for every $q\neq0$,  while by identity \eqref{eqlemma1} we have
\begin{equation*}
     E^{p,0}_2\simeq H^p\big(A,A^\circ\otimes \mathrm{End}(K)\big).
\end{equation*}
 The spectral sequence degenerating -- hence collapsing -- on the second page, 
 the latter is isomorphic to the cohomology of the filtered complex $(\widehat{\Omega}(E),s)$ associated to the given spectral sequence, namely
\begin{equation*}
    \widehat{H}^p(E,s)\simeq H^p\big(A,A^\circ\otimes \mathrm{End}(K)\big).
\end{equation*} 
For all other $p$, i.e. $p<0$ or $p>\mathrm{rk}(A)$, the first page $E^{p,q}_1$ is zero, and so is $\widehat{H}^p(E,s)$.

Regarding the second statement, consider Equation \eqref{isomatiyah} for $p=1$. It establishes that the cohomology class of the Atiyah cocycle $\alpha_E$ associated to the homotopy $A$-module $E$ (element of the left-hand side), is  mapped to a cohomology class taking values in $A^\circ\otimes \mathrm{End}(K)$ (on the right-hand side). Let us show that the latter is the cohomology class of the Atiyah cocycle $\mathrm{at}_K$ associated to the Lie algebroid representation $K$ (for a particular extension to $L$ of the $A$-connection $\nabla^K$).

The Atiyah cocycle $\alpha_E$ decomposes with respect to the form degree as a sum $\alpha_E=\sum_{p=0}^{\mathrm{rk}(A)}\alpha^{(p)}$, where $\alpha^{(p)}\in \Omega^p(A,A^\circ\otimes\mathrm{End}_{1-p}(E))$. The component $\alpha^{(1)}$ takes values in $A^\circ\otimes\mathrm{End}_0(E)$ so, at the level of  $[\partial,\cdot\,]$-cohomology, $\alpha^{(1)}$ induces a $1$-form on $A$ taking values in $A^\circ\otimes\mathrm{End}(K)$, denoted $\overline{\alpha}^{(1)}$ and defined by:
\begin{equation*}
    \overline{\alpha}^{(1)}(a;l)(\overline{e})=\overline{\alpha^{(1)}(a;l)(e)},
\end{equation*}
for any $a\in\Gamma(A)$ and $l\in\Gamma(L/A)$.
The Lie algebroid cohomology class of $\overline{\alpha}^{(1)}$ is the image of the class $[\alpha_E]$ under the isomorphism~\eqref{isomatiyah}.
Then, Equations \eqref{eq:aD} and \eqref{eqproj2} imply that
\begin{equation*}
    \overline{\alpha}^{(1)}(a;l)(\overline{e})
    =\overline{(D_L)^{2}(a,\ell)(e)}
    =(d^\nabla_L)^{2}(a,\ell)(\overline{e})
    =R_\nabla(a;l)(\overline{e}),
\end{equation*}
where $\ell\in\Gamma(L)$ is any preimage of $l$.
The latter term is precisely the Atiyah cocycle associated to the $L$-connection $\nabla$ on $K$, so we have $\overline{\alpha}^{(1)}=\mathrm{at}_K$. From this, we deduce that the Lie algebroid cohomology class of $\overline{\alpha}^{(1)}$ is the Atiyah class $[\mathrm{at}_K]$ of $K$.  The Atiyah class~$[\alpha_E]$ being sent to $[\overline{\alpha}^{(1)}]$ through the isomorphism~\eqref{isomatiyah} proves the second claim. \end{proof} 

\begin{remark}
Equation \eqref{eq:alpha4} and the latter part of the proof of Theorem \ref{prop:BRST} show that the Atiyah cocycle $\alpha_E$ \emph{contains} the Atiyah cocycle $\mathrm{at}_K$ as its first component $\alpha^{(1)}$.  It makes sense graphically (see Figure \ref{figure1}) as this component is the unique component of $\alpha_E$ lying on the column of endomorphism degree zero. This column is the unique column in the bigraded vector space $\widehat{\Omega}(E)$ which does not vanish when passing to the $[\partial,\cdot\,]$-cohomology, then giving  the space of forms $\Omega^{p}(A,A^\circ\otimes\mathrm{End}(K))$, to which belongs $\mathrm{at}_K$.
\end{remark}

Theorem \ref{prop:BRST} also establishes that the other components of $\alpha_E$ do not contain more cohomological information that is not already contained in $\mathrm{at}_K$.

\begin{corollary}\label{cor:equivalence} With the assumptions of Theorem \ref{prop:BRST}, 
\[[\alpha_E]=0\Leftrightarrow [\mathrm{at}_K]=0.\]
\end{corollary}

\section{Examples}\label{sectionExamples}

\subsection{Two-term example I: The double of a vector bundle} Let $A\hookrightarrow L$ be a Lie pair and  $K$ a vector bundle over $M$. Set $E=K[1]\oplus K$, namely: $E_{-1}=K[1],\;E_0=K$; it is a split graded vector bundle of amplitude 2. Take the vector bundle morphism $\partial:K[1]\to K$ connecting the two vector bundles $K[1]$ and $K$ to be the \emph{suspension}\footnote{The \emph{suspension operator} $sp$ on a graded vector space $V=\bigoplus_i V_i$ has the property of shifting the degree of homogeneous elements by $+1$. Namely, $(spV)_i=V_{i-1}$, so that if $x$ has degree $i-1$, then $sp(x)$ has degree~$i$. Obviously, this operator has an inverse, called the \emph{desuspension operator}, which shifts the degree of homogeneous elements by $-1$. The latter is also denoted $[1]$ in the literature.}  of the identity morphism; $\partial=\mathrm{id}_K\circ sp$. Let $\nabla$ be a Lie algebroid $A$-connection on $K$. It then automatically induces an $A$-connection on $K[1]$, also denoted $\nabla$,  and commutes with the operator~$\partial$.  Let $\omega^{(2)}$ be the 2-form on $A$ defined as the \emph{desuspension} of the curvature tensor of $\nabla$; for every two elements $a,b\in\Gamma(A)$, we have $\omega^{(2)}(a,b)=sp^{-1}R_\nabla(a,b):K\to K[1]$.  
The structure of representation up to homotopy on $E$ induced by this connection is then given by the total degree $+1$ operator
\begin{equation*}D_A=\partial+d_A^\nabla+\omega^{(2)}.\end{equation*}

We extend $\nabla$ 
to a Lie algebroid $L$-connection on $E$, also denoted $\nabla$. 
Since $\partial$ is the (suspension of) the identity morphism on $K$, the extended connection -- also denoted $\nabla$ -- commutes with $\partial$. Furthermore, we extend $\omega^{(2)}$ to a 2-form  on $L$ -- denoted $\omega^{(2)}_L$ -- as the desuspension of the curvature tensor of the $L$-connection $\nabla$, so that the restriction of $\omega^{(2)}_L$ to $\wedge^2A$ coincides with $\omega^{(2)}$.
Then, the Atiyah cocycle associated to these choices reads 
\begin{align}
\alpha^{(0)}(l)&=[\nabla_\ell,\partial]=0,\\
\alpha^{(1)}(a;l)&=R_\nabla(a,\ell)+[\partial,\omega^{(2)}_L](a,\ell)=0,\label{eq:session}\\
\alpha^{(2)}(a,b;l)&=d_L^\nabla(\omega^{(2)}_L)(a,b,\ell)=0.
\end{align}
The first and second equations are trivial, the third equation is the Bianchi identity.
Thus, no matter how we extend $\nabla$ to a Lie algebroid $L$-connection on~$K$, we can define 
a homotopy A-compatible superconnection $D_L$ on $E=K[1]\oplus K$. Be aware that this \emph{does not mean} that the Lie algebroid $L$-connection $\nabla$ on $K$ is $A$-compatible in the classical sense \cite{chenAtiyahClassesHomotopy2016}. Indeed, for this to happen, we would have $R_\nabla(a,\ell)=0$ for some choice of extension of $\nabla$, which is a particular case of Equation \eqref{eq:session}. Rather, and in full generality, one observes that the homotopy $A$-module $E$ is homotopy equivalent to the rank zero vector bundle over $M$, which always admits $A$-compatible $L$-connections. Then, using the results of Section \ref{subway}, this explains the triviality of the Atiyah class of the double vector bundle $E=K[1]\oplus K$.

\subsection{Two-term example II: The normal complex \texorpdfstring{$A\to L$}{A -> L}}\label{subnormal}

Let $A\hookrightarrow L$ be a Lie pair over a smooth manifold $M$. This example establishes the relationship between the \emph{normal complex $A\hookrightarrow L$} and the Bott representation $L/A$.
The Bott connection $\nabla^{Bott}$ -- defined in Equation \eqref{equationBott} -- can be extended to an $L$-connection on ${L/A}$ -- also denoted $\nabla^{Bott}$ --  which is not necessarily a Lie algebroid representation.
The Atiyah cocycle associated to this situation is a 1-form on $A$ taking values in $A^\circ\otimes \mathrm{End}(L/A)$ denoted $\mathrm{at}^{Bott}$. 

Let us now fix the notations for the normal complex of $A$. This is a two-term homotopy  $A$-module $(E,\partial)$, 
 with $E_{-1}=A[1], E_0=L$ and where $\partial=\iota\circ sp:A[1]\to L$ is the suspension of the inclusion map $\iota:A\hookrightarrow L$. For clarity, we denote this representation up to homotopy simply by $A\to L$  and we will use $\iota$ instead of $\partial=\iota\circ sp$ whenever it is convenient, keeping in mind that $A$ sits in degree $-1$ when understood from the graded perspective. Under these conventions, it is easy to see that $\mathrm{End}_{-1}(A\to L)\simeq\mathrm{Hom}(L,A)$, that $\mathrm{End}_{0}(A\to L)\simeq\mathrm{End}(A)\oplus\mathrm{End}(L)$, and that $\mathrm{End}_1(A\to L)\simeq\mathrm{Hom}(A,L)$.  Finally, observe that the chain complex $A\to L$ is a resolution of the  normal bundle $L/A$, so that we are meeting the assumptions established in Section \ref{subway}, in the sense that the diagram containing the representation up to homotopy and its cohomology is now
  \begin{center}
\begin{tikzcd}[column sep=2cm,row sep=2cm]
  A\ar[d,"p" left] \ar[r, "\iota"]& L \ar[d, "p" right] \\
 0\ar[r,dashed, "0" ]   & {L/A}
\end{tikzcd} 
\end{center}

Let $\nabla$ be any $L$-connection on the vector bundle $A$. Such a connection always exists since $A$ admits a $TM$-connection $\nabla^{TM}$, and then via the anchor map  $\rho:L\to TM$ it induces a $L$-connection $\nabla$ with e.g. $\nabla_\ell(a)=\nabla^{TM}_{\rho(\ell)}(a)$.  
Inspired by the constructions in \cite{crainicSecondaryCharacteristicClasses2005, abadRepresentationsHomotopyLie2011}, (see also \cite{blaomGeometricStructuresDeformed2006} for the interpretation in terms of \emph{Cartan connections}),
we define the following notions: 

\begin{definition}\label{basconnnormal}
\begin{enumerate}
\item The \emph{basic connection} of $A$ on $L$ is defined by
\begin{equation}\label{eq:basic1}
    \nabla^{bas}_a(\ell)=\nabla_\ell(a)+[a,\ell],
\end{equation}
for every $a\in\Gamma(A)$ and $\ell\in\Gamma(L)$.
\item The \emph{basic curvature} $ R^{bas}\in\Omega^2(A,\mathrm{Hom}(L,A))$  is
\begin{equation*}
    R^{bas}(a,b)(\ell)=\nabla_\ell([a,b])
    -[\nabla_\ell(a),b]-[a,\nabla_\ell(b)]+\nabla_{\nabla^{bas}_{a}(\ell)}(b)-\nabla_{\nabla^{bas}_{b}(\ell)}(a),
\end{equation*}
for every $a,b\in\Gamma(A)$ and $\ell\in\Gamma(L)$.
\end{enumerate}
\end{definition}
In particular, when $\ell$ takes values in $A$, the right-hand side of Equation \eqref{eq:basic1} takes values in $A$ by involutivity. Furthermore, since $\iota$ is the inclusion map, the basic curvature is related to the curvature of the basic connection $R_{\nabla^{bas}}$ by the equation
\begin{equation}\label{eq:bascurv1}
    R_{\nabla^{bas}}+[\partial, R^{bas}]=0.
\end{equation}
We used $\partial=\iota\circ sp$ in order to emphasize the degree $(+1)$ of this operator, which then has consequence on the sign in the graded bracket, namely $[\partial, R^{bas}]=\partial\circ R^{bas}+R^{bas}\circ\partial$.
From Equation \eqref{eq:bascurv1}, one can see that the basic curvature plays the role of the connection 2-form $\omega^{(2)}$. Because of the relationship \eqref{eq:bascurv1} between the two differential forms, the basic curvature satisfies the Bianchi-like identity
\begin{equation*}
    d^{\nabla^{bas}}_AR^{bas}=0,
\end{equation*}
and the degree $+1$ operator $D_A:\Omega(A,A\to L)_{\blacktriangle}\to \Omega(A,A\to L)_{\blacktriangle+1}$ corresponding to this representation up to homotopy reads
\begin{equation*}
    D_A=\partial+d^{\nabla^{bas}}_A+ R^{bas}\wedge\cdot .
\end{equation*}

Recall that we have used an $L$-connection $\nabla$ on $A$ in order to define the basic connection. Then the quotient map $p:L\to{L/A}$  satisfies $p(\nabla_\ell(a))=0$, so that we have, using Equation \eqref{equationBott},
\[p(\nabla^{bas}_a(\ell))=p(\nabla_\ell(a))+p([a,\ell])=\nabla_a^{Bott}(p(\ell)),\]
i.e. the basic and Bott connections are related by 
\begin{equation}\label{commutconnections}
p\circ\nabla^{bas}_a=\nabla^{Bott}_a\circ p.
\end{equation}
From this identity, we infer that the relationship between the curvatures~$R_{\nabla^{bas}}$ 
and~$R_{\nabla^{Bott}}$ associated to the basic and the Bott connections, respectively, is  
\begin{equation}\label{eq:commute}
    p\circ R_{\nabla^{bas}}(a,b)=R_{\nabla^{Bott}}(a,b)\circ p=0,
\end{equation}
as the Bott connection is flat. Then, by definition of the basic curvature, we have:
\begin{equation}\label{zerocurv}
p\circ R^{bas}(a,b)=0.
\end{equation}

From Equations \eqref{commutconnections} and \eqref{zerocurv}, we deduce that the quotient map $p$ induces a degree preserving morphism of representations up to homotopy, also denoted  $p$, making the following diagram commutative:
    \begin{center}
\begin{tikzcd}[column sep=2cm,row sep=2cm]
  \Omega(A,A\to L)_\blacktriangle\ar[d,"p" left] \ar[r, "D_A"]& \Omega(A,A\to L)_{\blacktriangle+1} \ar[d, "p", right] \\
   \Omega^{\bullet}(A,{L/A}) \ar[r, "d^{\nabla^{Bott}}"]& \Omega^{\bullet+1}(A,{L/A}) 
\end{tikzcd} 
\end{center}
The degree symbol in the upper row refers to the total degree of the bigraded complex, while the degree symbol in the lower row refers to the form degree, matching the total degree.

\begin{remark}\label{extend}
It is always possible to extend the basic connection to an $L$-connection~$\nabla^L$ on $L$ which stabilizes $A$,  
and to extend the basic curvature $R^{bas}$ to a 2-form $\omega^{(2)}_L\in\Omega^{2}(L,\mathrm{Hom}(L,A))$. 
This corresponds to extending the $A$-superconnection on $A\to L$ to an $L$-superconnection $D_L$. Furthermore, let  $\nabla^{Bott}$ denote the extension of the Bott connection to an $L$-connection on $L/A$. The  commutative diagram above then stays commutative whenever we extend the differential forms on $A$ to forms on $L$, and change $D_A$ by $D_L$. 
\end{remark}

Let us now compute the Atiyah cocycle in this concrete example.
For $k=0$, Equation~\eqref{eq:alpha0} gives
\begin{equation}\label{eq:alpha0tri}
\alpha^{(0)}(l)=[\nabla^L_\ell,\partial]=0,
\end{equation}
where recall that $\partial=\iota\circ sp:A[1]\to L$ and that the bracket is the graded commutator on $\mathrm{End}(E)$ as defined in Remark \ref{remark1}. 
The second identity of \eqref{eq:alpha0tri} holds as $\iota$ being the inclusion, we always have $\nabla^{L}\circ \iota=\iota\circ\nabla^{L}$. 
The two other components of the Atiyah cocycle are $\alpha^{(1)}\in\Omega^1(A,A^\circ\otimes (\mathrm{End}(A)\oplus\mathrm{End}( L)))$ and $\alpha^{(2)}\in\Omega^2(A,A^\circ\otimes \mathrm{Hom}(L,A))$, which are obtained from Equation \eqref{eq:alpha4} for $k=1$ and $k=2$:  
\begin{align*}
\alpha^{(1)}(a;l)&=R_{\nabla^L}(a,\ell)
+[\partial, \omega^{(2)}_L](a,\ell),\\
\alpha^{(2)}(a,b;l)&=d_L^{\nabla^L}\omega^{(2)}_L(a,b,\ell).
\end{align*}

\begin{remark}
One can show by direct computation that $\alpha=\alpha^{(0)}+\alpha^{(1)}+\alpha^{(2)}$ is $s$-closed.
\end{remark}

The cohomology class of $\alpha$ stays the same if one substracts from $\alpha$ a coboundary $s(\phi)$, for some $\phi=\phi^{(0)}+\phi^{(1)}$ with $\phi^{(0)}\in \Gamma(A^\circ\otimes  (\mathrm{End}(A)\oplus\mathrm{End}( L)))$ and $\phi^{(1)}\in\Omega^1(A,A^\circ\otimes \mathrm{Hom}(L,A))$ (see Figure \ref{figure2}). Then, the following Proposition establishes that the Atiyah class $[\alpha]$ of the normal complex $A\to L$ does not contain more nor less information than the Atiyah class $[\mathrm{at}^{Bott}]$ of the Bott representation $L/A$.

\begin{figure}[ht]
  \centering
  \begin{tikzpicture}[scale=1]
    \coordinate (Origin)   at (0,0);
    \coordinate (XAxisMin) at (1,0);
    \coordinate (XAxisMax) at (11,0);
    \coordinate (YAxisMin) at (6,0);
    \coordinate (YAxisMax) at (6,8);
    \draw [ultra thick, black,-latex] (XAxisMin) -- (XAxisMax) node [right] {\begin{tabular}[c]{@{}c@{}}endomorphism\\  degree\end{tabular}};
    \draw [ultra thick, black,-latex] (YAxisMin) -- (YAxisMax) node [above] {form degree};

    \clip (-1,-1) rectangle (12cm,7cm); 
    \coordinate (Bone) at (5,5);
    \coordinate (Btwo) at (1,7);
    \coordinate (B2) at (2,1);
    \coordinate (B3) at (2,3);
    \coordinate (B4) at (4,1);
    \coordinate (B5) at (0,5);
    \draw[style=help lines,dashed] (4,1) grid[step=2cm] (5,8);
    \draw[style=help lines,dashed] (7,1) grid[step=2cm] (8,8);
    \draw[style=help lines,dashed] (4,1) grid[step=2cm] (4,0);
    \draw[style=help lines,dashed] (8,1) grid[step=2cm] (8,0);
    \draw[style=help lines,dashed] (5,2) grid[step=2cm] (7,2);
    \draw[style=help lines,dashed] (5,4) grid[step=2cm] (7,4);
    \draw[style=help lines,dashed] (5,6) grid[step=2cm] (7,6);
    \foreach \x in {2,...,4}{
      \foreach \y in {1,2,...,5}{
        \node[draw,circle,inner sep=1pt,fill] at (2*\x,2*\y) {};
         \node[draw,circle,inner sep=1pt,fill,blue] at (4,4){};
          \node[draw,circle,inner sep=1pt,fill,blue] at (6,2){};
        \node[draw,circle,inner sep=1pt,fill,red] at (4,2){};
          \node[draw,circle,inner sep=1pt,fill,red] at (6,0){};
       \node[draw,circle,inner sep=1pt,fill,DarkGreen] at (B2){};
      }
    }
    \node [below] at (6,0)  {$0$};
    \node [below] at (4,0)  {$-1$};
    \node [below] at (8,0)  {$1$};
    \node [below left] at (6,2)  {$1$};
    \node [below left] at (6,4)  {$2$};
    \node [below left] at (6,6)  {$3$};

    \draw [ultra thick,DarkGreen, -latex] (B2) -- (B4);
    \draw [ultra thick,DarkGreen, -latex] (B2) -- (B3);
    \draw [ultra thick,DarkGreen, -latex] (B2) -- (B5);
    \node [below,DarkGreen] at (3,1) {\large $[\iota,.\,]$};
    \node [right,DarkGreen] at (2,2) {\large $d^{\widehat{\nabla}}_A$};
    \node [below left,DarkGreen] at (2,5) {\large $[ R^{bas},.\,]$};
    \draw [ultra thick,blue] (4,4) -- (6,2);
    \draw [ultra thick,red] (4,2) -- (6,0);
    \node [above left,blue] at (4,4) {\large $\alpha^{(2)}$};
    \node [above right,blue] at (6,2) {\large $\alpha^{(1)}$};
      \node [above left,red] at (4,2) {\large $\phi^{(1)}$};
    \node [above right,red] at (6,0) {\large $\phi^{(0)}$};
  \end{tikzpicture}
\caption{Representation of the bigraded space $\Omega^{\bullet}(A,A^\circ\otimes \mathrm{End}_\bullet(A\to L))$ and the Atiyah cocycle $\alpha$. Since the representation up to homotopy $A\to L$ is concentrated in degrees $-1,0$, the width of the diagram is 3, while the height corresponds to the rank of $A$.}   \label{figure2}
\end{figure}
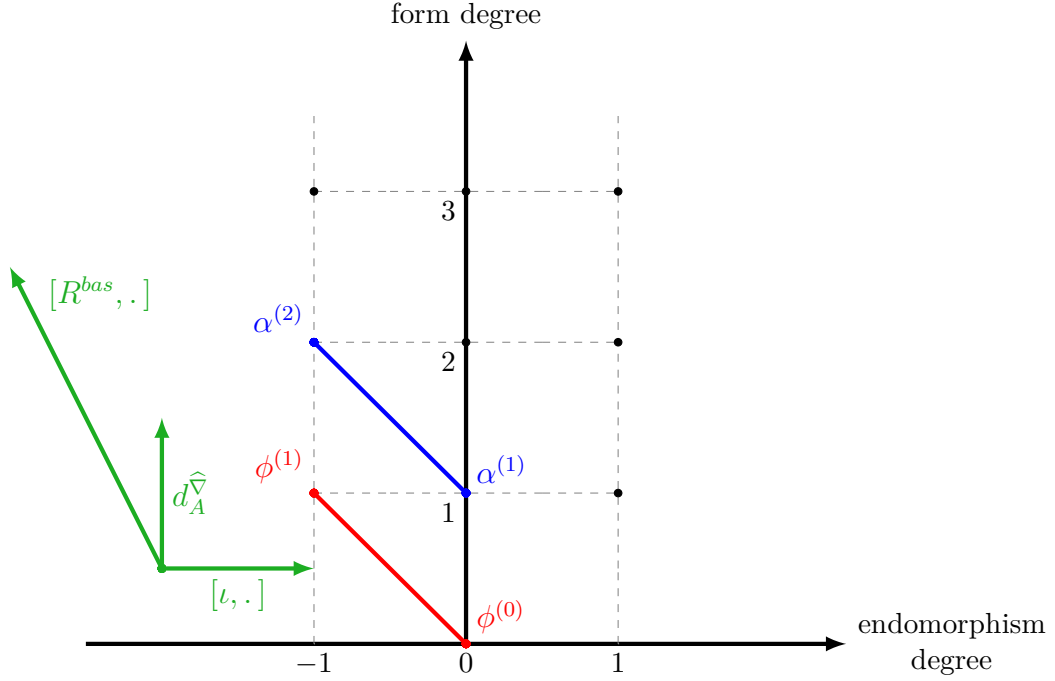

\begin{proposition}\label{propo}
$[\alpha]=0$ if and only if $[\mathrm{at}^{Bott}]=0$.
\end{proposition}

\begin{proof}
$\Longrightarrow$
The Atiyah cocycle $\alpha$ being $s$-exact means that there exists $\phi=\phi^{(0)}+\phi^{(1)}$ of total degree 0 such that $\alpha=s(\phi)$. The element  $\phi^{(0)}\in \Gamma(A^\circ\otimes  (\mathrm{End}(A)\oplus\mathrm{End}( L)))$ can be understood as a 1-form on $L$ taking values in $\mathrm{End}(A)\oplus\mathrm{End}( L)$ and denoted $\varphi^{(1)}$, 
while the element  $\phi^{(1)}\in\Omega^1(A,A^\circ\otimes \mathrm{Hom}(L,A))$ can be extended to a 2-form on $L$ taking values in $\mathrm{Hom}(L,A)$ and denoted $\varphi^{(2)}$. 
The elements $\varphi^{(1)}$ and $\varphi^{(2)}$ satisfy, for every $a\in\Gamma(A)$ and $\ell\in\Gamma(L)$ projecting down to $l\in\Gamma(L/A)$:
\begin{equation*}
\varphi^{(1)}(\ell)=\phi^{(0)}(l)\qquad\text{and}\qquad\varphi^{(2)}(a,\ell)=\phi^{(1)}(a;l).
\end{equation*}

The newly defined connection $\nabla'^L=\nabla^L-\varphi^{(1)}$  shifted by the connection form $\varphi^{(1)}$, and the 2-form $\omega'^{(2)}_L=\omega^{(2)}_L-\varphi^{(2)}$, coincide with the basic connection and basic curvature, when restricted  to $A$. 
Thus, the identity $\alpha=s(\phi)$ is equivalent to the following three identities:
\begin{align}
[\iota,\varphi^{(1)}(\ell)]&=0,\label{eq:beta4}\\
R_{\nabla'^L}(a,\ell)+[\partial,\omega'^{(2)}_L](a,\ell)&=0,\label{eq:beta5}\\
d_L^{\nabla'^L}\omega'^{(2)}_L(a,b,\ell)&=0,\label{eq:beta6}
\end{align}
for any $a,b\in\Gamma(A)$ and $\ell\in\Gamma(L)$.


Note that Equation \eqref{eq:beta4} is not void. As an element of $\Gamma(\mathrm{End}_0(A\to L))$, $\varphi^{(1)}(\ell)$ is actually the sum of two vector bundle endomorphisms,  $\varphi^{(1),A}(\ell):A\to A$ and $\varphi^{(1),L}(\ell):L\to L$.  
 Then, Equation~\eqref{eq:beta4} establishes that the former is the restriction to $A$ of the latter,  i.e. that $\varphi^{(1),L}(\ell)\big|_A=\varphi^{(1),A}(\ell)$.
Thus, substracted to the connection $\nabla^L$, 
the connection form $\varphi^{(1)}$ gives rises to an $L$-connection $\nabla'^L$ which, by Equation~\eqref{eq:beta4}, commutes with $\iota$:
\begin{equation*}
    [\nabla'^L_\ell,\iota]=0.
\end{equation*}
This means that this newly defined $L$-connection passes to the quotient, and then induces an alternative well-defined extension $\nabla'^{Bott}$ of the Bott connection on the normal bundle:
\begin{equation*}
    p\circ\nabla'^L=\nabla'^{Bott}\circ p.
\end{equation*}

From Equation \eqref{eq:commute} we deduce that the induced curvature $R_{\nabla'^L}$ is sent to the curvature of the newly defined Bott connection:
\begin{equation}\label{eq:commute2}
    p\circ R_{\nabla'^L}(\ell_1,\ell_2)=R_{\nabla'^{Bott}}(\ell_1,\ell_2)\circ p.
\end{equation}
The right-hand side is not necessarily zero because here we have the version of the Bott connection extended to $L$, which is not necessarily flat outside $A$.
However,  whenever Equation \eqref{eq:commute2} is restricted to $A\otimes L$, its left-hand side is the projection through $p$ of the left-hand side of Equation \eqref{eq:beta5}. Indeed, since the term $[\iota,\omega'^{(2)}_L(a,\ell)]$ in Equation  \eqref{eq:beta5} always lands in $A$, its projection under $p$ vanishes, and we are left with the term $p\circ R_{\nabla'^L}(a,\ell)$. The right-hand side of Equation \eqref{eq:beta5} being zero, we have that $p\circ R_{\nabla'^L}(a,\ell)=0$, implying in turn that $R_{\nabla'^L}(a,\ell)=0$.

Thus, conditions \eqref{eq:beta4} and \eqref{eq:beta5} say that the Atiyah cocycle associated to the
 newly defined extension of the Bott connection $\nabla'^{Bott}$ is zero.
Thus Equation \eqref{eq:beta5} does not contain more information than the classical Atiyah class; it just brings consistence with the representation up to homotopy $A\to L$ and is necessary to induce a modified Bott connection that is projectable.

$\Longleftarrow$ Conversely, 
first choose a splitting $\sigma:L/A\to L$ of the exact sequence
\begin{center}
\begin{tikzcd}[column sep=1.5cm,row sep=1.5cm]
  0\ar[r]&A \ar[r,  "\iota"]& L \ar[r,  "p"]& \ar[l,  "\sigma", bend left]L/A\ar[r]&0.
\end{tikzcd} 
\end{center}
Set $B$ to be the image of this splitting, so that, as vector bundle, $L$ is a direct sum $L=A\oplus B$. Under this splitting, the morphism $\partial=\iota\colon A\to L$ (we omit to write the suspension) reads, in a block matrix form,
\begin{equation*}
    \partial=\begin{pmatrix}
        \iota\\ 0
    \end{pmatrix}.
\end{equation*}

 The fact that the Atiyah cocycle $\mathrm{at}^{Bott}$ is exact, means that there exists an element $\eta\in\Omega^1(L,\mathrm{End}(L/A))$ vanishing on $A$ such that the modified extension of the Bott connection $\nabla'^{Bott}=\nabla^{Bott}-\eta$ satisfies
 \begin{equation}\label{eq:cara}
 R_{\nabla'^{Bott}}(a,\ell)=0.
 \end{equation}
By the splitting, the element $\eta$ induces an endomorphism of the vector bundle $B$. Using block matrix notation with respect to the decomposition $L=A\oplus B$, one can set $\varphi^{(1)}\in\Omega^{1}(L,\mathrm{End}(A)\oplus\mathrm{End}(L))$ to be of the following block matrix form:
\begin{equation}\label{eqmatrix11}
    \varphi^{(1),L}=\begin{pmatrix}
\chi& 0\\ 
0 &  \eta
\end{pmatrix} \qquad \text{and}\qquad  \varphi^{(1),A}=\chi,
\end{equation}
where $\chi$ 
 is any section of $A^\circ\subset L^*$ taking values in $\mathrm{End}(A)$. This satisfies  Equation \eqref{eq:beta4}, and
 in particular one can set $\chi=0$.  

Modify the connection $\nabla^L$ accordingly, by substracting the connection form~$\varphi^{(1)}$ from~$\nabla^L$; this gives $\nabla'^L$. By construction of $\varphi^{(1)}$, its curvature 
  is of block matrix form:
\begin{equation}\label{eqmatrix1}
    R_{\nabla'^L}=\begin{pmatrix}
R_{\nabla'^L}^{A\to A}& R_{\nabla'^L}^{B\to A}\\
0 &  R_{\nabla'^L}^{B\to B}\
\end{pmatrix}.
\end{equation}
The lower left block corresponds to $R_{\nabla'^L}^{A\to B}$ and is zero because $A$ is stable by $\nabla^L$, and hence by $\nabla'^L$.
The lower right block, when evaluated on two elements $a\in\Gamma(A)$ and $\ell\in\Gamma(L)$, is zero precisely because because of Equation~\eqref{eq:cara} and of our choice of $\varphi^{(1)}$ which involves $\eta$ in the lower-right corner of the matrix of  Equation \eqref{eqmatrix11}. 

On the other hand, the block matrix form of $\omega^{(2)}_L$ is made of only the upper line because it takes an element in $L=A\oplus B$ and sends it in $A$. Then  define $\varphi^{(2)}$ as the 2-block line matrix:
\begin{equation}\label{eqmatrix2}
    \varphi^{(2)}=\begin{pmatrix}
R_{\nabla'^L}^{A\to A}+\omega^{(2),A\to A}_L& R_{\nabla'^L}^{B\to A}+\omega^{(2),B\to A}_L\
\end{pmatrix}.
\end{equation}
This is a two form on $L$ taking values in $\mathrm{Hom}(L,A)$ and vanishing whenever restricted to~$A$.
By construction, the above data satisfy the following equation:
\begin{equation*}
    R_{\nabla'^L}(a,\ell)+[\partial,\omega^{(2)}_L](a,\ell)=[\partial,\varphi^{(2)}](a,\ell).
\end{equation*}
Setting $\omega'^{(2)}_L=\omega^{(2)}_L-\varphi^{(2)}$, we obtain Equation \eqref{eq:beta5}.

Now we have to show that  Equation \eqref{eq:beta6} is automatically satisfied.
Since $R_{\nabla'^L}$ satisfies the Bianchi identity $d_L^{\nabla'^L}R_{\nabla'^L}=0$,  Equation \eqref{eq:beta5} -- whose validity we have already shown -- implies that the action of the covariant derivative $d_L^{\nabla'^L}$ on the bracket $[\iota, \omega'^{(2)}_L]$ is trivial, whenever evaluated on a triple $(a,b,\ell)$, for $a,b\in\Gamma(A)$ and $\ell\in\Gamma(L)$. Because $\iota$ is the inclusion map, this necessary condition is precisely the content of Equation \eqref{eq:beta6}. Hence this identity is automatically satisfied and no extra information is contained there.

Since $\varphi^{(1)}$ and $\varphi^{(2)}$ identically vanish when restricted to $A$, they canonically induce elements  $\phi^{(0)}\in \Gamma(A^\circ\otimes  (\mathrm{End}(A)\oplus\mathrm{End}( L)))$ and $\phi^{(1)}\in\Omega^1(A,A^\circ\otimes \mathrm{Hom}(L,A))$, respectively. The construction has been made so that their sum  $\phi=\phi^{(0)}+\phi^{(1)}$ satisfies Equations \eqref{eq:beta4}-\eqref{eq:beta6} or, equivalently,  $\alpha=s(\phi)$. Notice that any other choice of   splitting $\sigma:L/A\to L$ does not change this result. \end{proof}

\begin{remark}
When $A=F, L=TM$ is the Lie pair associated to an involutive subbundle corresponding to a regular foliation by Frobenius' Theorem, the Atiyah class $\mathrm{at}^{Bott}$ recovers the \emph{Molino class} \cite{chenAtiyahClassesHomotopy2016}. Denote $\mathcal{L}$ to be the leaf space of this foliation. This topological space is not necessarily a smooth manifold. The Bott connection is then a prescription of how to identify fibers of the normal bundle $\nu(F)=TM/F$ over each leaf. The topological space of equivalence classes of elements of $\nu(F)$ with respect to the Bott connection is a fiber bundle over the leaf space -- denoted $\nu_{\mathcal{L}}$, which can be understood as the `tangent bundle' of $\mathcal{L}$ since, whenever $\mathcal{L}$ is a smooth manifold, the former is  canonically isomorphic to the latter.

Now, assume that the leaf space is a smooth manifold, and thus admits a tangent bundle.
The Molino class of the Bott connection being zero then means that the (extended) Bott $TM$-connection $\nabla$ on the normal bundle $\nu(F)$ is `projectable' to a $T\mathcal{L}$-connection on $\nu_{\mathcal{L}} \simeq T\mathcal{L}$ (not necessarily a representation).
If  the regular foliation is \emph{stricly simple} -- i.e. if the leaf space $\mathcal{L}$ is smooth and comes from a submersion $M\to \mathcal{L}$ with connected fibers -- then every $T\mathcal{L}$-connection on $T\mathcal{L}\simeq\nu_{\mathcal{L}}$ can be lifted to a $TM$-connection on $\nu(F)$ extending the Bott connection. In other words, it is always possible to find a projectable connection so in that case the Molino class is always zero. 
\end{remark}

\subsection{Three-term example: The adjoint complex of a regular Lie algebroid}\label{start}

Let $A\hookrightarrow L$ be a Lie pair, 
 and assume that the restriction to $A$ of the anchor map $\rho:L\to TM$ has constant rank, i.e.  that the map $M\to \mathbb{N},x\mapsto\mathrm{dim}\big(\rho_x(A_x)\big)$ is constant. 
By suspending the anchor map, this implies that the \emph{adjoint complex} $A[1]\overset{\rho\circ sp}{\to}TM$ is a regular representation up to homotopy of $A$   \cite{abadRepresentationsHomotopyLie2011}. This homotopy $A$-module structure is induced from  a choice of a $TM$-connection $\nabla$ on $A$, which in turn defines a \emph{basic connection} \cite{crainicSecondaryCharacteristicClasses2005, abadRepresentationsHomotopyLie2011} by the equations 
\begin{equation*}
    \nabla^{bas}_a(b)=\nabla_{\rho(b)}(a)+[a,b]\;\; \text{and}\;\;\nabla^{bas}_a(X)=\rho(\nabla_{X}(a))+[\rho(a),X],
\end{equation*}
for $a,b\in\Gamma(A)$, $X\in\mathfrak{X}(M)$. 
The \emph{basic curvature} associated to $\nabla^{bas}$ is given by the following tensor:
\begin{equation*}
R^{bas}(a,b)(X)=\nabla_X([a,b])-[\nabla_X(a),b]-[a,\nabla_X(b)]+\nabla_{\nabla^{bas}_{a}(X)}(b)-\nabla_{\nabla^{bas}_{b}(X)}(a).
\end{equation*}
 The homotopy $A$-module structure on the adjoint complex $A\to TM$ is then given by the equations
\begin{align}
\nabla^{bas}\circ\rho&=\rho\circ\nabla^{bas},\label{eqchat1}\\
R_{\nabla^{bas}}+[\rho\circ sp,R^{bas}]&=0,\label{eqchat2}\\
d_A^{\nabla^{bas}}R^{bas}&=0.\label{eqchat3}
\end{align}

Since the anchor map has constant rank on $A$, its kernel $\mathrm{Ker}(\rho|_A)$ -- denoted $\mathfrak{g}(A)$ --  is a Lie algebroid representation of $A$ induced by the flat connection
\begin{equation*}
\nabla^{\mathfrak{g}}_a(b)=[a,b],
\end{equation*}
for $a\in\Gamma(A)$,  $b\in\Gamma(\mathfrak{g}(A))$.
Denoting $\iota:\mathfrak{g}(A)\to A$ the inclusion of the kernel of the anchor map into $A$, one immediately observes that $\iota(\nabla^{\mathfrak{g}}_a(b))=\nabla^{bas}_a(\iota(b))$. That is to say, the connection $\nabla^{\mathfrak{g}}$ is the restriction of the basic connection to $\mathfrak{g}(A)$. Moreover, restricting Equation \eqref{eqchat2} to $\mathfrak{g}(A)$, one obtains
\begin{equation*}
0=R_{\nabla^{bas}}(a,b)(\iota(c))+R^{bas}(a,b)(\rho(c))=R_{\nabla^{bas}}(a,b)(\iota(c)),
\end{equation*}
for any $a,b\in\Gamma(A)$ and $c\in\Gamma(\mathfrak{g}(A))$.
This is the flatness condition for $\nabla^{\mathfrak{g}}$.

The graded vector bundle \[E=\mathfrak{g}(A)[2]\oplus A[1]\oplus TM\] then becomes a regular representation up to homotopy of $A$, with differential $\partial\big|_{E_{-2}}=\iota\circ sp:\mathfrak{g}(A)[2]\to A[1]$ and $\partial\big|_{E_{-1}}=\rho\circ sp:A[1]\to TM$. It is controlled by Equations~\eqref{eqchat1}-\eqref{eqchat3}, as the equations
\begin{align}
\nabla^{\mathfrak{g}}\circ\iota&=\iota\circ\nabla^{bas},\label{eqchat4}\\
R_{\nabla^{\mathfrak{g}}}&=0,\label{eqchat5}
\end{align}
are automatically satisfied by definition of the basic connection. 
Notice however that, from the graded perspective, the curvature of the connection $\nabla^{\mathfrak{g}}$ being zero, there is no need to introduce any connection 2-form taking values in $\mathrm{Hom}(A,\mathfrak{g}(A))$ and the only connection 2-form is the basic curvature.

In the following we denote by $\nabla^E=\nabla^\mathfrak{g}+\nabla^{bas}$ the $A$-connection on $E$ which restricts to $\nabla^{\mathfrak{g}}$ on $\mathfrak{g}(A)[2]$ and to the basic connection on $A[1]\oplus TM$. The associated curvature is $R_{\nabla^E}=R_{\nabla^\mathfrak{g}}+R_{\nabla^{bas}}$, while the basic connection will now be denoted as $\omega^{(2)}$. Hence, Equations \eqref{eqchat1}-
\eqref{eqchat5} can be grouped into the equations
\begin{align}
[ d^{\nabla^E}_A,\partial]&=0,\\
R_{\nabla^E}+[\partial,\omega^{(2)}]&=0,\label{eqbegaud}\\
d_A^{\nabla^E}\omega^{(2)}&=0.\label{eqbegaudo}
\end{align} 
In particular, in this situation, we do not have a connection 3-form $\omega^{(3)}$.

Now let us extend the connection $\nabla^E$ and the connection 2-form $\omega^{(2)}$, to an $L$-connection on $E$ -- still denoted $\nabla^E$ -- and a 2-form on $L$ denoted $\omega^{(2)}_L$ and taking values in $\mathrm{End}_{-1}(E)$. It splits into two components which satisfies, when evaluated on two elements $a,b\in\Gamma(A)$:
\begin{equation*}
\omega_L^{(2),A\to \mathfrak{g}(A)}(a,b)=0\qquad\text{and}\qquad\omega_L^{(2),TM\to A}(a,b)=R^{bas}.
\end{equation*}
These results are obtained from Equations \eqref{eqchat2} and \eqref{eqchat5}.
In particular, one can always choose to extend $\omega^{(2)}$ to $L$ so that its $\mathrm{Hom}(A,\mathfrak{g}(A))$ component is identically zero, namely
\begin{equation}\label{eq:always}
\omega_L^{(2),A\to \mathfrak{g}(A)}(\ell,\ell')=0,
\end{equation}
for every $\ell,\ell'\in\Gamma(L)$.

 Then, the Atiyah cocycle of the representation up to homotopy $E$ is given by Equations~\eqref{eq:alpha0} and \eqref{eq:alpha4} which in the present case take the form
\begin{align}
    \alpha^{(0)}_E(l)&=[\nabla^E_\ell,\partial],\label{eqpizz1}\\
    \alpha^{(1)}_E(a;l)&=R_{\nabla^E}(a,\ell)+[\partial,\omega^{(2)}_L](a,\ell),\label{eqpizz2}\\
    \alpha^{(2)}_E(a,b;l)&=d_L^{\nabla^E} \omega^{(2)}_L(a,b,\ell),\label{eqpizz3}\\
    \alpha^{(3)}_E(a,b,c;l)&=0, \label{eqpizz4} 
\end{align}
for every $a,b,c\in\Gamma(A)$ and $l\in\Gamma(L/A)$ with preimage $\ell\in\Gamma(L)$. 
Notice that the element $\alpha^{(0)}_E(l)\in\Gamma(\mathrm{End}_{1}(E))$  splits into two components, one taking values in $\mathrm{Hom}(\mathfrak{g}(A),A)$ and denoted $\alpha^{(0), \mathfrak{g}(A)\to A}_{E}(l)$, while the second one takes values in $\mathrm{Hom}(A,TM)$ and is denoted $\alpha^{(0),A\to TM}_{E}(l)$. Since the inclusion map intertwines $\nabla^{\mathfrak{g}}$ and $\nabla^{bas}$ by \eqref{eqchat4}, we have that $\alpha^{(0),\mathfrak{g}(A)\to A}_{E}(l)=0$ while the term $\alpha^{(0),A\to TM}_{E}(l)$ does not necessarily vanish.

Furthermore, since $\mathrm{End}_0(E)\simeq\mathrm{End}(\mathfrak{g}(A))\oplus \mathrm{End}(A)\oplus \mathrm{End}(TM)$,  the first component $\alpha_E^{(1)}(a;l)\in\Gamma(\mathrm{End}_0(E))$ can be split into three terms, each one of them acting as an endomorphism of either $\mathfrak{g}(A)$, $A$ or $TM$. The component of $\alpha_E^{(1)}$ taking values in $\mathrm{End}(\mathfrak{g}(A))$ is denoted $\mathrm{at}^{\mathfrak{g}}\in\Omega^{1}(A,A^\circ\otimes\mathrm{End}(\mathfrak{g}(A)))$.
By the choice made in Equation \eqref{eq:always},  the second term $[\partial,\omega^{(2)}_L]$ on the right-hand side of Equation \eqref{eqpizz2} does not contribute to $\mathrm{at}^{\mathfrak{g}}$, and so~$\mathrm{at}^{\mathfrak{g}}$ is precisely the Atiyah cocycle of the $A$-module $\mathfrak{g}(A)$. Next, the term $\alpha_E^{(2)}(a,b;l)\in\Gamma(\mathrm{End}_{-1}(E))$ can be decomposed into two terms, one taking values in $\mathrm{Hom}(TM,A)$, and the other one taking values in $\mathrm{Hom}(A,\mathfrak{g}(A))$. The former   is $d^{\nabla^{bas}}_LR^{bas}(a,b,\ell)$ while the latter is identically zero, due to the choice made in Equation~\eqref{eq:always}. Eventually, in full generality, since there is no connection 3-form in the present situation the last term of the Atiyah cocycle should read:
\begin{equation}\label{eqprince}
\alpha^{(3)}_E(a,b,c;l)=\omega_L^{(2),A\to\mathfrak{g}(A)}\wedge \omega_L^{(2),TM\to A}(a,b,c,\ell),
\end{equation}
but from Equation \eqref{eq:always}, it is identically zero. 

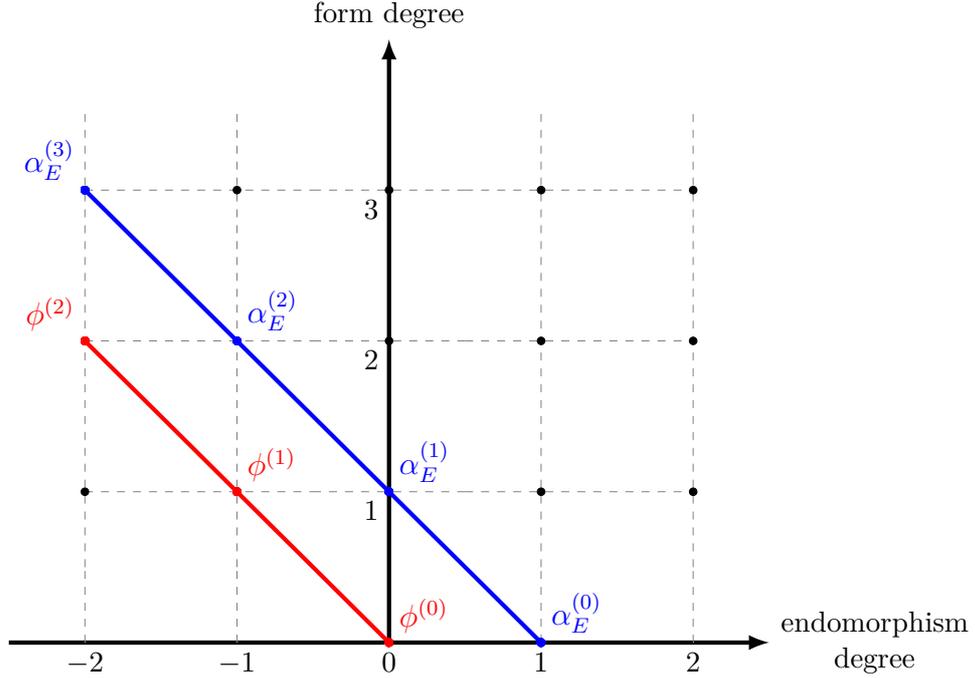
\begin{figure}[ht]
  \centering
  \begin{tikzpicture}[scale=1]
    \coordinate (Origin)   at (0,0);
    \coordinate (XAxisMin) at (1,0);
    \coordinate (XAxisMax) at (11,0);
    \coordinate (YAxisMin) at (6,0);
    \coordinate (YAxisMax) at (6,8);
    \draw [ultra thick, black,-latex] (XAxisMin) -- (XAxisMax) node [right] {\begin{tabular}[c]{@{}c@{}}endomorphism\\  degree\end{tabular}};
    \draw [ultra thick, black,-latex] (YAxisMin) -- (YAxisMax) node [above] {form degree};

    \clip (-1,-1) rectangle (12cm,7cm); 
    \coordinate (Bone) at (5,5);
    \coordinate (Btwo) at (1,7);
    \coordinate (B2) at (2,1);
    \coordinate (B3) at (2,3);
    \coordinate (B4) at (4,1);
    \coordinate (B5) at (0,5);
    \draw[style=help lines,dashed] (2,1) grid[step=2cm] (5,8);
    \draw[style=help lines,dashed] (7,1) grid[step=2cm] (10,8);
    \draw[style=help lines,dashed] (2,0) grid[step=2cm] (2,1);
    \draw[style=help lines,dashed] (4,0) grid[step=2cm] (4,1);
    \draw[style=help lines,dashed] (8,0) grid[step=2cm] (8,1);
    \draw[style=help lines,dashed] (10,0) grid[step=2cm] (10,1);
    \draw[style=help lines,dashed] (5,2) grid[step=2cm] (7,2);
    \draw[style=help lines,dashed] (5,4) grid[step=2cm] (7,4);
    \draw[style=help lines,dashed] (5,6) grid[step=2cm] (7,6);
    \draw[style=help lines,dashed] (5,8) grid[step=2cm] (7,8);
    
    \foreach \x in {1,...,5}{
      \foreach \y in {1,2,...,5}{
        \node[draw,circle,inner sep=1pt,fill] at (2*\x,2*\y) {};
        \node[draw,circle,inner sep=1pt,fill,blue] at (2,6){};
         \node[draw,circle,inner sep=1pt,fill,blue] at (4,4){};
          \node[draw,circle,inner sep=1pt,fill,blue] at (6,2){};
          \node[draw,circle,inner sep=1pt,fill,blue] at (8,0){};
          \node[draw,circle,inner sep=1pt,fill,red] at (2,4){};
        \node[draw,circle,inner sep=1pt,fill,red] at (4,2){};
          \node[draw,circle,inner sep=1pt,fill,red] at (6,0){};
      }
    }
    \node [below] at (6,0)  {$0$};
    \node [below] at (4,0)  {$-1$};
    \node [below] at (8,0)  {$1$};
     \node [below] at (2,0)  {$-2$};
      \node [below] at (10,0)  {$2$};
    \node [below left] at (6,2)  {$1$};
    \node [below left] at (6,4)  {$2$};
    \node [below left] at (6,6)  {$3$};

    \draw [ultra thick,blue] (2,6) -- (8,0);
    \draw [ultra thick,red] (2,4) -- (6,0);
    \node [above left,blue] at (2,6) {\large $\alpha_E^{(3)}$};
    \node [above right,blue] at (4,4) {\large $\alpha_E^{(2)}$};
    \node [above right,blue] at (6,2) {\large $\alpha_E^{(1)}$};
     \node [above right,blue] at (8,0) {\large $\alpha_E^{(0)}$};
      \node [above left,red] at (2,4) {\large $\phi^{(2)}$};
      \node [above right,red] at (4,2) {\large $\phi^{(1)}$};
    \node [above right,red] at (6,0) {\large $\phi^{(0)}$};
  \end{tikzpicture}
\caption{Representation of the bigraded space $\Omega^{\bullet}(A,A^\circ\otimes \mathrm{End}_\bullet(\mathfrak{g}(A)[2]\oplus A[1]\oplus TM))$ and the Atiyah cocycle $\alpha_E$. Since the representation up to homotopy $E$ is concentrated in degrees $-2,-1,0$, the width of the diagram is 5, while the height corresponds to the rank of $A$. 
The Atiyah cocycle $\alpha_E$ is exact whenever the Atiyah cocycle  $\mathrm{at}_\nu$ is exact.}   \label{figure22}
\end{figure}

The chain complex $(E_\bullet,\partial)$ is a resolution of the quotient bundle $\nu(A)=TM/\rho(A)$, so that we are meeting  the assumptions established in Section \ref{subway}. Denoting $p$ the chain map projecting $E$ onto $\nu(A)$,
    \begin{center}
\begin{tikzcd}[column sep=2cm,row sep=2cm]
0\ar[r, "0"]&  \mathfrak{g}(A)[2]\ar[r, "\partial_{-1}=\iota\circ sp"]\ar[d,"p" left] &A[1]\ar[d,"p" left] \ar[r, "\partial_0=\rho\circ sp"]& TM \ar[d, "p" left]\ar[r,"0"]&0 \\
0\ar[r, dashed,  "0"]&0\ar[r, dashed,  "0"]&0\ar[r, dashed,  "0"]& {\nu(A)} \ar[r, dashed,  "0"]&0
\end{tikzcd} 
\end{center}
there exists a distinguished flat $A$-connection  $\nabla^{\nu}$ on $\nu(A)$ defined similarly to the Bott connection \eqref{equationBott}. For $a\in\Gamma(A)$ and $X\in\mathfrak{X}(M)$, it is
\begin{equation*}
\nabla^\nu_a(p(X))=p([\rho(a),X]),
\end{equation*}
which means
\begin{equation}\label{eqcontrat}
\nabla^\nu_a(p(X))=p(\nabla^{bas}_a(X)).
\end{equation}
The map $p$ also intertwines the curvature of the basic connection and the curvature of $\nabla^\nu$: 
\begin{equation*}
p\circ R_{\nabla^{bas}}\big|_{TM}=R_{\nabla^\nu}\circ p.
\end{equation*}
More generally, $p\circ R_{\nabla^E}=R_{\nabla^\nu}\circ p$. Then,  one extends the $A$-connection $\nabla^{\nu}$ to $L$ -- still denoted $\nabla^\nu$ -- by generalizing Equation \eqref{eqcontrat} to sections of $L$:
\begin{equation}\label{eqcontrat2}
    \nabla^\nu_\ell(p(X))=p(\nabla^E_\ell(X)),
\end{equation}
for any $\ell\in\Gamma(L)$. This implies that the projection $p$ also intertwines $R_{\nabla^E}$ and $R_{\nabla^\nu}$ (here, with respect to the extended $L$-connections).
The Atiyah cocycle associated to the extended connection is
\begin{equation*}
\mathrm{at}^\nu(a;l)=R_{\nabla^\nu}(a,\ell)=[\nabla^\nu_a,\nabla^\nu_\ell]-\nabla^\nu_{[a,\ell]}.
\end{equation*}
We then have the result that the Atiyah class of $E$ does not contain more information than that of $\nu(A)$.

\begin{proposition}\label{atnu}
$[\alpha_E]=0$ if and only if $[\mathrm{at}^{\nu}]=0$.
\end{proposition}

\begin{proof}
The implication $[\alpha_E]=0 \ \Longrightarrow \ [\mathrm{at}^{\nu}]=0$ is proven similarly as that in the proof of Proposition~\ref{propo}, so we only focus on the reverse implication. Based on Theorem~\ref{main1}, we will show that there exists a $A$-compatible $L$-superconnection. That is to say, that under an adequate change of connection, connection 2-form and connection 3-form, the corresponding Atiyah cocycle vanishes.
Thus, assume that the Atiyah cocyle $\mathrm{at}^\nu$ is exact, i.e. that there exists a connection form $\eta\in\Omega^1(L,\mathrm{End}(\nu(A)))$ vanishing on $A$ such that the modified connection $\nabla'^\nu=\nabla^\nu-\eta$ satisfies $
R_{\nabla'^\nu}(a,\ell)=0$.

A splitting of the exact sequence of vector bundles
\begin{center}
\begin{tikzcd}[column sep=1.5cm,row sep=1.5cm]
  0\ar[r]&\mathfrak{g}(A)\ar[r,"\iota"]&\ar[l,  "\sigma_2", bend left]A \ar[r,  "\rho"]&\ar[l,  "\sigma_1", bend left] TM \ar[r,  "p"]& \ar[l,  "\sigma_0", bend left]\nu(A)\ar[r]&0
\end{tikzcd} 
\end{center}
is a family of vector bundle morphisms $(\sigma_0,\sigma_1,\sigma_2)$, all of which have constant rank, and satisfying the two following properties:
\begin{enumerate}
\item 
$\mathrm{Im}(\sigma_i)=\mathrm{Ker}(\sigma_{i+1})$,  and so $TM=\mathrm{Ker}(\sigma_1)\oplus \mathrm{Im}(\rho)$  and $A=\mathfrak{g}(A)\oplus \mathrm{Im}(\sigma_1)$;
\item $p\circ \sigma_0=\mathrm{id}_{\nu(A)}$, $\rho\circ\sigma_1+\sigma_0\circ p=\mathrm{id}_{TM}$, $\iota\circ\sigma_2+\sigma_1\circ \rho=\mathrm{id}_{A}$ and $\sigma_2\circ \iota=\mathrm{id}_{\mathfrak{g}(A)}$.
\end{enumerate}
Under this choice of splitting, the bundle morphisms $\partial_{-1}\colon\mathfrak{g}(A)\to A$, $\partial_{0}\colon A\to TM$ and $p\colon TM\to\nu(A)$ have the following block-matrix form:
\begin{equation*}
    \partial_{-1}=\begin{pmatrix}
        \iota\\
      \  0\
    \end{pmatrix},\qquad \partial_0=\begin{pmatrix}
       \ 0 & 0\ \\
      \  0 &\rho \
    \end{pmatrix} \qquad\text{and}\qquad p=\begin{pmatrix}
       \ p & 0 \
    \end{pmatrix},
\end{equation*}
where we omitted (here and in the following) to write the suspension maps.

Since the connection $\nabla^E$ leaves invariant the kernel $\mathfrak{g}(A)$ in the Lie algebroid~$A$, the vector bundle morphism $\alpha^{(0),A\to TM}_{E}(l)$ induces a vector bundle morphism from $\mathrm{Im}(\sigma_1)$ 
to $\mathrm{Im}(\rho)\subset TM$.
Post-composing it with $-\sigma_1$, we then obtain an endomorphism of $\mathrm{Im}(\sigma_1)$, denoted $\kappa=-\sigma_1\circ \alpha^{(0),A\to TM}_{E}\in\Gamma(A^\circ\otimes\mathrm{End}(\mathrm{Im}(\sigma_1))$. In particular, the definition of~$\kappa$ has been made so that $[\nabla^E_\ell,\partial]\big|_A=-\rho\circ\kappa(\ell)$. Moreover, consider the one-form on~$L$ (vanishing on $A$) valued in the endomorphisms of $\mathrm{Im}(\sigma_0)=\mathrm{Ker}(\sigma_1)$, $\tau=\sigma_0\circ \eta\circ p$.
Using the direct sum decomposition of $TM$ and $A$, we can then define three connection forms $\varphi^{(1),TM\to TM}\in\Omega^1(L,\mathrm{End}(TM))$, $\varphi^{(1),A\to A}\in\Omega^1(L,\mathrm{End}(A))$, and $\varphi^{(1),\mathfrak{g}(A)\to \mathfrak{g}(A)}\in\Omega^1(L,\mathrm{End}(\mathfrak{g}(A)))$ which, in a block matrix form, are
\begin{equation*}
    \varphi^{(1),TM\to TM}=\begin{pmatrix}
\tau& 0\\ 
0 & \rho\circ \pi\circ \sigma_1
\end{pmatrix}, \quad     \varphi^{(1),A\to A}=\begin{pmatrix}
\chi& 0\\ 
0 &  \pi+\kappa
\end{pmatrix}\quad  \text{and}\quad  \varphi^{(1),\mathfrak{g}(A)\to \mathfrak{g}(A)}=\chi.
\end{equation*}
Here $\pi$ (resp. $\chi$)  is any section of $A^\circ\subset L^*$ taking values in $\mathrm{End}(\mathrm{Im}(\sigma_1))$ (resp. in $\mathrm{End}(\mathfrak{g}(A))$). 
In particular, one can choose $\pi=0$ and $\chi=0$.

Then, by construction,  we have the identity:
\begin{equation*}
    [\varphi^{(1)}(\ell),\partial]=\alpha_E^{(0)}(p(\ell)),
\end{equation*}
for indeed, if $a\in\Gamma(\mathfrak{g}(A))$, we have:
\begin{equation*}
     [\varphi^{(1)}(\ell),\partial](a)
     =\begin{pmatrix}
\chi& 0\\ 
0 &  \pi+\kappa
\end{pmatrix}\begin{pmatrix}
        \iota(a)\\
      \  0\
    \end{pmatrix}-\begin{pmatrix}
        \iota\\
      \  0\
    \end{pmatrix}\chi(a)\\
    =\begin{pmatrix}
        \chi(\iota(a))\\
      \  0\
    \end{pmatrix}-\begin{pmatrix}
        \iota(\chi(a))\\
      \  0\
    \end{pmatrix}=0,
\end{equation*}
while, if $a\in\Gamma(\mathrm{Im(\sigma_1)})$, we have:\begin{align*}
    [\varphi^{(1)}(\ell),\partial](a)
    &=\begin{pmatrix}
\tau(\ell)& 0\\ 
0 & \rho\circ \pi(\ell)\circ \sigma_1
\end{pmatrix}
\begin{pmatrix}
    0 & 0\\
    0 &\rho(a) 
    \end{pmatrix}
    -\begin{pmatrix}
     0 & 0\\
     0 &\rho \
    \end{pmatrix}
    \begin{pmatrix}
\chi(\ell)(a)& 0\\ 
0 &  (\pi(\ell)+\kappa(\ell))(a)
\end{pmatrix}\\
&=\begin{pmatrix}
0& 0\\ 
0 & \rho\circ \pi(\ell)(a-\iota\circ \sigma_2(a))
\end{pmatrix}-\begin{pmatrix}
0& 0\\ 
0 & \rho\circ \pi(\ell)(a)+\rho\circ \kappa(\ell)(a))
\end{pmatrix}\\
&=\begin{pmatrix}
0& 0\\ 
0 & \alpha_E^{(0),A\to TM}(p(\ell))(a)- \sigma^{(0)}\circ p\circ \alpha_E^{(0),A\to TM}(p(\ell))(a)
\end{pmatrix}.
\end{align*}
We used the fundamental identities satisfied by the splitting, and the fact that $\kappa=-\sigma_1\circ \alpha_E^{(0),A\to TM}$. 
The term $\rho\circ \pi(\ell)\circ\iota\circ \sigma_2(a)$ is  zero because $\sigma_2(a)=0$ as $a$ is a section of $\mathrm{Im}(\sigma_1)=\mathrm{Ker}(\sigma_2)$.
Moreover, $\alpha_E^{(0),A\to TM}(p(\ell))$ takes values in the image of the anchor map so post-composing with the projection $p$ cancels the contribution $\sigma^{(0)}\circ p\circ \alpha_E^{(0),A\to TM}(p(\ell))(a)$.

Then, from Equation \eqref{eqpizz1}, by setting $\nabla'^E=\nabla^E-\varphi^{(1)}$ one obtains that:
\begin{equation}
\alpha'^{(0)}_E(p(\ell))=[\nabla'^E_\ell,\partial]=0.\label{eqatiy0}
\end{equation}
Thus, we modified the connection so that the first component $\alpha'^{(0)}_E$ of the Atiyah cocycle associated to $\nabla'^E$ vanishes. Moreover, the sub-vector bundle $\mathfrak{g}(A)$ remains stable under the action of this connection.
Eventually, by definition of $\tau$ and of $\varphi^{(1),TM\to TM}$, we have that Equation \eqref{eqcontrat2} still holds for $\nabla'^E$ and $\nabla'^\nu$:
\begin{align*}
    p(\nabla'^E_\ell(X))
    &=p(\nabla^E_\ell(X))-\begin{pmatrix}
p\circ\tau(\ell)& 0\\ 
0 & p\circ\rho\circ \pi(\ell)\circ \sigma_1
\end{pmatrix}(X)\\
&=\nabla^\nu_\ell(p(X))-
\begin{pmatrix}
\eta(\ell)\circ p(X)& 0\\ 
0 & 0
\end{pmatrix}=\nabla'^\nu_\ell(p(X)),
\end{align*}
where we used the identity $p\circ \sigma_0=\mathrm{id}_{\nu(A)}$ as well as the fact that $p\circ\rho=0$.

Defining $\varphi^{(2)}$ is slightly more subtle than in Example \ref{subnormal}. The curvature of the new connection $\nabla'^E$ -- denoted $R_{\nabla'^E}$ --  splits into three components;  $R_{\nabla'^E}^{TM\to TM}$, $R_{\nabla'^E}^{A\to A}$ and $R_{\nabla'^E}^{\mathfrak{g}(A)\to \mathfrak{g}(A)}$. Using the splitting $A=\mathfrak{g}(A)\oplus\mathrm{Im}(\sigma_1)$, one can write the  second curvature in a block-matrix form similar to that appearing in Equation \eqref{eqmatrix1}:
\begin{equation*}
    R_{\nabla'^E}^{A\to A}=\begin{pmatrix}
R_{\nabla'^E}^{\mathfrak{g}(A)\to \mathfrak{g}(A)}& R_{\nabla'^E}^{\mathrm{Im}(\sigma_1)\to \mathfrak{g}(A)}\\
0 &  R_{\nabla'^E}^{\mathrm{Im}(\sigma_1)\to \mathrm{Im}(\sigma_1)}\
\end{pmatrix}.
\end{equation*}
The lower left block corresponds to $R_{\nabla'}^{\mathfrak{g}(A)\to \mathrm{Im}(\sigma_1)}$ and is identically zero because $\mathfrak{g}(A)$ is stable under the action of $\nabla'^E$.
Regarding the curvature element $R_{\nabla'^E}^{TM\to TM}$, we have the following block-matrix form with respect to the splitting $\mathrm{Ker}(\sigma_1)\oplus \mathrm{Im}(\rho)$:
\begin{equation*}
    R_{\nabla'^E}^{TM\to TM}=\begin{pmatrix}
 R_{\nabla'^E}^{\mathrm{Ker}(\sigma_1)\to \mathrm{Ker}(\sigma_1)}&0\\
   R_{\nabla'^E}^{\mathrm{Ker}(\sigma_1)\to \mathrm{Im}(\rho)}&R_{\nabla'^E}^{\mathrm{Im}(\rho)\to \mathrm{Im}(\rho)}\
\end{pmatrix}.
\end{equation*}

We now explain why the $\mathrm{Im}(\rho)\to \mathrm{Ker}(\sigma_1)$ sector, in the upper-right corner, is identically zero.
Recall that Equation \eqref{eqcontrat2} -- valid for the primed connections -- implies that the projection $p:TM\to\nu(A)$ intertwines $R_{\nabla'^\nu}$ and $R_{\nabla'^E}^{TM\to TM}$. In block matrix form, decomposing a vector field $X\in\mathfrak{X}(M)$ with respect to the splitting as $X=X^{\mathrm{Ker}(\sigma_1)}+X^{\mathrm{Im}(\rho)}$, one has:
\begin{align*}
R_{\nabla'^\nu}\big(p(X)\big)
&=\begin{pmatrix}
       \ p & 0 \
    \end{pmatrix}R_{\nabla'^E}^{TM\to TM}\begin{pmatrix}
        X^{\mathrm{Ker}(\sigma_1)} \\ X^{\mathrm{Im}(\rho)}
    \end{pmatrix}\\
    &=p\circ R_{\nabla'^E}^{\mathrm{Ker}(\sigma_1)\to \mathrm{Ker}(\sigma_1)}(X^{\mathrm{Ker}(\sigma_1)})+ p\circ  R_{\nabla'^E}^{\mathrm{Im}(\rho)\to \mathrm{Ker}(\sigma_1)}(X^{\mathrm{Im}(\rho)}).
\end{align*}
 From this, and since $p(X)=p(X^{\mathrm{Ker}(\sigma_1)})$, we deduce that $p\circ  R_{\nabla'^E}^{\mathrm{Im}(\rho)\to \mathrm{Ker}(\sigma_1)}=0$. Post-composing by $\sigma_0$, we have:
\begin{equation*}
    0=R_{\nabla'^E}^{\mathrm{Im}(\rho)\to \mathrm{Ker}(\sigma_1)}-\rho\circ \sigma_1\circ R_{\nabla'^E}^{\mathrm{Im}(\rho)\to \mathrm{Ker}(\sigma_1)}.
\end{equation*}
The last term vanishes because the curvature lands in the kernel of $\sigma_1$, hence the result. Moreover, from this discussion, we also deduce that
\begin{equation}\label{eq:durdur}
    R_{\nabla'^\nu}(p(X))=p\circ R_{\nabla'^E}^{\mathrm{Ker}(\sigma_1)\to \mathrm{Ker}(\sigma_1)}(X),
\end{equation}
for any vector field $X\in\mathfrak{X}(M)$.
 Now, recall that the Atiyah cocycle associated to $\nabla'^\nu$ is identically zero. Thus, post-composing Equation \eqref{eq:durdur} with $\sigma_0$ and evaluating the result on a pair $(a,\ell)$ we obtain
that $R_{\nabla'^E}^{\mathrm{Ker}(\sigma_1)\to \mathrm{Ker}(\sigma_1)}(a,\ell)=0$. In other words, the curvature $R_{\nabla'^E}^{TM\to TM}$, when evaluated on a pair $(a,\ell)$, has two empty blocks in the upper line.



In a similar manner, the connection 2-form $\omega^{(2)}_L$ splits into two sectors $\omega_L^{(2),A\to \mathfrak{g}(A)}\in\Omega^2(L,\mathrm{Hom}(A,\mathfrak{g}(A)))$ and $\omega_L^{(2),TM\to A}\in\Omega^2(L,\mathrm{Hom}(TM,A))$.
As justified above Equation~\eqref{eq:always}, the first component can be chosen to be zero. 
On the other hand, the component $\omega_L^{(2),TM\to A}$ has the following block-matrix form:
\begin{equation*}
    \omega^{(2),TM\to A}_L=\begin{pmatrix}
\omega^{(2),\mathrm{Ker}(\sigma_1)\to \mathfrak{g}(A)}_L&\omega^{(2),\mathrm{Im}(\rho)\to \mathfrak{g}(A)}_L\\
\omega^{(2),\mathrm{Ker}(\sigma_1)\to \mathrm{Im}(\sigma_1)}_L&\omega^{(2),\mathrm{Im}(\rho)\to \mathrm{Im}(\sigma_1)}_L 
\end{pmatrix}.
\end{equation*}
Although present in the matrix, the upper left will play no role in modifying Equation~\eqref{eqpizz2} as the block matrix form of $\partial_0$ will always cancel it out. This observation is also valid for the basic curvature: Equation \eqref{eqchat2} does not make use of all of the components of this tensor.
Then, define  $\varphi^{(2), A\to\mathfrak{g}(A)}$ as in Equation \eqref{eqmatrix2}:
\begin{equation*}
    \varphi^{(2), A\to\mathfrak{g}(A)}=\begin{pmatrix}
R_{\nabla'^E}^{\mathfrak{g}(A)\to \mathfrak{g}(A)}& {\begin{tabular}[c]{@{}c@{}} $R_{\nabla'^E}^{\mathrm{Im}(\sigma_1)\to \mathfrak{g}(A)}$\\
$\hspace{1.5cm}+\omega^{(2),\mathrm{Im}(\rho)\to \mathfrak{g}(A)}_L\circ \rho$\end{tabular}}\
\end{pmatrix}.
\end{equation*}
This is a two form on $L$ taking values in $\mathrm{Hom}(A,\mathfrak{g}(A))$.
Take then $\varphi^{(2),TM\to A}$ to be the following 2-form on $L$ valued in $\mathrm{Hom}(TM,A)$:
\begin{equation*}
    \varphi^{(2),TM\to A}=\begin{pmatrix}
 0&0\\
  {\begin{tabular}[c]{@{}c@{}}$\sigma_1\circ R_{\nabla'^E}^{\mathrm{Ker}(\sigma_1)\to \mathrm{Im}(\rho)}$ \\ $\hspace{1cm}+\omega^{(2),\mathrm{Ker}(\sigma_1)\to \mathrm{Im}(\sigma_1)}_L$\end{tabular}}&{\begin{tabular}[c]{@{}c@{}}$\frac{1}{2}\sigma_1\circ R_{\nabla'^E}^{\mathrm{Im}(\rho)\to \mathrm{Im}(\rho)}$\\ $\hspace{1cm}+\frac{1}{2}R_{\nabla'^E}^{\mathrm{Im}(\sigma_1)\to \mathrm{Im}(\sigma_1)}\circ \sigma_1$\\ $\hspace{1cm}+\omega^{(2),\mathrm{Im}(\rho)\to \mathrm{Im}(\sigma_1)}_L$\end{tabular}}\
\end{pmatrix}.
\end{equation*}
They both vanish whenever restricted to $A$, as is expected.

The anchor map commuting with connection $\nabla'^E$ (see Equation \eqref{eqatiy0}), together with the respective properties of $\sigma_0,\sigma_1$ and $\sigma_2$, implies that we have the following relationships:
\begin{align}
    R_{\nabla'^E}^{\mathrm{Im}(\rho)\to \mathrm{Im}(\rho)}&=\rho\circ R_{\nabla'^E}^{\mathrm{Im}(\sigma_1)\to \mathrm{Im}(\sigma_1)}\circ\sigma_1, &R_{\nabla'^E}^{\mathrm{Im}(\sigma_1)\to \mathrm{Im}(\sigma_1)}&= \sigma_1\circ R_{\nabla'^E}^{\mathrm{Im}(\rho)\to \mathrm{Im}(\rho)}\circ\rho, \label{eqport1}\\
    R_{\nabla'^E}^{\mathrm{Im}(\rho)\to \mathrm{Im}(\rho)}&=\rho\circ \sigma_1\circ  R_{\nabla'^E}^{\mathrm{Im}(\rho)\to \mathrm{Im}(\rho)},
    &R_{\nabla'^E}^{\mathrm{Im}(\sigma_1)\to \mathrm{Im}(\sigma_1)}&= R_{\nabla'^E}^{\mathrm{Im}(\sigma_1)\to \mathrm{Im}(\sigma_1)}\circ\rho \circ \sigma_1. \label{eqport2}
\end{align}
One can check that these equations imply the following identities, which should be interpreted in block matrix form:
\begin{align}
R_{\nabla'^E}^{\mathfrak{g}(A)\to \mathfrak{g}(A)}(a,\ell)&=\varphi^{(2), A\to \mathfrak{g}(A)}(a,\ell)\circ \partial_{-1},\label{eqatiy1}\\
R_{\nabla'^E}^{A\to A}(a,\ell)+\omega_L^{(2),TM\to A}(a,\ell)\circ \partial_0&=\partial_{-1}\circ\varphi^{(2), A\to \mathfrak{g}(A)}(a,\ell)\label{eqatiy2}\\
&\hspace{2.5cm}+\varphi^{(2), TM\to A}(a,\ell)\circ \partial_0,\nonumber\\
    R_{\nabla'^E}^{TM\to TM}(a,\ell)+\partial_0\circ\omega_L^{(2),TM\to A}(a,\ell)&=\partial_0\circ\varphi^{(2), TM\to A}(a,\ell)\label{eqatiy3}.
\end{align}
Then, setting $\omega'^{(2)}_L=\omega^{(2)}_L-\varphi^{(2)}$, Equations \eqref{eqatiy1}-\eqref{eqatiy3} are equivalent to 
\begin{equation}\label{eqatiy4}
    \alpha_E'^{(1)}(a;p(\ell))=R_{\nabla'^E}(a,\ell)+[\partial,\omega_L'^{(2)}](a,\ell)=0.
\end{equation}
In other words, the second component  of the Atiyah cocycle associated to $\nabla'^E$ and $\omega'^{(2)}_L$ vanishes.

Now, for the last two components of the Atiyah cocycle, let us define a 3-form on $L$ taking values in $\mathrm{Hom}(TM,\mathfrak{g}(A)$, denoted $\varphi^{(3)}$, in the following block matrix form:
\begin{equation*}
    \varphi^{(3)}=\begin{pmatrix}
d^{\nabla'^E}\omega_L'^{(2),\mathrm{Ker}(\sigma_1)\to \mathfrak{g}(A)}& d^{\nabla'^E}\omega_L'^{(2),\mathrm{Im}(\rho)\to \mathfrak{g}(A)}
\end{pmatrix}.
\end{equation*}
Whenever restricted to $A$, given that $\varphi^{(1)}$ and $\varphi^{(2)}$ vanish on $A$, this 3-form becomes:
\begin{equation*}
    \varphi^{(3)}(a,b,c)=\begin{pmatrix}
d^{\nabla^E}\omega^{(2),\mathrm{Ker}(\sigma_1)\to \mathfrak{g}(A)}(a,b,c)& d^{\nabla^E}\omega^{(2),\mathrm{Im}(\rho)\to \mathfrak{g}(A)}(a,b,c)
\end{pmatrix},
\end{equation*}
which identically vanish, by Equation \eqref{eqbegaudo}. 
Then, using Equation \eqref{eq:always}, the Bianchi identities for $R_{\nabla'^E}^{\mathfrak{g}(A)\to \mathfrak{g}(A)}$ and $R_{\nabla'^E}^{\mathrm{Im}(\sigma_1)\to \mathfrak{g}(A)}$, together with Equation \eqref{eqatiy0}, one has
\begin{equation*}
    d^{\nabla'^{E}}\omega_L'^{(2),A\to\mathfrak{g}(A)}
    =-d^{\nabla'^{E}}\varphi^{(2),A\to\mathfrak{g}(A)}
    =
    -\begin{pmatrix}
\ 0& d^{\nabla'^{E}}\omega_L^{(2),\mathrm{Im}(\rho)\to \mathfrak{g}(A)}\circ \rho
\end{pmatrix}=-\varphi^{(3)}\circ\partial_0.
\end{equation*}
On the other hand, using the Bianchi identity for $R_{\nabla'^E}^{\mathrm{Ker}(\sigma_1)\to\mathrm{Im}(\rho)}$, $R_{\nabla'^E}^{\mathrm{Im}(\rho)\to \mathrm{Im}(\rho)}$ and $R_{\nabla'^E}^{\mathrm{Im}(\sigma_1)\to \mathrm{Im}(\sigma_1)}$, together with Equation \eqref{eqatiy0} applied to Equations \eqref{eqport1} and \eqref{eqport2}, we deduce that:
\begin{equation*}
    d^{\nabla'^{E}}\omega_L'^{(2),TM\to A}
    =\begin{pmatrix}
 d^{\nabla'^{E}}\omega'^{(2),\mathrm{Ker}(\sigma_1)\to \mathfrak{g}(A)}_L& d^{\nabla'^{E}}\omega'^{(2),\mathrm{Im}(\rho)\to \mathfrak{g}(A)}_L\\
0&0
\end{pmatrix}=\partial_{-1}\circ\varphi^{(3)}.
\end{equation*}
Therefore, we can summarize these two identities under the following unique one:
\begin{equation*}
    d^{\nabla'^{E}}\omega_L'^{(2)}=[\partial,\varphi^{(3)}].
\end{equation*}
If we define  $\omega_L'^{(3)}=-\varphi^{(3)}$ -- understanding that $\omega_L^{(3)}=0$ -- the third component of the Atiyah cocycle associated to $\nabla'^E$, $\omega'^{(2)}_L$ and $\omega'^{(3)}_L$ then reads:
\begin{equation}\label{eqatiy5}
    \alpha'^{(2)}_E(a,b;p(\ell))=d^{\nabla'^{E}}\omega_L'^{(2)}(a,b,\ell)+[\partial,\omega_L'^{(3)}](a,b,\ell)=0.
\end{equation}

Now, we need only prove that the fourth component of the Atiyah cocycle associated to $\nabla'^E$, $\omega'^{(2)}_L$ and $\omega'^{(3)}_L$ vanishes. This component is the primed version of Equation \eqref{eqprince}, together with an additional term involving the connection 3-form:
\begin{equation}\label{eqprince2}
\alpha'^{(3)}_E(a,b,c;p(\ell))=d^{\nabla'^E}\omega_L'^{(3)}(a,b,c,\ell)+\omega_L'^{(2), A\to\mathfrak{g}(A)}\wedge \omega_L'^{(2), TM\to A}(a,b,c,\ell).
\end{equation}
There is no additional term as the graded vector bundle is concentrated in degree $-2,-1,0$. Let us apply the covariant derivative $d^{\nabla'^E}$ on Equation \eqref{eqatiy5} so that, using Equations \eqref{eqatiy0} and  \eqref{eqatiy4} which are valid on any subset of elements of the quadruple $(a,b,c,\ell)$, we have:
\begin{align*}
0
&=\big(d^{\nabla'^E}\big)^2\omega_L'^{(2)}(a,b,c,\ell)+d^{\nabla'^E}[\partial,\omega_L'^{(3)}](a,b,c,\ell)\\
&=-[[\partial,\omega_L'^{(2)}],z\omega_L'^{(2)}](a,b,c,\ell)-[\partial,d^{\nabla'^E}\omega_L'^{(3)}](a,b,c,\ell)\\
&=-\frac{1}{2}[\partial,[\omega_L'^{(2)},\omega_L'^{(2)}]](a,b,c,\ell)-[\partial,\alpha'^{(3)}_E](a,b,c;l)+[\partial,\omega_L'^{(2)}\wedge\omega_L'^{(2)}](a,b,c,\ell).
\end{align*}
The first and last terms cancel each other and we are left with:
\begin{equation}
[\partial,\alpha'^{(3)}_E](a,b,c;l)=0.\label{eqproxyyy}
\end{equation}
Since $\alpha'^{(3)}_E(a,b,c;l)$ is a vector bundle morphism from $TM$ to $\mathfrak{g}(A)$, evaluating Equation~\eqref{eqproxyyy} on a vector field $X\in\mathfrak{X}(M)$ gives:
\begin{equation*}
0=\partial_{-1}\circ \alpha'^{(3)}_E(a,b,c;l)(X)=\iota\circ \alpha'^{(3)}_E(a,b,c;l)(X).
\end{equation*}
The inclusion map being injective we deduce that the fourth component of the Atiyah cocycle vanishes, as desired.

Thus, we have shown that under the change of connection $\nabla^E\mapsto \nabla'^E$ and connection 2-form $\omega^{(2)}_L\mapsto \omega'^{(2)}_L$, together with adding a connection 3-form $\omega_L'^{(3)}$ (vanishing on $A$), 
their associated Atiyah cocycle becomes:
\begin{align*}
    \alpha'^{(0)}_E(l)&=[\nabla'^E_\ell,\partial]=0,\\
    \alpha'^{(1)}_E(a;l)&=R_{\nabla'^E}(a,\ell)+[\partial,\omega'^{(2)}_L](a,\ell)=0,\\
    \alpha'^{(2)}_E(a,b;l)&=d_L^{\nabla'^E} \omega'^{(2)}_L(a,b,\ell)+[\partial,\omega'^{(3)}_L](a,b,\ell)=0,\\
    \alpha'^{(3)}_E(a,b,c;l)&=d^{\nabla'^E}\omega_L'^{(3)}(a,b,c,\ell)+\omega_L'^{(2)}\wedge \omega_L'^{(2)}(a,b,c,\ell)=0, 
\end{align*}
for any $a,b,c\in\Gamma(A)$ and any $\ell\in\Gamma(L)$ projecting down to $l\in\Gamma(L/A)$.
That is to say, the Atiyah cocycle $\alpha_E'$ associated to the $L$-superconnection $D'_L=\partial+d^{\nabla'^E}+\omega_E'^{(2)}\wedge\cdot +\omega_E'^{(3)}\wedge \cdot$ is identically zero, so $D_L$ is $A$-compatible. By Theorem~\ref{main1}, it means that the Atiyah class $[\alpha'_E]=[\alpha_E]$  vanishes, which proves the claim of the proposition. Notice that any other choice of  splittings $\sigma_0,\sigma_1, \sigma_2$  does not change this result.\end{proof}

\begin{remark}
Notice that 
the Atiyah cocycle $\mathrm{at}^\nu$ being a coboundary was instrumental in proving Equation \eqref{eqatiy3}, and thus Equation \eqref{eqatiy4} too. Using this identity, we then could prove that $\alpha'^{(3)}_E=0$ too so, to prove that $[\alpha_E]=0$ we crucially used the  
fact that $[\mathrm{at}^\nu]=0$.
\end{remark}

\begin{remark}
From the graded geometric perspective, since $\varphi^{(1)}$, $\varphi^{(2)}$ and $\varphi^{(3)}$ identically vanish when restricted to $A$, they canonically induce elements  $\phi^{(0)}\in \Gamma(A^\circ\otimes \mathrm{End}_0(E))$, $\phi^{(1)}\in\Omega^1(A,A^\circ\otimes \mathrm{End}_{-1}(E))$ and $\phi^{(2)}\in\Omega^2(A,A^\circ\otimes \mathrm{End}_{-2}(E))$, respectively. Their definition has been made so that their sum  $\phi=\phi^{(0)}+\phi^{(1)}+\phi^{(2)}$ satisfies $\alpha_E=s(\phi)$, showing explicitly that $\alpha_E$ is a coboundary. 
\end{remark}

\section{Relationship with Atiyah classes of differential graded vector bundles}

This section is devoted to generalizing a result obtained for the normal complex in \cite{liaoAtiyahClassesTodd2023}, to representations up to homotopy.
\subsection{Generalities about graded geometry and \texorpdfstring{$Q$}{Q}-connections}
A graded manifold $\mathcal{M}$ is a locally ringed space whose structure sheaf $\mathcal{C}^\infty(\mathcal{M})$ is locally of the form $\mathcal{C}^\infty(U)\otimes\widehat{S}(V^*)$, where $U\subset M$ is an open subset of a given smooth manifold $M$ -- called the \emph{base} or \emph{body} -- and $\widehat{S}(V^*)$ stands for  the formal power series over the dual of a graded vector space $V=\bigoplus_{i\in\mathbb{Z}}V_i$. We refer to \cite{faironIntroductionGradedGeometry2017,kotovCategoryGradedManifolds2024} for an explanation of why one should use formal power series instead of mere polynomials (which are not local rings). We say that~$\mathcal{M}$ is positively (resp. negatively) graded if the formal power series over $V^*$ are positively (resp. negatively) graded, namely, if $V$ is a negatively (resp. positively) graded vector space.  If the graded manifold is (non-) positively or (non-) negatively graded, which will be the case in the present section, using formal power series is equivalent to using polynomials.

Following \cite[Remark 2.2.1]{mehtaSupergroupoidsDoubleStructures2006}, one should mainly consider graded vector spaces of finite dimension, such that all but a finite number of the $V_i$ are zero, and those which are not are finite dimensional. This choice allows performing tensor product operations in the category of graded vector spaces. This assumption is in line with the choice of graded vector bundles made in Section \ref{sec0}.
Batchelor's theorem establishes that every positively (resp. negatively) graded manifold is locally isomorphic to a split negatively (resp. positively) graded vector bundle over a smooth manifold. For a generalization of Batchelor's statement to gradings over the integers, we refer to \cite{kotovCategoryGradedManifolds2024}.
While in Section \ref{sec0} we worked with (split, finite dimensional) graded vector bundles over smooth manifolds, we refer to \cite{mehtaQalgebroidsTheirCohomology2009, mehtaAtiyahClassDgvector2015} for details about defining graded vector bundles over a graded manifold $\mathcal{M}$ in full generality. We will not need those apart from settling the conventions.

\begin{definition}
A \emph{differential graded (dg) manifold} or \emph{$Q$-manifold} is a graded manifold $\mathcal{M}$ equipped with a cohomological degree $+1$ vector field $Q$, i.e. such that $[Q,Q]=0$.
\end{definition}

\begin{definition}\cite{aschieriCurvatureTorsionCourant2021}\label{dg con} Let $\mathcal{E}$ be a graded vector bundle over the  dg manifold $(\mathcal{M},Q)$.
    A \emph{dg connection} (or \emph{$Q$-connection}) on $\mathcal{E} $ is a degree $+1$ vector field $Q_{\mathcal{E}}\in\mathfrak{X}(\mathcal{E^*})$ satisfying the following conditions:%
    \begin{enumerate}
        \item $Q_{\mathcal{E}}$ preserves $\Gamma(\mathcal{E})\subset \mathcal{C}^\infty(\mathcal{E}^*)$,
        \item $Q_{\mathcal{E}}$
        projects to $Q$, i.e $Q_{\mathcal{E}}|_{\mathcal{M}}=Q$.   
    \end{enumerate}
 The curvature of a dg connection $Q_{\mathcal{E}}$ on $\mathcal{E}$ is the degree 2 vector field on~$\mathcal{E^*}$ given by:
    \begin{equation*}
        R_{Q_{\mathcal{E}}}=Q_{\mathcal{E}}^2=\frac{1}{2}[Q_{\mathcal{E}},Q_{\mathcal{E}}].
    \end{equation*}
\end{definition}


\begin{lemma}\label{lemmameta} \cite{mehtaAtiyahClassDgvector2015} Let  $Q_{\mathcal{E}}$ be a $Q$-connection on $\mathcal{E}$.
  If the curvature of  $Q_\mathcal{E}$ is zero then $(\mathcal{E},Q_{\mathcal{E}})$ is a \emph{differential graded (dg) vector bundle} -- or  \emph{$Q$-bundle} -- i.e. a vector bundle is the category of dg manifolds. 
  \end{lemma}

A Lie algebroid $A$ can be realized as a differential graded manifold by shifting the degree of the fibers by $-1$. The corresponding graded manifold being denoted $A[1]$ is modelled on the vector bundle $\pi: A[1]\to M$. It is also known that the Lie algebroid differential corresponds to a cohomological vector field of degree $+1$.
Moreover, the space $\Omega(A,E)$ of $A$-forms valued in $E$ can be seen as the space of sections of the pull-back graded vector bundle $\mathcal{E}=\pi^*E\to A[1]$ over the dg manifold $A[1]$. By duality \cite{mehtaQalgebroidsTheirCohomology2009}, the sheaf of sections of $\mathcal{E}\to A[1]$ -- denoted $\Gamma_{A[1]}(\mathcal{E})$ -- is the space of linear functions on the dual graded vector bundle $\mathcal{E}^*=\pi^*E^*\to A[1]$, denoted $\mathcal{C}^\infty_{lin}(\mathcal{E}^*)$. Definition \ref{defquillen} of an $A$-superconnection on $E$ can now be understood from the differential graded geometric language.

\begin{lemma}
    The data of an  $A$-superconnection on the graded vector bundle $E$ is equivalent to the data of a dg connection on the dg vector bundle $\mathcal{E}=\pi^* E\to A[1]$.
\end{lemma}

\begin{proof}
 $\Longrightarrow$ An $A$-superconnection $D_A\colon\Omega(A,E)\to \Omega(A,E)$ of total degree $+1$ is understood as a linear vector field on $\mathcal{E}^*$, i.e. a degree $+1$ vector field $D_A$ on $\mathcal{E}^*$ preserving $\mathcal{C}^\infty_{lin}(\mathcal{E}^*)$. Moreover, this vector field is compatible with the Lie algebroid differential $d_A$ in the sense that the canonical projection $\mathcal{P}:\mathcal{E}^*\to A[1]$ is a morphism of differential graded manifolds:
\begin{equation*}
    \mathcal{P}^*\circ d_A=D_A\circ \mathcal{P}^*.
\end{equation*}
The $A$-superconnection $D_A$ is thus  a $Q$-connection on $\mathcal{E}$. 

\noindent $\Longleftarrow$ Conversely, choose a $Q$-connection on $\mathcal{E}$. Under the identification 
\begin{equation*}
    \Gamma(\mathcal{E})\simeq \Omega(A)\otimes \Gamma(E)
\end{equation*}
the  dg-connection $Q_{\mathcal{E}}$ is determined by its action on $\Gamma( E)$. Since it is a degree 1 operator with respect to the total degree we have that the restriction on $\Gamma(E_i)$ is a map
\begin{align*}
    Q_{\mathcal{E}}\colon\Gamma(E_i)\longrightarrow \bigoplus_{k=0}^{\mathrm{rk}(A)}\Omega^k(A)\otimes \Gamma(E_{i-k+1}).
\end{align*}
We write $Q_{\mathcal{E}}$ in term of components: 
\begin{equation*}
    Q_{\mathcal{E}}=\sum_{k=0}^{\mathrm{rk}(A)}Q^{(k)}_{\mathcal{E}},
\end{equation*}
where $Q^{(k)}_{\mathcal{E}}\colon\Gamma(E_\bullet)\to \Omega^k(A)\otimes\Gamma(E_{\bullet-k+1})$. 
Observe that the component $Q^{(0)}_{\mathcal{E}}\colon\Gamma(E_\bullet)\to\Gamma(E_\bullet)$ is a vector bundle morphism $\partial:\Gamma(E_\bullet)\to\Gamma(E_{\bullet+1})$, while the component
$Q^{(1)}_{\mathcal{E}}\colon\Gamma(E_\bullet)\to\Omega^1(A)\otimes\Gamma(E_\bullet)$ is the exterior covariant derivative associated to an $A$-connection on $E$.

Introducing local coordinates $\{x^i,\xi^\alpha\}$ on $A[1]$ and fiber coordinates $\{s_i^l\}_{1\leq l\leq r_i}$  on $E$, where $r_i=\text{rk}(E_i)$, then $Q_\mathcal{E}$  is given by:
    %
\begin{equation}\label{2.35}
    Q_{\mathcal{E}}  =d_A+\Gamma_l^ps_i^l\frac{\partial}{\partial s_i^p}
      = d_A +(\Gamma_{\alpha l}^p\xi^\alpha s_i^l+\Gamma_{\beta\gamma l}^p\xi^\beta \xi^\gamma s_{i-1}^l+\underbrace{\cdots}_{\text{higher degree terms}})\frac{\partial}{\partial s_i^p},
\end{equation}
where $\Gamma^p_l\in C^\infty(A[1])\simeq\Omega(A)$ has degree $|\Gamma^p_l|=1+|s_p|-|s_l|$ and is a polynomial of $\xi$ variables, as shown on the right-hand side. The sum of all terms of polynomial degree $k$ define a connection $k$-form on $E$. Thus we indee obtain the data of an $A$-superconnection on $E$. 
\end{proof}

\begin{remark}
 If $ E$ is a vector bundle over $M$ concentrated in degree zero, terms of total degree other than one vanish and \eqref{2.35} reduces to
\begin{equation*}
    Q_{\mathcal{E}}=d_A+\Gamma_{\alpha l}^k \xi^\alpha s^l\frac{\partial}{\partial s^k},
\end{equation*}
which is the local expression of the covariant exterior derivative associated to an $A$-connection  on $E$.
\end{remark}

Definition \ref{repuptohom} of a representation up to homotopy of a Lie algebroid $A\overset{\pi}{\longrightarrow}M$ can also be understood in the differential graded geometric picture. 
A homotopy $A$-module structure on $E$ is an $A$-superconnection $D_A\colon\Omega(A,E)\to \Omega(A,E)$ squaring to zero. 
The homological vector field $D_A$ is thus  a $Q$-connection on $\mathcal{E}$, with vanishing curvature. By Lemma \ref{lemmameta}, it makes $\mathcal{E}^*$ a differential graded vector bundle over $A[1]$, which is equivalent to  $\mathcal{E}$ being a differential graded vector bundle over $A[1]$. This equivalence follows from $E^*$ being a representation up to homotopy of $A$ isomorphic to~$E$ \cite{abadRepresentationsHomotopyLie2011}.
 As a consequence of this, representations up to homotopy of $A$ correspond to flat $Q$-connections on $\mathcal{E}$.  
\begin{proposition}\cite[Lemma 4.4]{mehtaLieAlgebroidModules2014}\label{prop:equivalencehomot}
    Homotopy $A$-module structures on a (split, finite dimensional) graded vector bundle $E\to M$ are in one-to-one correspondence with differential graded vector bundle structures on $\mathcal{E}=\pi^*E\to A[1]$.
\end{proposition}

\subsection{Atiyah classes of dg vector bundles and homological perturbation}

Differential graded (dg) Lie algebroids are dg vector bundles whose homological vector field is linear (it preserves the arity of functions) and which are equipped with a compatible Lie algebroid structure:

\begin{definition} \cite[Definition 2.4.26]{mehtaSupergroupoidsDoubleStructures2006}
A dg vector bundle $(\mathcal{L},\mathcal{Q})$ over a dg manifold $(\mathcal{M},Q)$ is a dg Lie algebroid if the two following conditions are satisfied:
\begin{enumerate}
    \item \textbf{linearity:} $\mathcal{Q}$ preserves fiberwise linear functions on $\mathcal{L}$\footnote{This linearity condition implies that $\mathcal{Q}$ is invariant under the (de)suspension operator so that for any $j\in\mathbb{Z}$, it canonically induces a linear homological vector field $\mathcal{Q}[j]$ (itself) on $\mathcal{M}[j]$.}:
    \begin{equation*}
        \mathcal{Q}(\mathcal{L}^*)\subset \mathcal{L}^*,
    \end{equation*}
    \item \textbf{compatibility:} on $\mathcal{L}[1]$, $\mathcal{Q}[1]$ and the Chevalley-Eilenberg differential $d_\mathcal{L}$ commute:
    \begin{equation*}
        [d_{\mathcal{L}},\mathcal{Q}[1]]=0.
    \end{equation*}
\end{enumerate}

\end{definition}

Let $(\mathcal{L}, \mathcal{Q})$ be a dg Lie algebroid, defined over a dg manifold $(\mathcal{M},Q)$, and let $(\mathcal{E},Q_{\mathcal{E}})$ be a dg vector bundle over $\mathcal{M}$. We denote the sections of the graded vector bundle $\mathcal{E}$ over~$\mathcal{M}$ as $\Gamma_{\mathcal{M}}(\mathcal{E})$. A $\mathcal{L}$-connection on $\mathcal{E}$ is a degree 0 map $\nabla:\Gamma_{\mathcal{M}}(\mathcal{L})\times\Gamma_{\mathcal{M}}(\mathcal{E})\to \Gamma_{\mathcal{M}}(\mathcal{L})$ satisfying the usual identities:
\begin{align*}
    \nabla_{fX}(e)&=f\nabla_X(e),\\
    \nabla_X(fe)&=\rho(X)(f)\,e+ (-1)^{|f||X|}f\nabla_X(e),
\end{align*}
for every $f\in\mathcal{C}^\infty(\mathcal{M})$,  $X\in\Gamma_{\mathcal{M}}(\mathcal{L})$ and $e\in\Gamma_{\mathcal{M}}(\mathcal{E})$. The following Definition and Proposition were first introduced in \cite{mehtaAtiyahClassDgvector2015}.

\begin{definition}
    Given a $\mathcal{L}$-connection on $\mathcal{E}$, the \emph{Atiyah cocycle} associated to $\nabla$ is the bundle map $\mathrm{At}_{\mathcal{E}}\in\mathcal{L}^*\otimes\mathrm{End}(\mathcal{E})$ defined by:
\begin{equation*}
    \mathrm{At}_{\mathcal{E}}(X,e)=\mathcal{Q}_{\mathcal{E}}(\nabla_{X}(e))-(-1)^{|X|}\nabla_X(\mathcal{Q}_{\mathcal{E}}(e))- \nabla_{\mathcal{Q}(X)}(e).
\end{equation*}
\end{definition}

\begin{proposition}
    Let $\mathfrak{Q}=\mathcal{Q}+[\mathcal{Q}_{\mathcal{E}},.\,]$ be the induced homological vector field on the graded manifold $\mathcal{L}^*\otimes\mathrm{End}(\mathcal{E})$. Then
    \begin{enumerate}
        \item $\mathrm{At}_{\mathcal{E}}$ is a $\mathfrak{Q}$-cocycle of degree $1$;
        \item the cohomology class of $\mathrm{At}_{\mathcal{E}}$ does not depend on the choice of connection $\nabla$.
    \end{enumerate}
\end{proposition}

Let $A\hookrightarrow L$ be a Lie pair, and denote by $\mathcal{L}$ the pull-back Lie algebroid through $\pi:A[1]\to M$ of the Lie algebroid $L$ \cite{mackenzieGeneralTheoryLie2005}, namely $\mathcal{L}=\pi^!L$ is a graded vector bundle over $A[1]$. This graded manifold, originally introduced by Sti\'enon, Vitagliano and Xu \cite{stienonAinfinityAlgebrasLie2022}, turns out to be a dg Lie algebroid, whose homological vector field we denote by $\mathcal{Q}$.
Using these data,  the authors in \cite{stienonAinfinityAlgebrasLie2022} establish a bijection between $\Omega^\bullet(A,A^{\circ}\otimes\mathrm{End}(L/A))$ and $\Gamma_{A[1]}(\mathcal{L}^*\otimes\mathrm{End}(\mathcal{L}))$, inducing a map sending the Atiyah class of the Bott representation to the Atiyah class of the dg vector bundle $\mathcal{L}\to A[1]$.

In \cite{liaoAtiyahClassesTodd2023}, the author simplifies the proof of the former by summoning a more straightforward homological perturbation theoretic argument relying on a splitting of the following diagram
\begin{center}
\begin{tikzcd}[column sep=2cm,row sep=1.5cm]
  0\ar[r]&A[1] \ar[r,  "\iota\circ sp" above]\ar[d,  "p" left, shift right]&  L \ar[r]\ar[d,  "p" left, shift right]& 0\\
  0\ar[r, dashed,  "0"]&0\ar[r, dashed,  "0"]&L/A\ar[r, dashed,  "0"]&0
\end{tikzcd} 
\end{center}
In the present section, we will generalize the result in \cite{liaoAtiyahClassesTodd2023}, valid for the normal complex and the Bott representation, to a wide class of  Lie algebroid representations of the Lie algebroid~$A$. Namely,
those representations $K$ who are resolved by regular homotopy $A$-modules $(E,\partial)$, introduced in Section \ref{subway}: 

\begin{center}
\begin{equation}\label{basicdiag}
\begin{tikzcd}[column sep=2cm,row sep=1.5cm]
  \ldots\ar[r,  "\partial"]&E_{-2}\ar[d,"\phi" left]\ar[r,  "\partial"]&E_{-1}\ar[d,"\phi" left] \ar[r,  "\partial"]& E_{0}\ar[d,"\phi" left] \ar[r,  "0"]& 0 \\
 \ldots\ar[r, dashed,  "0"]&0\ar[r, dashed,  "0"]&0\ar[r, dashed,  "0"]& K\ar[r, dashed,  "0"]& 0 
\end{tikzcd} 
\end{equation}
\end{center}

Recall that a \emph{contraction} -- or \emph{strong deformation retract} -- consists of two chain complexes $(V,d_V)$ and $(W,d_W)$, together with two chain maps $\phi:V\to W$ and $\sigma:W\to V$, and a degree $-1$ map of vector spaces $\theta: V\to V$, such that:
\begin{enumerate}
    \item $W$ is a \emph{retract} of $V$, so $\phi\circ \sigma=\mathrm{id}_W$,
    \item $\theta$ is a \emph{chain homotopy} between $\mathrm{id}_V$ and $\sigma \circ \phi$, namely:
    \begin{equation*}
        \mathrm{id}_V-\sigma\circ \phi=[d_V,\theta]\,,
    \end{equation*}
    \item the \emph{side conditions} hold:
    \begin{equation*}
        \theta^2=0\,,\qquad\theta\circ \sigma=0\,\qquad \phi\circ \theta=0.
    \end{equation*}
\end{enumerate}
We draw these data under the compact form
\begin{equation*}
\begin{tikzcd}
\big(W,d_W\big)
\arrow[r, "\sigma", shift left,hook] &  \big(V, d_V\big)
\arrow[l, "\phi", shift left,two heads] \arrow[loop, "\theta",out=12,in= -12,looseness = 3]
\end{tikzcd}
\end{equation*}
and use it as a shorthand notation for contraction data.

In the following, we denote $\nabla^K$ the flat Lie algebroid $A$-connection on $K$ inducing the $A$-module structure on $K$ (with associated exterior covariant derivative $d_A^{\nabla^K}$), and $D_A$ the flat $A$-superconnection on $E$ inducing the homotopy $A$-structure,  so that the compatibility condition \eqref{eqproj} holds. We denote $d^K$ the exterior covariant derivative associated to the $A$-module structure on $A^\circ\otimes\mathrm{End}(K)$. 
We denote $\mathcal{E}=\pi^*E$ the pull-back bundle of $E$ along the projection $\pi:A[1]\to M$. 
We will show that we have the contraction data
\[\begin{tikzcd}
  \hspace{1.3cm}    \left(\Omega\big(A,A^\circ\otimes\mathrm{End}(K)\big),d^K\right) \arrow[r, shift left=0.75ex, hook] & \arrow[l, shift left=0.75ex, ->>] \left(\Gamma_{A[1]}\big(\mathcal{L}^*\otimes\mathrm{End}(\mathcal{E})\big),\mathcal{Q}+[D_A,.\,]\right)\hspace{3.5cm} \ar[loop,out=20,in=-20,distance=30]
\end{tikzcd}\]
which is such that, whenever $K=L/A$ and $E=A[1]\oplus L$ -- so $\mathcal{E}\simeq\mathcal{L}$ -- we recover the contraction data presented in \cite{liaoAtiyahClassesTodd2023}.

 Recall that the chain complex $(E,\partial)$ being a finite length resolution of the vector bundle $K$, there exists a section $\sigma:K_\bullet\to E_\bullet$ and a map $\theta:E_\bullet\to E_{\bullet-1}$, such that $\phi:(E_\bullet,\partial)\to (K_\bullet,0)$ is a homotopy equivalence with homotopy inverse $\sigma$ and chain homotopy~$\theta$ (see Proposition \ref{prop:quasiso}).
 Denote $\mathcal{K}=\pi^*K$ the pull-back bundle of $K$  through the projection $\pi:A[1]\to M$. The morphisms $\partial$, $\theta$, $\phi$ and $\sigma$ can be pulled back too, inducing graded vector bundle morphisms (over $A[1]$) satisfying similar properties as the former with respect to $\mathcal{K}$ and $\mathcal{E}$. We also denote them $\partial$, $\theta$, $\phi$ and $\sigma$, except that they now become $\mathcal{C}^\infty(A[1])\simeq\Omega(A)$-linear. It is straightforward to see that we have the following contraction data
\[\begin{tikzcd}
      \left(\Gamma_{A[1]}(\mathcal{K}),0\right) \arrow[r, shift left=0.75ex, "\sigma", hook] & \arrow[l, shift left=0.75ex, "\phi", ->>] \left(\Gamma_{A[1]}(\mathcal{E}),\partial\right) \ar[loop,out=10,in=-10,distance=20, "\theta"]
    \end{tikzcd}\]

    Now we only need proving similar results to \cite[Lemma 2.6, Theorem 2.7]{liaoAtiyahClassesTodd2023}, in the present context. Set
    \begin{equation*}
   d=D_A-\partial\qquad \text{and}\qquad F^q(E)=\bigoplus_{i\geq q} \Gamma_{A[1]}(\mathcal{E}_{-i})\quad \text{for all $q$.}
    \end{equation*}
Let also  $\varsigma, \varphi, \vartheta$ be the following perturbed contraction operators: 
\begin{equation*}
\vartheta=\theta_d=\sum_{k=0}^\infty \theta(-d\theta)^k\,,\quad
\varsigma=\sigma_d=\sum_{k=0}^\infty(-\theta d)^k\sigma\,,\quad\text{and}\quad
\varphi=\phi_d=\sum_{k=0}^\infty\phi(-d\theta)^k.
\end{equation*}
  
\begin{lemma}
The operator $d$ is a perturbation of $(\Gamma_{A[1]}(\mathcal{E}),\partial)$ and satisfies the property:
\begin{equation*}
d\circ\theta\left(F^q(E)\right)\subset F^{q+1}(E),\quad \text{for all $q$}.
\end{equation*}
\end{lemma}
\begin{proof}
This is straightforward because $D_A$ contains $\partial$ by definition, see Equation~\eqref{eq:defD}.
\end{proof}

\begin{theorem}\label{theoremLiao}
 The operator $d=D_A-\partial$ is a small perturbation of the contraction  $(\Gamma_{A[1]}(\mathcal{E}), \partial, \theta)$ over $\mathbb{R}$. The perturbed contraction
\[\begin{tikzcd}
      \left(\Gamma_{A[1]}(\mathcal{K}),d_A^{\nabla^K}\right) \arrow[r, shift left=0.75ex, "\varsigma", hook] & \arrow[l, shift left=0.75ex, "\varphi", ->>] \left(\Gamma_{A[1]}(\mathcal{E}),D_A\right) \ar[loop,out=10,in=-10,distance=20, "\vartheta"]
    \end{tikzcd}\]
    forms a contraction data over $(\mathcal{C}^\infty(A[1]),d_A)$.
\end{theorem}

\begin{proof}

The perturbed operator associated to the zero differential on the chain complex $(\Gamma_{A[1]}(\mathcal{K}),\delta=0)$ is given by 
\begin{equation*}
\delta_d=\sum_{k=0}^\infty\phi d(-\theta d)^k\sigma=\phi D_A \sigma.
\end{equation*}
The last equality can be explained because the bottom chain complex spaces in \eqref{basicdiag} vanish except at degree 0. Since $\phi d(-\theta d)^k\sigma=\phi(-d\theta)^k d\sigma$, and $(d\theta)(F^q(E))\subset F^{q+1}(E)$, we deduce that the only non-vanishing term, after projection under $\phi$, can only be $\phi d \sigma =\phi D_A\sigma$.

We only need to show by direct computations that $\vartheta$, $\varsigma$ and $\varphi$ are $\mathcal{C}^\infty(A[1])$-linear, and that $\delta_d$ coincides with $d_A^{\nabla^K}$.
Let us start with $\vartheta$. Since $\theta^2=0$, we have (by induction):
\begin{align*}
    \theta(d\theta)^k(f\cdot \lambda)&=\theta\bigg((-1)^{|f|}d_Af\theta(d\theta)^{k-1}(\lambda)+f(d\theta)^k(\lambda)\bigg)\\
    &=(-1)^{2|f|+1}d_Af\underbrace{\theta^2}_{=\, 0}(d\theta)^{k-1}(\lambda)+(-1)^{|f|}f\theta(d\theta)^k(\lambda)\\
    &=(-1)^{|f|}f\theta(d\theta)^k(\lambda),
\end{align*}
for every $f\in\Omega(A)$ and $\lambda\in\Gamma(\mathcal{E})$. This proves the $\mathcal{C^\infty}(A[1])$-linearity of $\vartheta$.

Let us turn to $\varsigma$. By definition of the homotopy operator $\theta$, we have $\theta\sigma=0$, so for every $f\in\Omega(A)$ and $\kappa\in\Gamma_{A[1]}(\mathcal{K})$, on the one hand:
\begin{align*}
\theta d\sigma(f\cdot\kappa)&=\theta(d(f\cdot\sigma(\kappa)))\\
&=\theta(d_Af\sigma(\kappa))+\theta((-1)^{|f|}f d\sigma(\kappa))\\
&=(-1)^{|f|+1}d_Af\underbrace{\theta\sigma}_{=\,0}(\kappa)+f\theta d\sigma(\kappa)\\
&=f\theta d\sigma(\kappa),
\end{align*}
and on the other hand, since $\theta^2=0$,
\begin{align*}
(\theta d)^2\sigma(f\cdot\kappa)&=\theta d(f\theta d\sigma(\kappa))\\
&=\theta(d_Af \theta d\sigma (\kappa)+(-1)^{|f|}fd\theta d\sigma(\kappa))\\
&=(-1)^{|f|+1}d_Af\underbrace{\theta^2}_{=\,0}d\sigma (\kappa)+f(\theta d)^2\sigma(\kappa)\\
&=f(\theta d)^2\sigma(\kappa).
\end{align*}
More generally, by induction:
\begin{equation*}
    (\theta d)^k\sigma(f\cdot\kappa)=(-1)^{|f|+1}d_Af\underbrace{\theta^2}_{=\,0}d (\theta d)^{k-2}\sigma (\kappa)+f(\theta d)^k\sigma(\kappa)=f(\theta d)^k\sigma(\kappa).
\end{equation*}
This proves the $\mathcal{C^\infty}(A[1])$-linearity of $\varsigma$.

Finally, again by definition of the homotopy operator $\theta$, we have $\phi\theta=0$,  and so we have, for every $f\in\Omega(A)$,
\begin{align*}
    \phi d\theta(f\cdot\lambda)&=(-1)^{|f|}\phi d (f\theta(\lambda))\\
    &=(-1)^{|f|}\phi( d_A(f)\theta(\lambda))+\phi(fd\theta(\lambda))\\
    &=(-1)^{|f|}d_A(f)\underbrace{\phi\theta}_{=\,0} (\lambda)+f\phi d\theta(\lambda)\\
    &=f\phi d\theta(\lambda).
\end{align*}
Since $\theta^2=0$, it is
\begin{align*}
    \phi (d\theta)^2(f\cdot\lambda)&=(-1)^{|f|}\phi d\theta( d_A(f)\theta(\lambda))+\phi d\theta(fd\theta(\lambda))\\
    &=(-1)^{2|f|+1}\phi d(d_Af \underbrace{\theta^2}_{=\,0}(\lambda))+(-1)^{|f|}\phi d(f\theta d\theta(\lambda))\\
    &=(-1)^{|f|}\phi (d_Af\theta d\theta(\lambda))+\phi(f (d\theta)^2(\lambda))\\
    &=(-1)^{|f|}d_Af\underbrace{\phi\theta}_{=\,0} d\theta(\lambda)+\phi(f (d\theta)^2(\lambda))\\
    &=f\phi (d\theta)^2(\lambda),
    \end{align*}
    and more generally, by induction, it is
    \begin{equation*}
        \phi (d\theta)^k(f\cdot\lambda)=(-1)^{|f|}d_Af\underbrace{\phi\theta}_{=\,0} (d\theta)^{k-1}(\lambda)+f\phi (d\theta)^k(\lambda)=f\phi (d\theta)^k(\lambda).
    \end{equation*}
  This proves the $\mathcal{C^\infty}(A[1])$-linearity of $\varphi$.

Denote by $d_A^\nabla$ the exterior covariant derivative associated to the Lie algebroid $A$-connection $\nabla$ on $E$, that is: $d_A^\nabla=D_A^{(1)}$. The compatibility condition~\eqref{eqproj} then reads
\begin{equation}\label{eqlabor}
    \phi \circ d_A^\nabla=d_A^{\nabla^K}\circ  \phi.
\end{equation}
For every $\kappa\in\Gamma_{A[1]}(\mathcal{K})$, $D_A\sigma(\kappa)$ has components in $F^{-1}(E)$, but only the component in $E_0$ is non-trivially projected down to $K$ as the projections of the other components vanish. The only possibly non-vanishing component is then precisely $d_A^\nabla\sigma(\kappa)$. Thus, by Equation~\eqref{eqlabor}, we have
\begin{equation*}
\delta_d(\kappa)=\phi D_A\sigma(\kappa)=\phi d_A^\nabla \sigma(\kappa)=d_A^{\nabla^K}\phi\sigma(\kappa)=d_A^{\nabla^K}\kappa.
\end{equation*}
 Since both $\phi D_A\sigma$ and $d_A^{\nabla^K}$ satisfy the Leibniz rule with respect to $\mathcal{C^\infty}(A[1])=\Omega(A)$, we deduce that $\delta_d=d_A^{\nabla^K}$.\end{proof}

Theorem \ref{theoremLiao} together with \cite[Proposition A.9]{liaoAtiyahClassesTodd2023}, imply:
\begin{corollary}\label{corollarymel} The following forms a contraction data over $(\mathcal{C}^\infty(A[1]),d_A)$:
    \[\begin{tikzcd}
\hspace{0.8cm}\left(\Gamma_{A[1]}\big(\pi^*(A^\circ)\otimes\mathrm{End}(\mathcal{K})\big),d^K\right) \arrow[r, shift left=0.75ex, hook] & \arrow[l, shift left=0.75ex, ->>] \left(\Gamma_{A[1]}\big(\mathcal{L}^*\otimes\mathrm{End}(\mathcal{E})\big),\mathcal{Q}+[D_A,.\,]\right)\hspace{3.5cm} \ar[loop,out=20,in=-20,distance=30]
    \end{tikzcd}\]
    \end{corollary}
    The space on the left-hand side being canonically isomorphic to $\Omega\big(A,A^\circ\otimes\mathrm{End}(K)\big)$, we see that the space of sections of $\pi^*(A^\circ)\otimes\mathrm{End}(\mathcal{K})$ over $A[1]$ is isomorphic to $\Omega^1\big(A,A^\circ\otimes\mathrm{End}(K)\big)$, namely where belongs the Atiyah cocycle of the representation $K$ with respect to the Lie algebroid pair $A\hookrightarrow L$. On the other hand, 
the space of  sections of $\mathcal{L}^*\otimes\mathrm{End}(\mathcal{E})$ over $A[1]$ is where lives the Atiyah cocycle of the dg vector bundle $\mathcal{E}\to A[1]$ with respect to an $\mathcal{L}$-connection (in the sense of \cite{mehtaAtiyahClassDgvector2015}). Hence, the isomorphism induced by the contraction data at the level of cohomology sends the former Atiyah class to the latter, and vice versa. The main results of this section 
are gathered in the following statement.

\begin{theorem}\label{theoremfinal}
Let $A\hookrightarrow L$ be a Lie pair and let $(E,\partial)$ be a regular representation up to homotopy of $A$. Assume $(E,\partial)$ is a resolution of a Lie algebroid representation $K$ of $A$. Then we have the following quasi-isomorphisms
\begin{equation*}
 \left(\widehat{\Omega}(E),s\right)\simeq \left(\Omega\big(A,A^\circ\otimes \mathrm{End}(K)\big),d^K\right)\simeq \left(\Gamma_{A[1]}\big(\mathcal{L}^*\otimes\mathrm{End}(\mathcal{E})\big),\mathcal{Q}+[D_A,.\,]\right).
\end{equation*}
\end{theorem}

\begin{proof}
The first isomorphism is Theorem \ref{prop:BRST} while the second is Corollary \ref{corollarymel}.
\end{proof}
We should then be able to recover the result about the Atiyah class of the normal complex presented in \cite{liaoAtiyahClassesTodd2023} as an application of the statement above to the setting of Section~\ref{subnormal}. 
First, choosing a splitting of the short exact sequence 
\begin{equation}\label{shortexact}
\begin{tikzcd}[column sep=1.5cm,row sep=1.5cm]
  0\ar[r]&A \ar[r,  "\iota"]& \ar[l,  "\tau", bend left] L \ar[r,  "p"]& \ar[l,  "\sigma", bend left]L/A\ar[r]&0,
\end{tikzcd} 
\end{equation}
we have the contraction data
\[\begin{tikzcd}
      \left(\Gamma_{A[1]}\big(\pi^*(L/A)\big),0\right) \arrow[r, shift left=0.75ex, "\sigma", hook] & \arrow[l, shift left=0.75ex, "p", ->>] \left(\Gamma_{A[1]}(\mathcal{E}),\widetilde{\iota}\right) \ar[loop,out=10,in=-10,distance=20, "\widetilde{\tau}"]
    \end{tikzcd}\]
where $\widetilde{\iota}:A[1]\to A\to L$ is the suspension of the inclusion  $\iota$ and  $\widetilde{\tau}:L\to A\to A[1]$ is the desuspension of the projection $\tau$. 
Theorem \ref{theoremLiao} then establishes the contraction data
\[\begin{tikzcd}
      \left(\Gamma_{A[1]}\big(\pi^*(L/A)\big),d_A^{Bott}\right) \arrow[r, shift left=0.75ex, "\varsigma", hook] & \arrow[l, shift left=0.75ex, "p", ->>] \left(\Gamma_{A[1]}(\mathcal{E}),D_A\right)  \ar[loop,out=10,in=-10,distance=20, "\widetilde{\tau}"]
    \end{tikzcd}\]
    where $\varsigma=\sigma-\widetilde{\tau}D_A\sigma$.
From this, 
the right-hand side of Theorem \ref{theoremfinal} reads
\begin{equation*}
\left(\Omega\big(A,A^\circ\otimes \mathrm{End}(L/A)\big),d^{L/A}\right)
\simeq \left(\Gamma_{A[1]}\big(\mathcal{L}^*\otimes\mathrm{End}(\mathcal{E})\big),\mathcal{Q}+[D_A,.\,]\right),
\end{equation*}
where  $d^{L/A}$ is understood as the exterior covariant derivative associated to the representation $A^\circ\otimes\mathrm{End}(L/A)$. This, together with the following Lemma,  establishes the result on the Atiyah class of the normal complex from \cite{liaoAtiyahClassesTodd2023} as a particular case of Theorem~\ref{theoremfinal}.

\begin{lemma}\label{prop:Liao} Let $A\hookrightarrow L$ be a Lie pair and 
    let $E=A[1]\oplus L$ be the normal complex. The dg vector bundle $(\mathcal{L},\mathcal{Q})$ is dg isomorphic to the dg vector bundle $(\mathcal{E},D_A)$.
\end{lemma} 

\begin{proof} For the sake of completeness, let us use the notations of \cite[Remark 2.5]{liaoAtiyahClassesTodd2023}. Choose a trivialization of $L$ and a local frame $e_1,\ldots, e_r$ of $A$, extended to a local frame $e_1,\ldots, e_r, e_{r+1},\ldots, e_{r+r'}$ of $L$, that in turn induces a local dual frame of $L^*$ denoted $e^1,\ldots, e^{r+r'}$.
Choosing a splitting of the sequence \eqref{shortexact}, let $\frac{\partial}{\partial \eta^1},\ldots, \frac{\partial}{\partial \eta^r}$ be the local frame of $A[1]$ obtained from $e_1,\ldots, e_r$ through the homotopy map $\widetilde{\tau}$, namely $\frac{\partial}{\partial \eta^i}=\widetilde{\tau}(e_i)$ and $\iota\left(\frac{\partial}{\partial \eta^i}\right)=e_i$. This choice of notation comes from the identification $TA[1]^{\text{vertical}}\simeq A[1]$, so that $\eta^1,\ldots,\eta^r$ is a local frame of $A[1]^*$ such that $\frac{\partial}{\partial \eta^j}(\eta^i)=\delta^i_j$.

     Pick up a $TM$-connection  on $A$ -- say $\nabla$, giving an $L$-connection on $A$ which in turn induces a basic connection $\nabla^{bas}$ (see Definition \ref{basconnnormal} and the preceding paragraph).
    This connection allows to write the $A$-superconnection on $E$ with respect to the fiber coordinates under a form that explicitly involves the basic connection. More precisely,
    the connection 1-form $\omega^{(1)}$ associated to $\nabla$ is expressed in local coordinates as     
    \begin{equation*}\omega^{(1)}=\sum_{\mu=1}^n \sum_{k,l=1}^{r}\omega_{\mu k}^l dx^\mu e_l\otimes e^k+\sum_{k,l=1}^{r}\omega_{\mu k}^l\, dx^\mu \frac{\partial}{\partial \eta^l}\otimes \eta^k.
    \end{equation*}
The dual connection $\nabla^*$, acting on $E^*$, is then associated to the following dual connection 1-form:
\begin{equation*}
\omega^{(1)*}=-\sum_{\mu=1}^n \sum_{k,l=1}^{r}\omega_{\mu k}^l dx^\mu  e^k\otimes e_l+\sum_{k,l=1}^{r}\omega_{\mu k}^l\, dx^\mu \eta^k\otimes \frac{\partial}{\partial \eta^l}.
\end{equation*}
    
With respect to the chosen trivialization of $L$, the basic connection $\nabla^{bas}$ induced by $\nabla$ induces a connection 1-form $\omega^{bas}$ whose expression in local coordinates is $\omega_{ij}^{bas,k}=\sum_{\mu=1}^n\rho_j^\mu\omega_{\mu i}^k+c_{ij}^k$, where the indices range from 1 to $r$ for $i$, and from 1 to $r+r'$ for $j,k$. Then, the Lie algebroid derivative $d_A$ can be recast under the  form
    \begin{equation*}
d_A = \sum_{i=1}^r  \eta^i \nabla^*_{\rho(e_i)} - \frac{1}{2}\sum_{i,j,k =1}^r T_{\nabla^{bas}}(e_i,e_j)^k \eta^i \eta^j  \frac{\partial}{\partial \eta^k}.
\end{equation*}
In particular see Equation (4.5) in \cite{chatzistavrakidisBasicCurvatureAtiyah2023}, reminding that $\nabla^{bas}$ is the opposite connection of the induced connection $\overset{\bullet}{\nabla}$.

     With respect to the choice of connection $\nabla$ in the given local trivialization, the other components of $D_A$ are given by:
\begin{align*}
D_A\Big(\frac{\partial}{\partial \eta^m}\Big) & = \sum_{i,k =1}^r \omega_{im}^{bas,k} \eta^i \frac{\partial}{\partial \eta^k} + e_m, \\
D_A \left(e_l\right) &  =  \frac{1}{2}  \sum_{i,j,k=1}^r R^{bas}(e_i,e_j)(e_l)^k \eta^i \eta^j \frac{\partial}{\partial \eta^k}  +  \sum_{k=1}^{r+r'} \sum_{i=1}^r  \omega_{il}^{bas,k} \eta^i\otimes e_k.
\end{align*}
for every $1\leq m\leq r$ and $1\leq l\leq r+r'$. Both formulas are obtained by hand computations, and one can compare them to Equations (15)-(16) of \cite{liaoAtiyahClassesTodd2023} to observe that $c_{ij}^k$ (resp. $\rho_l(c_{ij}^k)$) in the latter becomes $\omega_{ij}^{bas,k}$ (resp. $(R^{bas}_{ij})_l^k$) in the former.

In other words, the local expression of the homological vector field as given in \cite[Remark 2.5]{liaoAtiyahClassesTodd2023} corresponds to taking the $A$-superconnection on $E$ induced by the standard flat $TM$-connection on $A$ with respect to the chosen trivialisation of $L$, i.e. the unique connection whose connection 1-form vanishes when evaluated on the local frame.
     Thus, the expressions of $\mathcal{Q}$ in\cite[Remark 2.5]{liaoAtiyahClassesTodd2023} and that of $D_A$ here only differ by a choice of $TM$-connection on $A$, which, regarding the theory of representations up to homotopy, are essentially indistinguishable \cite{abadRepresentationsHomotopyLie2011}.
\end{proof}

\bibliography{Maths}

\begin{thebibliography}{}

\bibitem[Abad and Crainic, 2011]{abadRepresentationsHomotopyLie2011}
Abad, C.~A. and Crainic, M. (2011).
\newblock Representations up to homotopy of {{Lie}} algebroids.
\newblock {\em J. Reine Angew. Math.}, 2012(663):91--126.

\bibitem[Aschieri et~al., 2021]{aschieriCurvatureTorsionCourant2021}
Aschieri, P., Bonechi, F., and Deser, A. (2021).
\newblock On {{Curvature}} and {{Torsion}} in {{Courant Algebroids}}.
\newblock {\em Ann. Henri Poincar{\'e}}, 22(7):2475--2496.

\bibitem[Batakidis and Voglaire, 2018]{batakidisAtiyahClassesDgLie2018}
Batakidis, P. and Voglaire, Y. (2018).
\newblock Atiyah classes and dg-{{Lie}} algebroids for matched pairs.
\newblock {\em J. Geom. Phys.}, 123:156--172.

\bibitem[Blaom, 2006]{blaomGeometricStructuresDeformed2006}
Blaom, A. (2006).
\newblock Geometric structures as deformed infinitesimal symmetries.
\newblock {\em Trans. Amer. Math. Soc.}, 358(8):3651--3671.

\bibitem[Chatzistavrakidis and Jonke,
  2023]{chatzistavrakidisBasicCurvatureAtiyah2023}
Chatzistavrakidis, A. and Jonke, L. (2023).
\newblock Basic curvature and the {{Atiyah}} cocycle in gauge theory.

\bibitem[Chen et~al., 2019]{chenAtiyahClassesStrongly2019}
Chen, Z., Lang, H., and Xiang, M. (2019).
\newblock Atiyah classes of strongly homotopy {{Lie}} pairs.
\newblock {\em Algebra Colloq.}, 26(02):195--230.

\bibitem[Chen et~al., 2016]{chenAtiyahClassesHomotopy2016}
Chen, Z., Sti{\'e}non, M., and Xu, P. (2016).
\newblock From {{Atiyah}} classes to homotopy {{Leibniz}} algebras.
\newblock {\em Commun. Math. Phys.}, 341(1):309--349.

\bibitem[Crainic and Fernandes,
  2005]{crainicSecondaryCharacteristicClasses2005}
Crainic, M. and Fernandes, R.~L. (2005).
\newblock Secondary {{Characteristic Classes}} of {{Lie Algebroids}}.
\newblock In {\em Quantum {{Field Theory}} and {{Noncommutative Geometry}}},
  pages 157--176. Springer, Berlin, Heidelberg.

\bibitem[Fairon, 2017]{faironIntroductionGradedGeometry2017}
Fairon, M. (2017).
\newblock Introduction to graded geometry.
\newblock {\em Eur. J. Math.}, 3(2):208--222.

\bibitem[Kotov and Salnikov, 2024]{kotovCategoryGradedManifolds2024}
Kotov, A. and Salnikov, V. (2024).
\newblock The category of z-graded manifolds: What happens if you do not stay
  positive.
\newblock {\em Differ. Geom. Appl.}, 93:102109.

\bibitem[{Laurent-Gengoux} et~al.,
  2020]{laurent-gengouxUniversalLieInfinityalgebroid2020}
{Laurent-Gengoux}, C., Lavau, S., and Strobl, T. (2020).
\newblock The universal {{Lie}} infinity-algebroid of a singular foliation.
\newblock {\em Doc. Math.}, 25:1571--1652.

\bibitem[Liao, 2023]{liaoAtiyahClassesTodd2023}
Liao, H.-Y. (2023).
\newblock Atiyah {{classes}} and {{Todd classes}} of {{pullback}} dg {{Lie
  algebroids associated}} with {{Lie pairs}}.
\newblock {\em Commun. Math. Phys.}, 404(2):701--734.

\bibitem[Mackenzie, 2005]{mackenzieGeneralTheoryLie2005}
Mackenzie, K. C.~H. (2005).
\newblock {\em General {{Theory}} of {{Lie Groupoids}} and {{Lie Algebroids}}}.
\newblock London {{Mathematical Society Lecture Note Series}}. Cambridge
  University Press, Cambridge.

\bibitem[Mehta, 2006]{mehtaSupergroupoidsDoubleStructures2006}
Mehta, R.~A. (2006).
\newblock {\em Supergroupoids, Double Structures, and Equivariant Cohomology}.
\newblock PhD thesis, University of California, Berkeley.

\bibitem[Mehta, 2009]{mehtaQalgebroidsTheirCohomology2009}
Mehta, R.~A. (2009).
\newblock Q-algebroids and their cohomology.
\newblock {\em J. Symplectic Geom.}, 7(3):263--293.

\bibitem[Mehta, 2014]{mehtaLieAlgebroidModules2014}
Mehta, R.~A. (2014).
\newblock Lie algebroid modules and representations up to homotopy.
\newblock {\em Indag. Math.}, 25(5):1122--1134.

\bibitem[Mehta et~al., 2015]{mehtaAtiyahClassDgvector2015}
Mehta, R.~A., Sti{\'e}non, M., and Xu, P. (2015).
\newblock The {{Atiyah}} class of a dg-vector bundle.
\newblock {\em C. R. Math.}, 353(4):357--362.

\bibitem[Quillen, 1985]{quillenSuperconnectionsChernCharacter1985}
Quillen, D. (1985).
\newblock Superconnections and the {{Chern}} character.
\newblock {\em Topology}, 24(1):89--95.

\bibitem[Sti{\'e}non et~al., 2022]{stienonAinfinityAlgebrasLie2022}
Sti{\'e}non, M., Vitagliano, L., and Xu, P. (2022).
\newblock A-infinity algebras from {{Lie}} pairs.

\bibitem[Sti{\'e}non and Xu, 2020]{stienonFedosovDgManifolds2020}
Sti{\'e}non, M. and Xu, P. (2020).
\newblock Fedosov dg manifolds associated with {{Lie}} pairs.
\newblock {\em Math. Ann.}, 378(1):729--762.

\bibitem[Weibel, 1994]{weibelIntroductionHomologicalAlgebra1994}
Weibel, C.~A. (1994).
\newblock {\em An {{Introduction}} to {{Homological Algebra}}}.
\newblock Cambridge {{Studies}} in {{Advanced Mathematics}}. Cambridge
  University Press, Cambridge.

\bibitem[Zambon, 2022]{zambonSingularSubalgebroids2022}
Zambon, M. (2022).
\newblock Singular subalgebroids.
\newblock {\em Ann. Inst. Fourier}, 72(6):2109--2190.

\end{thebibliography}

\end{document}